\documentclass[10pt]{article}
\usepackage{color}
\usepackage{amssymb}
\usepackage{amsthm,array,amsfonts}
\usepackage{amsmath}
\usepackage{verbatim}
\usepackage{tikz-cd}
\usepackage{enumitem}
\setitemize{itemsep=0pt}
\setenumerate{itemsep=0pt}

\usepackage{hyperref}
\usepackage[justification=centering]{caption}

\usepackage{stmaryrd}
\hypersetup{
    colorlinks,
    citecolor=blue,
    filecolor=black,
    linkcolor=black,
}

\usepackage{accents}
\newcommand{\dbtilde}[1]{\accentset{\approx}{#1}}

\newtheorem{theo}{Theorem}[section]
\newtheorem{thm}{Theorem}[section]
\newtheorem{prop}[theo]{Proposition}

\newtheorem{lemma}[theo]{Lemma}

\newtheorem{cor}[theo]{Corollary}
\newtheorem{question}[theo]{Question}

\newtheorem{defn}[theo]{Definition}

\newtheorem{conjecture}{Conjecture}  
\newtheorem{conj}[conjecture]{Conjecture}

\theoremstyle{remark}

\newtheorem{rmk}[theo]{Remark}
\newtheorem{example}[theo]{Example}

\newcommand{\BA}{{\mathbb{A}}}

\newcommand{\BC}{{\mathbb{C}}}

\newcommand{\BG}{{\mathbb{G}}}
\newcommand{\BH}{{\mathbb{H}}}

\newcommand{\BL}{{\mathbb{L}}}

\newcommand{\BN}{{\mathbb{N}}}

\newcommand{\BQ}{{\mathbb{Q}}}
\newcommand{\BR}{{\mathbb{R}}}

\newcommand{\BZ}{{\mathbb{Z}}}

\newcommand{\CA}{{\mathcal A}}
\newcommand{\CB}{{\mathcal B}}
\newcommand{\CC}{{\mathcal C}}
\newcommand{\CD}{{\mathcal D}}

\newcommand{\CL}{{\mathcal L}}
\newcommand{\CM}{{\mathcal M}}
\newcommand{\CN}{{\mathcal N}}
\newcommand{\CO}{{\mathcal O}}
\newcommand{\CP}{{\mathcal P}}

\newcommand{\CT}{{\mathcal T}}
\newcommand{\CU}{{\mathcal U}}

\newcommand{\CX}{{\mathcal X}}

\newcommand{\Fg}{{\mathfrak{g}}}

\newcommand{\Fq}{{\mathfrak{q}}}

\newcommand{\fq}{{\mathfrak{q}}}

\newcommand{\ch}{\mathsf{ch}}

\newcommand{\id}{\mathrm{id}}

\newcommand{\pt}{{\mathsf{p}}}

\newcommand{\blangle}{\big\langle}
\newcommand{\brangle}{\big\rangle}

\newcommand{\Mbar}{{\overline M}}
\newcommand\ev{\operatorname{ev}}
\newcommand{\Pic}{\mathop{\rm Pic}\nolimits}

\newcommand{\DT}{\mathsf{DT}}
\newcommand{\GW}{\mathsf{GW}}

\DeclareFontFamily{OT1}{rsfs}{}
\DeclareFontShape{OT1}{rsfs}{n}{it}{<-> rsfs10}{}
\DeclareMathAlphabet{\curly}{OT1}{rsfs}{n}{it}

\newcommand{\p}{\mathbb{P}}
\newcommand\Id{\operatorname{Id}}
\newcommand\Spec{\operatorname{Spec}}

\newcommand\Hilb{\operatorname{Hilb}}

\newcommand{\A}{\mathsf{A}}
\newcommand{\vd}{\mathsf{vd}}
\newcommand{\vir}{\mathsf{vir}}

\newcommand{\std}{\mathsf{std}}

\newcommand{\rel}{\mathrm{rel}}

\DeclareFontFamily{OT1}{rsfs}{}
\DeclareFontShape{OT1}{rsfs}{n}{it}{<-> rsfs10}{}
\DeclareMathAlphabet{\curly}{OT1}{rsfs}{n}{it}

\newcommand\End{\operatorname{End}}

\newcommand{\1}{v_{\varnothing}}

\newcommand{\Aut}{\mathrm{Aut}}
\newcommand{\DR}{\mathsf{DR}}
\newcommand{\QMod}{\mathsf{QMod}}

\newcommand{\Mod}{\mathsf{Mod}}

\newcommand{\wt}{{\mathsf{wt}}}

\DeclareMathOperator{\Wt}{\mathsf{WT}}
\DeclareMathOperator{\mon}{\mathrm{mon}}

\newcommand{\taut}{{\mathrm{taut}}}

\def\act{{\textup{act}}}

\newcommand{\Mon}{\mathsf{Mon}}

\newcommand{\SO}{\mathrm{SO}}
\newcommand{\so}{\mathfrak{so}}

\newcommand{\SL}{\mathrm{SL}}

\newcommand{\AHJ}{\mathsf{AHJac}}
\newcommand{\QJac}{\mathsf{QJac}}
\newcommand{\QJ}{\mathsf{QJac}}
\newcommand{\MQJac}{\mathsf{MQJac}}
\newcommand{\Jac}{\mathsf{Jac}}

\newcommand{\Z}{\mathsf{Z}}

\newcommand{\vacuum}{\1}

\title{Holomorphic anomaly equations for the Hilbert scheme of points of a K3 surface}
\author{Georg Oberdieck}
\date{}

\begin{document}
\maketitle


\setcounter{section}{0}

\begin{abstract}
We conjecture that the generating series of Gromov-Witten invariants of the Hilbert schemes of $n$ points on a K3 surface are
quasi-Jacobi forms and satisfy a holomorphic anomaly equation.
We prove the conjecture in genus $0$ and for at most $3$ markings -- for all Hilbert schemes and for arbitrary curve classes.
In particular, for fixed $n$,
the reduced quantum cohomologies of all hyperk\"ahler varieties of $K3^{[n]}$-type are determined up to finitely many coefficients.

As an application we show that the generating series of $2$-point Gromov-Witten classes are vector-valued Jacobi forms of weight $-10$,
and that the fiberwise Donaldson-Thomas partition functions of an order two CHL Calabi-Yau threefold are Jacobi forms.

\end{abstract}
\vspace{20pt}
\setcounter{tocdepth}{1}
\tableofcontents

\newpage
\section{Introduction}

\subsection{Overview}
An irreducible hyperk\"ahler variety is a simply-connected smooth projective
variety $X$ such that $H^0(X,\Omega_X^2)$ is generated by a holomorphic-symplectic form \cite{Beauville}.
A topological classification of these varieties is unknown so far.
However, among the 4 families of examples which are known,
the most studied case are the Hilbert schemes of points on a K3 surface and their deformations,
which are called of $K3^{[n]}$-type.
In this article we study the Gromov-Witten theory (the intersection theory of the moduli space
of stable maps) with target a hyperk\"ahler variety of $K3^{[n]}$-type.
We state two new fundamental conjectures: finite generation by quasi-Jacobi forms and holomorphic anomaly equations.
We prove this conjecture for the case of most interest: in genus $0$ and for up to three markings, with no restriction on the curve class.
As a corollary, the reduced quantum cohomology of a $K3^{[n]}$-hyperk\"ahler variety is determined up to finitely many coefficients.
We also find that the series of $2$-point Gromov-Witten classes define vector-valued Jacobi forms of weight $-10$.
This implies that the Donaldson-Thomas partition functions of CHL Calabi-Yau threefolds are Jacobi forms
proving conjectures of \cite{BO-CHL}.

Together with the multiple cover conjecture \cite{ObMC, QuasiK3} we obtain a complete conjectural
picture of the Gromov-Witten theory of $K3^{[n]}$-hyperk\"ahler varieties, which is proven for genus zero and up to three markings.

\subsection{Gromov-Witten theory}
Let $S^{[n]}$ be the Hilbert scheme of $n$ points on a smooth projective K3 surface $S$.
Let 
\[ \Mbar_{g,N}(S^{[n]},\beta+rA) \]
be the moduli space 
of $N$-marked genus $g$ stable maps to $S^{[n]}$ of
non-zero
degree 
\[ \beta + rA \in H_2(S^{[n]},\BZ) \cong H_2(S,\BZ) \oplus \BZ A, \]
where 
$A$ is the exceptional curve class.
Because $S^{[n]}$ 
is irreducible hyperk\"ahler, 
the virtual fundamental class of the moduli space of stable maps in the sense of \cite{LT,BF} vanishes.
Instead Gromov-Witten theory is defined by the reduced virtual class \cite{MP_GWNL,BL, KT1, KL}:
\[
\left[ \Mbar_{g,N}(S^{[n]},\beta+rA) \right]^{\vir} \in A_{\vd}(\Mbar_{g,N}(S^{[n]},\beta+rA)),
\quad \vd=2n(1-g)+N+1.
\]
The first values of the virtual dimension $\vd$ are listed in Table~\ref{table:virtual dimension} below.

If $2g-2+N>0$, 
let 
$\tau : \Mbar_{g,N}(S^{[n]},\beta+rA) \to \Mbar_{g,N}$
be the forgetful morphism
to the moduli space of stable curves. Consider the pullback of a {\em tautological class} \cite{FP}
\[ \taut := \tau^{\ast}(\alpha), \quad \alpha \in R^{\ast}(\Mbar_{g,N}). \] 
In the unstable cases $2g-2+N \leq 0$ we always set $\taut := 1$.
Given cohomology classes $\gamma_i \in H^{\ast}(S^{[n]})$ the reduced Gromov-Witten invariants of $S^{[n]}$ are defined by
\[
\blangle \taut ; \gamma_1, \ldots, \gamma_N \brangle^{S^{[n]}}_{g, \beta+rA}
=
\int_{[ \Mbar_{g,n}(S^{[n]}, \beta+rA) ]^{\vir}}
\taut \cup
\prod_{i=1}^{N} \ev_i^{\ast}(\gamma_i).
\]

\subsection{Generating series}
Consider an elliptic K3 surface $\pi : S \to \p^1$ with a section, and let
\[ B,F \in H_2(S,\BZ) \]
be the class of the section and a fiber of $\pi$ respectively.

By \cite[Cor.2]{ObMC} (based on the global Torelli theorem \cite{Ver,HuyTor}) 
for any hyperk\"ahler variety of $K3^{[n]}$-type $X$
and for any effective curve class $\gamma \in H_2(X,\BZ)$
there exists an $\ell \geq 1$ and a deformation 
\[ (X, \gamma) \rightsquigarrow (S^{[n]}, \ell B + dF + rA), \quad \quad (d \geq 0, r \in \BZ) \]
such that $\gamma$ is kept of Hodge type along the deformation. 
If $\gamma$ is primitive, we can choose $\ell=1$.
By deformation invariance,
it follows that all Gromov-Witten invariants of $K3^{[n]}$-type hyperk\"ahler varieties
are determined by the generating series:
\begin{equation} \label{Fgl} 
F^{S^{[n]}}_{g,\ell}(\taut; \gamma_1, \ldots, \gamma_N) = 
\sum_{d = -\ell}^{\infty} \sum_{r \in \BZ} \left\langle \taut ; \gamma_1, \ldots, \gamma_N \right\rangle^{S^{[n]}}_{g, \ell (B+F)+dF+rA} q^d (-p)^r.
\end{equation}
By convention we assume here that $r=0$ in case $n=1$. We will always assume that $n \geq 1$.

The series \eqref{Fgl} and in particular their modular properties are the main topic of this paper.
To state our main conjectures and results we require 
the Looijenga-Lunts-Verbitsky (LLV) Lie algebra and quasi-Jacobi forms.

\begin{table}
\begin{center}
{\renewcommand{\arraystretch}{1.1}\begin{tabular}{| c | c | c | c | c | c | c | c |}
\hline
   Genus $g$ & $0$ & $1$ & $2$ & $3$ & $4$ & $5$ & $6$ \\
\hline
$S^{[1]}$ & $0$ & $1$ & $2$ & $3$ & $4$ & $5$ & $6$ \\
$S^{[2]}$ & $2$ & $1$ & $0$ & & & & \\
$S^{[3]}$ & $4$ & $1$ &  &  & & & \\
$S^{[4]}$ & $6$ & $1$ &  &  &  & & \\
$S^{[5]}$ & $8$ & $1$ &  &  &  &  & \\
$S^{[6]}$ & $10$ & $1$ & &  &  &  &  \\
\hline
\end{tabular}}
\caption[]{The first non-negative values of the (reduced) virtual dimension of $\Mbar_{g,0}(S^{[n]},\beta+rA)$.
If a field is empty, all Gromov-Witten invariants in this genus vanish.
Hence for $S^{[n]}$ with $n>1$ the most interesting case is genus zero.
\label{table:virtual dimension}}
\end{center}
\end{table}

\subsection{Looijenga-Lunts-Verbitsky Lie algebra}
The LLV algebra \cite{LL, V} of the hyperk\"ahler variety $S^{[n]}$ is the Lie subalgebra
\[ \act : \Fg(S^{[n]}) \hookrightarrow \End H^{\ast}(S^{[n]}) \]
generated by the operators of cup product with classes in $H^2(S^{[n]},\BQ)$ as well as their Lefschetz duals (if they exist),
see Section~\ref{subsec:LLV algebra}.
Concretely, we have an isomorphism
\[ \Fg(S^{[n]}) \cong \wedge^2 (V \oplus U_{\BQ}) \]
where $U_{\BQ}$ is the hyperbolic lattice with basis $e,f$ and intersection form $\binom{0\ 1}{1\ 0}$, and
\[ V = H^2(S^{[n]},\BQ) \cong H^2(S,\BQ) \overset{\perp}{\oplus} \BQ \delta, \quad (\delta,\delta)=2-2n \]
is endowed with the Beauville-Bogomolov-Fujiki quadratic form.

We require the operators
\begin{equation} \label{LLV operators}
\begin{gathered}
U = \act(F \wedge f), \quad T_{\alpha} = \act(\alpha \wedge F), \quad \alpha \in \{ B, F \}^{\perp} \subset V  \\
\Wt = \act(e \wedge f + B \wedge F)
\end{gathered}
\end{equation}
The {\em weight operator} $\Wt \in \End H^{\ast}(S^{[n]})$ is semisimple and defines a grading:
\[ \Wt(\gamma) = \wt(\gamma) \gamma, \quad \wt(\gamma) \in \{ -n, \ldots, n \}. \]
For a class $\gamma \in H^{2i}(S^{[n]})$, the complex cohomological degree of $\gamma$ is denoted by $\deg(\gamma)=i$.

\subsection{Quasi-Jacobi forms}
Jacobi forms are holomorphic functions $f : \BC \times \BH \to \BC$
which satisfy a transformation law under the Jacobi group $\Gamma \ltimes \BZ^2$, where
$\Gamma \subset \SL_2(\BZ)$ is a congruence subgroup \cite{EZ}.
Quasi-Jacobi forms are constant terms of almost holomorphic Jacobi forms, see Section~\ref{sec:quasi-Jacobi forms}.
The algebra of quasi-Jacobi forms is bi-graded by weight $k$ and index $m$:
\[ \QJac(\Gamma) = \bigoplus_{m \geq 0} \bigoplus_{k \in \BZ} \QJac(\Gamma)_{k,m}. \]
The graded summands $\QJac(\Gamma)_{k,m}$ are finite-dimensional.
We usually identify a quasi-Jacobi form with its Fourier expansion
in the variables
\[ p = e^{2 \pi i x}, \quad q = e^{2 \pi i \tau}, \quad (x,\tau) \in \BC \times \BH. \] 

Recall that the algebra of quasi-modular forms $\QMod(\Gamma)$ is
a free polynomial ring over the subalgebra of its modular forms,
\[ \QMod(\Gamma) = \Mod(\Gamma)[ G_2 ] \]
where we used the second Eisenstein series
\[ G_2(\tau) = -\frac{1}{24} + \sum_{n \geq 1} \sum_{d|n} d q^n. \]
Similarly, for quasi-Jacobi forms we always have an embedding
\[ \QJac(\Gamma) \subset \Jac(\Gamma)[ G_2, \A ] \]
where $\Jac(\Gamma)$ is the algebra of weak Jacobi forms and
$\A$ is the logarithmic derivative
\[ \A(p,q) = p \frac{d}{dp} \log \Theta(p,q) \]
of the classical Jacobi theta function
\[
\Theta(p,q) 
=  (p^{1/2}-p^{-1/2})\prod_{m\geq 1} \frac{(1-pq^m)(1-p^{-1}q^m)}{(1-q^m)^{2}}.
\]

Since the generators $G_2$ and $\A$ are free over $\Jac$ one obtains
{\em anomaly operators}
\[ \frac{d}{d G_2} : \QJac(\Gamma)_{k,m} \to \QJac(\Gamma)_{k-2,m}, \quad \frac{d}{d \A} : \QJac(\Gamma)_{k,m} \to \QJac(\Gamma)_{k-1,m} \]
which control the transformation behavior of any quasi-Jacobi form under the Jacobi group.

\subsection{Main conjectures}
We state three fundamental conjectural properties of the series $F_{g,\ell}$.
The first expresses the series $F_{g,\ell}$ in terms of the series of primitive invariants $F_{g,1}$.
Consider the $\ell$-th formal Hecke operator
of weight $k$
which acts on power series $f = \sum_{d,r} c(d,r) q^d p^r$
by
\[ T_{k,\ell} f =
\sum_{n,r} \left( \sum_{a|(\ell,n,r)} a^{k-1} c\left( \frac{\ell n}{a^2}, \frac{r}{a} \right) \right) q^n p^r. \]

For $i \in \{ 1, \ldots, N \}$ let $\gamma_i \in H^{\ast}(S^{[n]})$ be $(\wt,\deg)$-bihomogeneous classes.

\begin{conj}[Multiple Cover Conjecture, {\cite[Sec.2.6]{ObMC}}] \label{MCConjecture}
For all $\ell>0$ we have
\begin{equation} 
F^{S^{[n]}}_{g,\ell}(\taut; \gamma_1, \ldots, \gamma_N) = 
\ell^{\sum_i (\deg(\gamma_i) - n - \wt(\gamma_i))}
T_{k,\ell} 
F^{S^{[n]}}_{g,1}(\taut; \gamma_1, \ldots, \gamma_N)
\label{ellmc} \end{equation}
where $k = n(2g-2+N) + \sum_i \wt(\gamma_i)$.
\end{conj}

The second conjecture concerns the modular behaviour.
Define the modular discriminant
\[ \Delta(q) = q \prod_{n \geq 1} (1-q^n)^{24}, \]
which is a modular form for $\SL_2(\BZ)$ of weight $12$,
and the congruence subgroup 
\[ \Gamma_{0}(\ell) = \left\{ \binom{a\ b}{c\ d} \in \SL_2(\BZ) \middle| c \equiv 0 \text{ mod } \ell \right\}. \]

\begin{conj}[Quasi-Jacobi form Property] \label{Conj: Quasi Jacobi property} For all $\ell > 0$:
\[ F^{S^{[n]}}_{g,\ell}(\taut; \gamma_1, \ldots, \gamma_N) \in \frac{1}{\Delta(q)^{\ell}} \QJac_{k+12 \ell, \ell (n-1)}(\Gamma_0(\ell)), \]
where $k=n(2g-2+N) + \sum_i \wt(\gamma_i) - 10$.
\end{conj}

The difference in the values of $k$ in Conjectures~\ref{MCConjecture} and~\ref{Conj: Quasi Jacobi property}
was explained in \cite[Sec.7.3]{OPix2}.
It is responsible for the appearence of the congruence subgroup $\Gamma_0(\ell)$,
and also leads to the unusual $4$-th term in the holomorphic anomaly equation for $\frac{d}{d G_2}$ below.

Conjecture~\ref{Conj: Quasi Jacobi property} would determine any $F_{g,\ell}(\cdots)$ up to finitely many coefficients.
However, in order to know their transformation property under the Jacobi group
and also to make them depend on substantially less coefficients,
we will conjecture their dependence on the quasi-Jacobi generators $G_2$ and $\A$:

For $2g-2+N>0$ define the degree $0$ Gromov-Witten invariants
\[
F_{g}^{S^{[n]},\std}(\taut ; \gamma_1, \ldots, \gamma_N) := \int_{[ \Mbar_{g,N}(S^{[n]},0) ]^{\std}} \tau^{\ast}(\taut) \prod_{i=1}^{N} \ev_i^{\ast}(\gamma_i),
\]
where we let $[ \ldots ]^{\std}$ denote the standard (non-reduced!) virtual class
in the sense of \cite{LT,BF}.
Explicit formulas are given in \eqref{degree zero virtual class} below.

\begin{conj}[Holomorphic Anomaly Equation]\label{Conj:HAE}
 Assume Conjecture~\ref{Conj: Quasi Jacobi property}. 
We have
\begin{align*}
\frac{d}{d G_2} F^{S^{[n]}}_{g,\ell}(\taut; \gamma_1, \ldots, \gamma_N)
= & F^{S^{[n]}}_{g-1,\ell}(\taut'; \gamma_1, \ldots, \gamma_N, U) \\
& + 2 \sum_{\substack{g=g_1 + g_2 \\ \{ 1, \ldots, N \} = A \sqcup B }} F^{S^{[n]}}_{g_1,\ell}(\taut_1 ; \gamma_A, U_1) F_{g_2}^{S^{[n]},\std}(\taut_2 ; \gamma_B , U_2 ) \\
& - 2 \sum_{i=1}^{N} F^{S^{[n]}}_{g,\ell}( \psi_i \cdot \taut ; \gamma_1, \ldots, \gamma_{i-1}, U \gamma_i, \gamma_{i+1}, \ldots, \gamma_N) \\
& - \frac{1}{\ell} \sum_{a,b} (g^{-1})_{ab} T_{e_a} T_{e_b} F^{S^{[n]}}_{g,\ell}(\taut; \gamma_1, \ldots, \gamma_N)
\end{align*}
and
\[ \frac{d}{d \A} F^{S^{[n]}}_{g,\ell}(\taut; \gamma_1, \ldots, \gamma_N) = T_{\delta} F^{S^{[n]}}_{g,\ell}(\taut; \gamma_1, \ldots, \gamma_N) \]
where
\begin{itemize}
\item we have identified the operator $U \in \End H^{\ast}(S^{[n]})$ with the class
\[ U \in H^{\ast}(S^{[n]} \otimes S^{[n]}) \]
using Poincar\'e duality and the conventions of Section~\ref{subsec:Conventions},
\item $U_1, U_2$ stands for summing over the K\"unneth decomposition of $U \in H^{\ast}((S^{[n]})^2)$,
\item the $e_a$ form a basis of $\{ F, B \}^{\perp} \subset H^2(S,\BQ)$ and
$g_{ab} = \langle e_a, e_b \rangle$ is the pairing matrix,
\item for any $\alpha \in \{ B, F \}^{\perp} \subset V  $ we set
\begin{equation} \label{T alpha action}
T_{\alpha} F^{S^{[n]}}_{g, \ell}(\taut; \gamma_1, \ldots, \gamma_N)
:= \sum_{i=1}^{N} F^{S^{[n]}}_{g,\ell}(\taut; \gamma_1, \ldots, \gamma_{i-1}, T_{\alpha} \gamma_i, \gamma_{i+1}, \ldots,  \gamma_N),
\end{equation}
\item in the stable case, where $\taut = \tau^{\ast}(\alpha)$, we let $\taut':=\tau^{\ast} \iota^{\ast}(\alpha)$
where $\iota : \Mbar_{g-1,N+2} \to \Mbar_{g,N}$ is the gluing map, in the unstable case, where $\taut=1$, we set $\taut':=1$,
\item $\taut_1, \taut_2$ stands for summing over the K\"unneth decomposition of
$\xi^{\ast}(\taut)$ where $\xi$ is the gluing map
\[ \xi : \Mbar_{g_1,|A|+1}(S^{[n]},\beta+rA) \times \Mbar_{g_2,|B|+1} \to \Mbar_{g,N}(S^{[n]},\beta+rA), \]
\item we let $\psi_i \in H^2(\Mbar_{g,N}(S^{[n]},\beta+rA))$ be the cotangent line class at the $i$-th marking.
\end{itemize}
\end{conj}

Conjecture~\ref{Conj:HAE} determines any $F_{g,\ell}$ up to a finite list of coefficients,
where the list is sufficiently short for this to be actually useful in applications.
For example, the conjecture determines all Gromov-Witten invariants of $S^{[2]}$
from seven elementary computations, see \cite{COT} where this leads to a Yau-Zaslow type formula for the counts of genus $2$ curves on hyperk\"ahler fourfolds of $K3^{[2]}$-type.

For K3 surfaces (the case of the Hilbert scheme of $n=1$ points) the above conjectures are well-known.
In this case, Conjecture~\ref{MCConjecture} was made in \cite{K3xE}, and
Conjecture~\ref{Conj: Quasi Jacobi property} reduces to the prediction
of Maulik, Pandharipande and Thomas that the series $F_{g,\ell}$ are quasi-modular forms for $\Gamma_0(\ell)$, see \cite{MPT}.
The holomorphic anomaly equation (Conjecture~\ref{Conj:HAE}) was proven in \cite{HAE} for $\ell=1$ and then conjectured in \cite{BB} for arbitrary $\ell$.
There is also sufficient evidence:
%

\begin{thm}[\cite{MPT, HAE, BB}]
For $S^{[1]} \cong S$, the above conjectures hold for all $g,N$ and $\ell \in \{ 1,2 \}$.
\end{thm}

For Hilbert schemes of points $S^{[n]}$ with $n>1$
Conjecture~\ref{MCConjecture} was proposed in \cite{ObMC}
based on computations using Noether-Lefschetz theory.
Since then the following strong evidence for all $n \geq 1$ was given:

\begin{thm}[{\cite[Thm.1.4]{QuasiK3}}]\label{MCConjecture for g=0 N<=3}
 Conjecture~\ref{MCConjecture} holds for $g=0$ and $N \leq 3$ markings. \end{thm}

The quasi-Jacobi form property (Conjecture~\ref{Conj: Quasi Jacobi property})
appeared in an early form already in \cite[Conj.J]{HilbK3},
where it was stated in genus $0$ for primitive classes.
On the other hand, the holomorphic anomaly equation (Conjecture~\ref{Conj:HAE}) is new and a main result of this paper.

Holomorphic anomaly equations are predicted for the Gromov-Witten theory of Calabi-Yau manifolds by string theory \cite{BCOV}.
In the last years, this structure was proven in various geometries,
such as for elliptic orbifold projective lines \cite{MRS}, 
elliptic curves \cite{HAE}, formal elliptic curves \cite{Wang2},
local $\p^2$ \cite{LP1,CI}, local $\p^1 \times \p^1$ \cite{Lho1,Wang1}
relative $(\p^2,E)$ \cite{BFGW},
$\BC^3/\BZ_3$ \cite{LP2,CI},
toric Calabi-Yau 3-folds \cite{Eynard1,Eynard2,Fang1,Fang2},
the formal quintic 3-fold \cite{LP3}, the quintic 3-fold \cite{GJR,CGL},
and (partially) elliptic fibrations \cite{OPix2} and K3 fibrations \cite{Lho2}.
Conjecture~\ref{Conj:HAE} is maybe the first instance where a general holomorphic anomaly equation is considered in higher dimensions.
The interaction here with the LLV Lie algebra is a new phenomenon that needs further exploration.
E.g., are there connections with the Lie algebra which appears in \cite{AMS}?

\subsection{Main results}
The main result of this paper is the following.
\begin{thm} \label{thm:Main result}
For all Hilbert schemes of points $S^{[n]}$ (i.e. any $n \geq 1$),
Conjecture~\ref{Conj: Quasi Jacobi property} and Conjecture~\ref{Conj:HAE} holds for $g=0$ and $N \leq 3$ markings.
\end{thm}

In particular, this result shows that for fixed $n$, computing finitely many Gromov-Witten invariants of $S^{[n]}$,
where $S$ is the elliptic K3 surface,
determines all $3$-pointed genus $0$ invariants
of all Hilbert schemes of $n$ points on K3 surfaces.
This shows the following qualitative result (see \cite{HilbK3} for the definition of reduced quantum cohomology):
\begin{cor}
For any $n \geq 1$, the reduced quantum cohomologies of $\mathrm{QH}^{\ast}(X)$ of all hyperk\"ahler varieties of $K3^{[n]}$-type $X$
can be effectively reconstructed from finitely many Gromov-Witten invariants of $S^{[n]}$, where $S \to \p^1$ is the elliptic K3 surface with section.
\end{cor}
%

\begin{example}
Let $\CL \in H^{2n}(S^{[n]})$ be the class of a fiber of the Lagrangian fibration
$S^{[n]} \to \p^n$.
An easy computation\footnote{In the Nakajima basis of Section~\ref{subsec:Nakajima operators}, we have $\CL=\Fq_1(F)^{n} \1$,
which implies the claim.}
shows $\wt(\CL)=-n$.
Hence by the Theorem we find
\begin{equation} \label{exm fdfsdf} F_{g=0,1}^{S^{[n]}}( 1; \CL, \CL ) \in \frac{1}{\Delta(q)} \QJac_{2-2n, n-1}. \end{equation}
The space $\QJac_{2-2n, n-1}$ is one-dimensional spanned by $\Theta(p,q)^{2n-2}$,
so \eqref{exm fdfsdf} is determined up to a single constant.
The class of a line in the section $\p^n \subset S^{[n]}$ is $B-(n-1)A$.
Since there is a unique line through any two points in $\p^n$, we have
\[ \blangle 1 ; \CL, \CL \brangle^{S^{[n]}}_{0, B-(n-1)A} = 1. \]
This yields the explicit evaluation
\begin{equation} \label{FF evaluation} 
F_{0,1}^{S^{[n]}}( 1; \CL, \CL ) = (-1)^{n-1}\frac{\Theta(p,q)^{2n-2}}{\Delta(q)}.
\end{equation}
This evaluation was previously obtained (with hard work) in \cite[Thm 1]{HilbK3}. \qed
\end{example}

We give a more fundamental example, where the holomorphic anomaly equation determines the transformation law of the quasi-Jacobi form.
Consider the generating series of $2$-point Gromov-Witten classes
\[
\widetilde{\Z}^{S^{[n]}}(p,q) = \sum_{d = -1}^{\infty} \sum_{r \in \BZ} q^d (-p)^r (\ev_1 \times \ev_2)_{\ast}\left( [ \Mbar_{0,2}(S^{[n]}, B+(d+1)F+rA) ]^{\vir} \right)
\]
which is an element of $H^{\ast}(S^{[n]})^{\otimes 2} \otimes q^{-1} \BC((p))[[q]]$.
Add a quasi-Jacobi correction term:
\begin{equation} \label{def 2 pt function}
\Z^{S^{[n]}}(p,q) := \widetilde{\Z}^{S^{[n]}}(p,q) - \frac{\mathbf{G}(p,q)^{n}}{\Theta(p,q)^2 \Delta(q)} \Delta_{S^{[n]}}
\end{equation}
where $\Delta_{S^{[n]}}$ is the class of the diagonal in $(S^{[n]})^{2}$, and
\[ \mathbf{G}(p,q) = -\Theta(p,q)^2 \left( p \frac{d}{dp} \right)^2 \log(\Theta(p,q)). \]

We have the following corollary:
\begin{cor} \label{cor:transformation law of 2 pt operator}
Under the variable change $p=e^{2 \pi i x}$ and $q=e^{2 \pi i \tau}$, the function
\[ \Z^{S^{[n]}} : \BC \times \BH \to H^{\ast}(S^{[n]} \times S^{[n]},\BC),
\quad (x,\tau) \mapsto \Z^{S^{[n]}}(x,\tau) \] 
is a vector-valued Jacobi form of weight $-10$ and index $n-1$ with a double poles at lattice points.
In particular, we have the transformation laws
\begin{multline*}
\Z^{S^{[n]}}\left( \frac{x}{c \tau + d}, \frac{a \tau + b}{c \tau + d} \right)  
= (c \tau + d)^{-10-\Wt} e\left( \frac{c (n-1) x^2}{c \tau + d} \right) \\
\cdot \exp\left( - \frac{c}{c \tau + d} \left[ \frac{1}{4 \pi i} \sum_{\alpha,\beta} (\tilde{g}^{-1})_{\alpha \beta} T_{\alpha} T_{\beta} + x T_{\delta} \right] \right)
\Z^{S^{[n]}}(x,\tau)
\end{multline*}
\vspace{-15pt}
\begin{gather*}
\Z^{S^{[n]}}(x + \lambda \tau + \mu, \tau)
= e\left( -(n-1) \lambda^2 \tau - 2 \lambda (n-1) x \right) \exp\left( \lambda T_{\delta} \right) \Z^{S^{[n]}}(x,\tau),
\end{gather*}
for all $\binom{a\ b}{c\ d} \in \SL_2(\BZ)$ and $\lambda, \mu \in \BZ$,
where we have written $e(x)=e^{2 \pi i x}$ for $x \in \BC$.
\end{cor}

We refer to Section~\ref{subsec:2 pt function} for the precise definitions and conventions that we use here.
A formula for the series $\Z^{S^{[n]}}(p,q)$ was conjectured in \cite{HilbK3}
and then refined to an explicit conjecture in \cite{vIOP}.
The above corollary yields strong evidence for this conjecture.

The cycle $\Z^{S^{[n]}}(p,q)$ also appears naturally in
the Pandharipande-Thomas theory of the relative threefold $(S \times \p^1,S_{0,\infty})$.
Indeed, by Nesterov's quasi-map wallcrossing \cite{N1,N2} and the computation of the wall-crossing term in \cite{QuasiK3} one has
\[
\Z^{S^{[n]}}(p,q) = \sum_{d,r} q^{d} (-p)^{r} (\ev_{0} \times \ev_{\infty})_{\ast} \left[ P^{\sim}_{r,  (B+(d+1)F,n)}(S \times \p^1, S_{0,\infty}) \right]^{\vir}
\]
where the moduli space on the right parametrizes stable pairs $(F,s)$ on the
relative rubber target $(S \times \p^1, S_{0,\infty})^{\sim}$ with Chern character $\ch_3(F)=r$.
Consider the Pandharipande-Thomas theory of $S \times E$, where $E$ is an elliptic curve.
By using the evaluation in \cite{HAE}
and by degenerating the elliptic curve \cite{K3xE}, one obtains the closed formula:
\[
\sum_{n=0}^{\infty} \tilde{q}^{n-1} \int_{S^{[n]} \times S^{[n]}} \Z^{S^{[n]}}(p,q) \cup \Delta_{S^{[n]}} = -\frac{1}{\chi_{10}(p,q,\tilde{q})}
\]
where $\chi_{10}$ is the weight $10$ Igusa cusp form (as in \cite{K3xE}).
Because Fourier coefficients of Siegel modular forms are Jacobi forms, this matches nicely with
Corollary~\ref{cor:transformation law of 2 pt operator}.

\subsection{An application: CHL Calabi-Yau threefolds} \label{subsec:intro CHL}
Let $S \to \p^1$ be an elliptic K3 surface with section $B$ and fiber class $F$,
and let $g : S \to S$ be a symplectic involution such that
\[ \Pic(S) = \begin{pmatrix} -2 & 1 \\ 1 & 0 \end{pmatrix} \oplus E_8(-2) \]
where the first summand is generated by $B,F$ and the second summand is the anti-invariant part.\footnote{These K3 surfaces arise as follows: Let $R \to \p^1$
be a generic rational elliptic surface,
and let $\p^1 \to \p^1$ be a double cover, branched away from the discriminant.
Then consider the K3 surface $S = R \times_{\p^1} \p^1$
and let $g$ be the composition of the covering involutions with the fiberwise multiplication by $(-1)$. This involution is symplectic and has the desired properties, see \cite[Sec.5.1]{BO-CHL}.}
Let $E$ be an elliptic curve and let $\tau : E \to E$ be translation by a $2$-torsion point.
The Chaudhuri--Hockney--Lykken (CHL) Calabi-Yau threefold associated to $(g,\tau)$ is the quotient
\[ X = ( S \times E )/ \langle g \times \tau \rangle. \]
The group of algebraic $1$-cycles on $X$ is
\[ N_1(X) \cong \mathrm{Span}_{\BZ}(B,F) \oplus \BZ [E'] \]
where the second summand records the degree over the elliptic curve $E' = E/\langle \tau \rangle$.

Define the Donaldson-Thomas partition function
\[ \DT_n(X) = \sum_{d \geq -1} \sum_{r \in \BZ} \DT_{r,(B+dF,n)} q^{d-1} (-p)^{r} \]
where we used the reduced Donaldson-Thomas invariants (see \cite{BO-CHL})
\[ \DT_{r,\beta} = \int_{[ \Hilb_{r,\beta}(X)/E ]^{\vir}} 1. \]

\begin{thm} \label{thm:CHL} Every $\DT_n(X)$ is a Jacobi form of weight $-6$ and index $n$, that is
\[ \DT_n(X) \in \frac{1}{\Theta(p,q)^2 \Delta(\tau)} \Jac_{4, n}(\Gamma_0(2)). \]
\end{thm}

The rank $1$ Donaldson-Thomas invariants of $X$ in arbitrary curve classes are determined from the series 
$\DT_n$ by the multiple cover formula of \cite{QuasiK3} and a degeneration argument \cite{BO-CHL}.
Hence Theorem~\ref{thm:CHL} puts strong constraints on the full rank $1$ Donaldson-Thomas theory of $X$.
For an explicit conjectural formula for the $\DT_n$, see \cite{BO-CHL}.

Our methods can apply also to arbitrary CHL Calabi-Yau threefolds which are
associated to symplectic automorphism of K3 surfaces of any finite order.
The above is just the simplest case notation-wise and choosen here to illustrate the method.
The Donaldson-Thomas theory of general CHL Calabi-Yau threefolds will be studied at a later time.

\subsection{Fiber classes and Lagrangian fibrations}
Assume that we are in the stable case
$2g-2+N>0$. Consider the generating series of Gromov-Witten invariants
in fiber classes of the Lagrangian fibration $S^{[n]} \to \p^n$:
\[ F^{S^{[n]}}_{g,0}(\taut; \gamma_1, \ldots, \gamma_N) :=
\sum_{d \geq 0} \sum_{\substack{k \in \BZ \\ (d,k) \neq 0}} 
\left\langle \taut ; \gamma_1, \ldots, \gamma_N \right\rangle^{S^{[n]}}_{g, dF+rA} q^d (-p)^r. \]
We have to exclude here the term $(d,r) = (0,0)$, because
reduced Gromov-Witten invariants are not defined for a vanishing curve class.
The price that we pay for this unnatural definition
is that we work modulo the constant term below.
Given power series $f,g \in \BC((p))[[q]]$ we write
$f\equiv g$ if the two power series are equal in $\BC((p))[[q]]/\BC$, or equivalently if $f=g+c$ for a constant $c\in \BC$.
In the unstable cases $2g-2+N \leq 0$ we define
\[ F^{S^{[n]}}_{g,0}(\taut; \gamma_1, \ldots, \gamma_N) = 0. \]

We first state the conjectural quasi-Jacobi property and holomorphic anomaly equation.

\begin{conj} \label{conj:Fiber class} Assume that $2g-2+N>0$.  We have the following:
\begin{enumerate}[leftmargin=17pt]
\item[(i)] (Quasi-Jacobi form property)
Up to a constant term,
$F^{S^{[n]}}_{g,0}(\taut; \gamma_1, \ldots, \gamma_N)$
is a meromorphic quasi-Jacobi form of weight $k=n(2g-2+N) + \sum_i \wt(\gamma_i)$ and index $0$
with poles at torsion points $z=a \tau + b$, $a,b \in \BQ$.
\item[(ii)] (Holomorphic Anomaly Equations) Modulo constants, i.e. in $\BC((p))[[q]]/ \BC$, we have
\begin{align*}
\frac{d}{d G_2} F^{S^{[n]}}_{g,0}(\taut; \gamma_1, \ldots, \gamma_N)
\equiv & F^{S^{[n]}}_{g-1,0}(\taut'; \gamma_1, \ldots, \gamma_N, U) \\
+ & 2 \sum_{\substack{g=g_1 + g_2 \\ \{ 1, \ldots, N \} = A \sqcup B }} F^{S^{[n]}}_{g_1,0}(\taut_1 ; \gamma_A, U_1) F_{g_2}^{S^{[n]}, \std}(\taut_2 ; \gamma_B , U_2 ) \\
- & 2 \sum_{i=1}^{N} F^{S^{[n]}}_{g,0}( \tau^{\ast}(\psi_i) \cdot \taut ; \gamma_1, \ldots, \gamma_{i-1}, U \gamma_i, \gamma_{i+1}, \ldots, \gamma_N),
\end{align*}
where $\psi_i \in H^2(\Mbar_{g,N})$ is the cotangent line class, and
\[
\frac{d}{d \A} F^{S^{[n]}}_{g,0}(\taut; \gamma_1, \ldots, \gamma_N)
\equiv
\sum_{i=1}^{N} F^{S^{[n]}}_{g,0}(\taut; \gamma_1, \ldots, T_{\delta} \gamma_i, \ldots,  \gamma_N).
\]
\end{enumerate}
\end{conj}

\begin{thm} \label{thm:fiber classes intro}
Conjecture~\ref{conj:Fiber class} holds in the following cases:
\begin{enumerate}
\item[(i)] For the K3 surface $S$ (i.e. if $n=1$) and for all $g,N$
\item[(ii)] For all Hilbert schemes $S^{[n]}$ (that is for arbitrary $n$), if $(g,N)=(0,3)$.
\end{enumerate}
\end{thm}

We refer to Theorem~\ref{thm:fiber classes} for the precise form which the quasi-Jacobi forms described in (i) have.
The multiple cover conjecture (Conjecture~\ref{MCConjecture}) was proven for the K3 surface $S$ in fiber classes $dF$
in \cite{BB}. The observation that the corresponding generating series
is quasi-modular and satisfies a holomorphic anomaly equation
appears to be new (but follows easily from the known methods).
The case of the Hilbert scheme of points also follows from the multiple cover conjecture,
together with some subtle vanishing arguments.

Deformation invariance and similar methods as in our proof should show that for any
Lagrangian fibration $\pi: X \to \p^n$ of a $K3^{[n]}$-hyperk\"ahler with a section,
the generating series of Gromov-Witten invariants in fiber classes is a (lattice index) quasi-Jacobi form
and satisfies a holomorphic anomaly equation.
This raises the following question:

\begin{question}
Consider any Lagrangian fibration $X \to B$ with section of a holomorphic-symplectic variety $X$.
Are the generating series of Gromov-Witten invariants in fiber classes quasi-Jacobi forms,
and do they satisfy a holomorphic anomaly equation?
\end{question}

The answer is very likely 'yes'. However, more interestingly we can also ask this for cases where $X$ is quasi-projective hyperk\"ahler.
A prototypical example to consider is the Hitchin map $\CM_{C,n} \to \oplus_{i} H^0(C,K_C^i)$
from the moduli space of rank $n$ Higgs bundles on a curve $C$.
Evidence for a positive answer to the question will be given in the genus $1$ case (more precisely, for the Hilbert scheme of points on $E \times \BC$) in \cite{OP_ExC}.

\subsection{Strategy of the proof}
Hilbert schemes of points on K3 surfaces lie in the intersections of two very special classes of varieties:
(i) (irreducible) hyperk\"ahler varieties, and (ii) Hilbert schemes of points on surfaces.
The geometry of (i) and (ii) will each imply a modular constraint on the generating series of Gromov-Witten invariants.
We will show that these two constraints are precisely the two modular transformation equations that a Jacobi form has to satisfy.

From hyperk\"ahler geometry, we use the global Torelli theorem \cite{Ver, HuyTor} and the description of the monodromy in \cite{Markman}.
The locus parametrizing Hilbert schemes of points $S^{[n]}$ on K3 surfaces
is a divisor in the moduli space of all hyperk\"ahler varieties of $K3^{[n]}$-type.
In particular, there are deformations of $S^{[n]}$ 
which do not arise from deformations of the underlying K3 surface $S$ (these deformation may be thought of as deforming the K3 surface $S$ in a non-commutative way).
Utilizing these extra deformations yields precisely one of the transformation properties that we need. 

The other ingredient follows from the Hilbert scheme side.
Given a surface $S$ there is a correspondence between three different counting theories:
\begin{enumerate}
\item[(i)] Quantum Cohomology (i.e. $(g,N)=(0,3)$ Gromov-Witten theory) of $S^{[n]}$,
\item[(ii)] Pandharipande-Thomas theory of the relative threefold $(S \times \p^1, S_{0,1,\infty})$,
\item[(iii)] Gromov-Witten theory of the relative threefold $(S \times \p^1, S_{0,1,\infty})$.
\end{enumerate}
This correspondence is often represented in the triangle:
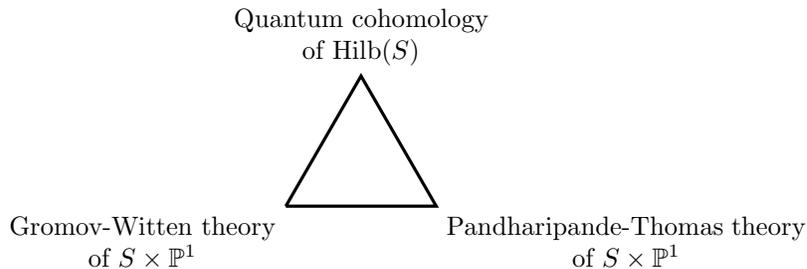
\begin{center}
\begin{tikzpicture}
 \draw[very thick]
   (0,0) node[below left, align = center]{Gromov-Witten theory\\of $S \times \p^1$}
-- (1,1.732) node[above, align = center]{Quantum cohomology\\of $\Hilb(S)$}
-- (2,0) node[below right, align = center]{Pandharipande-Thomas theory\\of $S \times \p^1$}
-- (0,0);
\end{tikzpicture}
\end{center}
The GW/PT correspondence (meaning the correspondence between (ii) and (iii)) was proposed in \cite{MNOP1, MNOP2}
and was proven since then in many instances in \cite{MOOP, PPDC, PaPix_GWPT}. 
For $K3 \times \p^1$ it was recently established in \cite{MarkedRelative} for curve classes which are {\em primitive}
over the surface.
The Hilb/PT correspondence (between (i) and (ii)) was recently established in full generality by Nesterov \cite{N1}.
For $\BC^2$ and resolutions of $A_n$ singularities the triangle of correspondences was worked out previously in \cite{OkPandHilbC2, BP, OPLocal, MO, M, MO2,Liu}.

In the case of K3 surfaces the above correspondences take the simplest form: they are straight equalities, without wallcrossing corrections, see Theorem~\ref{thm:rel HAE} and \cite{N2}.
By expressing invariants of the Hilbert schemes in terms of invariants of $S \times \p^1$ and then applying the product formula in Gromov-Witten theory
we hence have expressed the Gromov-Witten invariants of the Hilbert scheme in terms of those of the K3 surface.
This allows us to lift modular properties which are known for K3 surfaces
to the Hilbert scheme of points.
Altogether, this provides precisely the other half of the modularity that we were missing.

This leads to the proof of Theorem~\ref{thm:Main result} for primitive classes $(\ell=1)$.
To deduce the arbitrary case we use the proven case of the multiple cover conjecture \cite{QuasiK3}
and check the compatibility of our conjectures under the formal Hecke operator.
Except for working out the required compatibility on quasi-Jacobi forms this last step is not difficult.

\subsection{History}
The Gromov-Witten theory of the Hilbert schemes of points of K3 surfaces was first studied
by the author in his PhD thesis in \cite{PhD}.
Many ideas behind the current work were already anticipated then.
For example, the potential role of the monodromy was discussed in \cite[Sec.6.3]{PhD},
and the quasi-Jacobi form property was conjectured in a simple case in \cite[Sec.5.1.3]{PhD}.
Interestingly, the simplest evaluation on the Hilbert scheme from a weight point of view, given in \eqref{FF evaluation},
is precisely also the case where the moduli space of stable maps is the simplest to describe, and indeed this case was the first to be computed back then.


\subsection{Plan of the paper}
In Section~\ref{sec:quasi-Jacobi forms} we review the definition of quasi-Jacobi forms
and prove basic properties regarding their $z$-expansions, their anomaly operators, and how they interact with Hecke operators.
In Section~\ref{sec:cohomology} we introduce the LLV algebra on the cohomology of the Hilbert scheme and then describe explicitly the two monodromy operators that we need for constraints of the Gromov-Witten generating series (see Sections~\ref{monodromy:involution} and~\ref{monodromy:involution}).
In Section~\ref{sec:constraints from monodromy} we use these two monodromies and
obtain our first structure result for the generating series of the Hilbert scheme in Proposition~\ref{prop:constrains monodromy},
essentially proving the elliptic transformation law.

Then we turn to the part on GW/PT/Hilb correspondences:
In Section~\ref{sec:Rel GW theory} we discuss first several basic structures in relative Gromov-Witten theory. The main new technical result here is a formula for the restriction of relative Gromov-Witten classes to the non-separating boundary divisor in the moduli space of curves, which is of independent interest.
In Section~\ref{sec:Rel GW of K3xC} we specialize to $(K3 \times C, K3_z)$ for a curve $C$, state the GW/Hilb correspondence (Theorem~\ref{thm:GW/Hilb}), 
the reduced degeneration formula, and make some preliminary explicit computations of invariants.
The goal of Section~\ref{sec:HAE on K3} is to use the product formula and results about the K3 surface to show that the Gromov-Witten invariants of $(K3 \times C, K3_z)$ are quasi-modular forms and satisfy a holomorphic anomaly equation (Theorem~\ref{thm:rel HAE}). This is our second main structure result.

Section~\ref{sec:HAE primitive case} is the heart of the paper: Here we combine the two structure results we obtained before (Proposition~\ref{prop:constrains monodromy} and Theorem~\ref{thm:rel HAE}) and 
match the holomorphic anomaly equation on the Hilbert scheme
with the holomorphic anomaly equation for $(K3 \times C, K3_z)$
under the GW/Hilb correspondence.
This proves Theorem~\ref{thm:Main result} in case $\ell=1$.
The case $\ell>1$ follows then in Section~\ref{sec:HAE imprimitive case} by a formal argument using Hecke operators.
Section~\ref{sec:fiber case}
deals with the fiber classes proving Theorem~\ref{thm:fiber classes intro} by a combination of the GW/Hilb correspondence and known cases of the multiple cover conjecture.
Section~\ref{sec:applications} discusses the applications to the $2$-point function and the CHL Calabi-Yau threefolds.
%

\subsection{Conventions} \label{subsec:Conventions}
Let $X$ be a smooth projective variety.
Given a cohomology class $\gamma \in H^k(X)$ we let
$\deg(\gamma) = k/2$ denote its complex degree.
We will use the identification
$H^{\ast}(X \times X) \cong \End H^{\ast}(X)$
which is given by sending 
a class $\Gamma \in H^{\ast}(X \times X)$ to the operator
\[ \Gamma : H^{\ast}(X) \to H^{\ast}(X), \quad \gamma \mapsto \pi_{2 \ast}(\pi_1^{\ast}(\gamma) \cdot \Gamma), \]
where $\pi_1, \pi_2$ are the projections of $X^2$ to the factors.
Given a function $Z : H^{\ast}(X) \to \BQ$ we will often write
$Z(\Gamma_1) \cdot Z(\Gamma_2)$
and say that $\Gamma_1, \Gamma_2$ stands for summing over the K\"unneth decomposition of the class $\Gamma \in H^{\ast}(X \times X)$.
By this we mean
\[ Z(\Gamma_1) \cdot Z(\Gamma_2) := \sum_{i} Z( \phi_i ) Z( \phi_i^{\vee} ) \]
where $\Gamma = \sum_i \phi_i \otimes \phi_i^{\vee} \in H^{\ast}(X \times X)$ is a K\"unneth decomposition.
A {\emph curve class} on $X$ is any homology class $\beta \in H_2(X,\BZ)$. It is effective if there exists a non-empty algebraic curve $C \subset X$
with $[C]=\beta$. In particular, any effective class $\beta$ is non-zero. An effective class $\beta$ is primitive if it is not divisible in $H_2(X,\BZ)$.

\subsection{Acknowledgements}
Its a pleasure to thank Jim Bryan, Davesh Maulik, Denis Nesterov, and Rahul Pandharipande for discussions related to this work.
I also thank Johannes Schmitt for helpful discussions about how to restrict relative Gromov-Witten classes to the boundary,
and for providing the technical result stated in Proposition~\ref{prop:Schmitt}.

The author was funded by the Deutsche Forschungsgemeinschaft (DFG) - OB 512/1-1, and by the starting grant 'Correspondences in enumerative geometry: Hilbert schemes, K3 surfaces and modular forms', No 101041491
 of the European Research Council.

\section{Quasi-Jacobi forms} \label{sec:quasi-Jacobi forms}
\subsection{Overview}
Jacobi forms are two-variable generalizations of classical modular forms.
Quasi-Jacobi forms are constant terms of almost-holomorphic Jacobi forms.
We introduce here the basic facts we need on quasi-Jacobi forms and refer to \cite{Lib}, \cite[Sec.1]{OPix2} and \cite{vIOP} for more detailed discussions.
The topics we cover are the generators of the ring of quasi-Jacobi forms, differential and anomaly operators, and the Fourier and Taylor expansion of quasi-Jacobi forms.
Conversely, we give criteria on two-variable generating series to be Taylor or Fourier expansions of quasi-Jacobi forms.
In Section~\ref{subsec:Hecke operators} 
we discuss Hecke operators on quasi-Jacobi forms, 
and in Section~\ref{subsec:Hecke operators wrong} we consider their action on forms of the wrong weight.
In Section~\ref{subsec:index 0 meromorphic JAcobi forms} we discuss a classical series of meromorphic quasi-Jacobi forms which will later appear for fiber classes of Lagrangian fibrations in Section~\ref{sec:fiber case}.

\subsection{Definition}
Let $\BH = \{ \tau \in \BC : \mathrm{Im}(\tau) > 0 \}$ be the upper half plane and let $q = e^{2 \pi i \tau}$.
Let also $x \in \BC$ and $p=e^{2 \pi i x}$. We will also frequently use the variable
\[ z = 2 \pi i x.\]
We often write $f(p)$ or $f(z)$ for a function $f(x)$ under the above variable change.
Consider the real-analytic functions
\[ \nu = \frac{1}{8 \pi \Im(\tau)}, \quad \alpha = \frac{\Im(x)}{\Im(\tau)}. \]
An almost holomorphic function on~$\BC \times \BH$ is a function of the form
\begin{equation} \Phi = \sum_{i, j \geq 0} \phi_{i,j}(x,\tau) \nu^i \alpha^j \label{ahm function} \end{equation}
such that each of the finitely many non-zero functions $\phi_{i,j}$ is holomorphic
and admits a Fourier expansion of the form~$\sum_{n \geq 0} \sum_{r \in \BZ} c(n,r)q^n p^r$ in the region~$|q|<1$.

Consider a congruence subgroup
\[ \Gamma \subset \SL_2(\BZ) \]
and write $e(x) = e^{2\pi i x}$ for $x \in \BC$.
\begin{defn}
An \emph{almost holomorphic weak Jacobi form} of weight $k$ and index $m$
for the group $\Gamma$ is a function $\Phi(x,\tau) : \BC \times \BH \to \BC$
which (i) satisfies the transformation laws
\begin{equation} \label{TRANSFORMATIONLAWJACOBI}
\begin{aligned}
\Phi\left( \frac{x}{c \tau + d}, \frac{a \tau + b}{c \tau + d} \right)
& = (c \tau + d)^k e\left( \frac{c m x^2}{c \tau + d} \right) \Phi(x,\tau) \\
\Phi\left( x + \lambda \tau + \mu, \tau \right)
& = e\left( - m \lambda^2 \tau - 2 \lambda m x \right) \Phi(x,\tau)
\end{aligned}
\end{equation}
for all $\binom{a\ b}{c\ d} \in \Gamma$ and $\lambda, \mu \in \BZ$,
and (ii) such that
\[
(c \tau + d)^{-k} e\left( - \frac{c m x^2}{c \tau + d} \right) \Phi\left( \frac{x}{c \tau + d}, \frac{a \tau + b}{c \tau + d} \right)
\]
is an almost-holomorphic function for all $\binom{a\ b}{c\ d} \in \SL_2(\BZ)$.
\end{defn}

\begin{rmk}
By taking $\binom{a\ b}{c\ d}$ to be the identity in (ii), we see that any almost holomorphic weak Jacobi form 
is an almost holomorphic function, and hence has an expansion \eqref{ahm function}.
Condition(i) implies that (ii) only needs to be checked for a set of representatives of $\Gamma \backslash \SL_2(\BZ)$.
In particular, if $\Gamma = \SL_2(\BZ)$ the condition (ii) simply says that $\Phi$ is an almost-holomorphic function.
\end{rmk}
An almost-holomorphic weak Jacobi form $\Phi$, which is as a function $\Phi : \BC \times \BH \to \BC$ holomorphic, is called a {\em weak Jacobi form}.
More generally, we can consider the holomorphic part of an almost-holomorphic weak Jacobi form:

\begin{defn}
A \emph{quasi-Jacobi form} of weight $k$ and index $m$ for $\Gamma$ is a function $\phi(x,\tau)$ on $\BC \times \BH$
such that there exists an almost holomorphic weak Jacobi form
$\sum_{i,j} \phi_{i,j} \nu^i \alpha^j$ of weight $k$ and index $m$
with $\phi_{0,0} = \phi$.
\end{defn}

We let $\AHJ_{k,m}(\Gamma)$ (resp. $\QJ_{k,m}(\Gamma)$, resp $\Jac_{k,m}(\Gamma)$) be the vector space of almost holomorphic weak (resp. quasi-, resp. weak) Jacobi forms
of weight $k$ and index $m$ for the group $\Gamma$.
We write
\[
\QJac(\Gamma) = \bigoplus_{m \geq 0} \bigoplus_{k \in \BZ} \QJac(\Gamma)_{k,m}
\]
for the bigraded $\BC$-algebra of quasi-Jacobi forms,
and similar for $\AHJ(\Gamma)$ and $\Jac(\Gamma)$.

\begin{lemma} \label{lemma:constant term maps}
The constant term map
\[ \AHJ(\Gamma)_{k,m} \to \QJ(\Gamma)_{k,m}, \quad \sum_{i,j} \phi_{i, j} \nu^{i} \alpha^j \mapsto \phi_{0,0} \]
is well-defined and an isomorphism.
\end{lemma}
\begin{proof}
This is proven in \cite{Lib}.
\end{proof}

A quasi-modular form of weight $k$ for the congruence subgroup $\Gamma$
is a quasi-Jacobi form of weight $k$ and index $0$ for $\Gamma$. 
The algebra of quasi-modular forms is denoted by
\[ \QMod(\Gamma) = \bigoplus_k \QMod(\Gamma)_k, \quad \QMod(\Gamma)_k = \QJac(\Gamma)_{k,0}. \]

\begin{rmk}
(i) If $\Gamma$ is the full modular group $\SL_2(\BZ)$, we will usually omit $\Gamma$ from our notation, e.g.
\[ \QJac = \QJac(\SL_2(\BZ)). \]
(ii) In what follows, we will often identify a quasi-Jacobi form $f(x,\tau) \in \QJac_{k,m}$ with its power series in $p,q$. We will also often write $f(p,q)$ instead of $f(x,\tau)$.
\end{rmk}

\subsection{Presentation by generators: Quasi-modular forms}
For all even~$k > 0$ consider the Eisenstein series
\[ G_k(\tau) = - \frac{B_k}{2 \cdot k} + \sum_{n \geq 1} \sum_{d|n} d^{k-1} q^n. \]
Set also $G_k = 0$ for all odd $k>0$.
We have that $G_k$ is a modular form of weight $k$ for $k>2$, and $G_2$ is quasi-modular.
By \cite{KZ,BO} the algebra of quasi-modular forms is a free polynomial ring in $G_2$
over $\Mod(\Gamma)$, i.e. the ring of modular forms for the group $\Gamma$:
\[ \QMod(\Gamma) = \Mod(\Gamma)[G_2]. \]
For the full modular group $\Gamma=\SL_2(\BZ)$ we have
\[ \QMod = \BC[G_2, G_4, G_6]. \]

\subsection{Presentation by generators: Quasi-Jacobi forms}
Consider the odd (renormalized) Jacobi theta function\footnote{We have 
$\Theta(x,\tau) = \vartheta_1(x,\tau) / \eta^3(\tau)$ where
\[ \vartheta_1(x,\tau) = \sum_{\nu\in \mathbb{Z}+\frac{1}{2}} (-1)^{\lfloor \nu \rfloor} p^\nu q^{\nu^2/2} \]
is the odd Jacobi theta function,
i.e. the unique section on the elliptic curve~$\BC_{x}/(\BZ+\tau \BZ)$ which vanishes at the origin,
and $\eta(\tau) = q^{1/24} \prod_{n \geq 1} (1-q^n)$ is the Dedekind eta function.}
\[
\Theta(x,\tau) 
=  (p^{1/2}-p^{-1/2})\prod_{m\geq 1} \frac{(1-pq^m)(1-p^{-1}q^m)}{(1-q^m)^{2}}.
\]
Consider the derivative operator
$p \frac{d}{dp} = \frac{1}{2\pi i} \frac{d}{dx} = \frac{d}{dz}$
and consider also the series
\begin{align*}
\A(x,\tau) = \frac{p \frac{d}{dp} \Theta(x,\tau)}{\Theta(x,\tau)}
& = - \frac{1}{2} - \sum_{m \neq 0} \frac{p^m}{1-q^m}.
\end{align*}

By the same argument as in \cite{KZ,BO} we have
that $G_2$ and $\A$ are free generators:
\begin{lemma} \label{lemma:G2 A independent}
$\QJac(\Gamma) \subset \Jac(\Gamma)[ G_2, \A ]$.
\end{lemma}

As in the case of quasi-modular forms,
for the full modular group,
the algebra of quasi-Jacobi forms
can be embedded in a polynomial algebra.
Consider the classical Weierstra{\ss} elliptic function
\[
\wp(x,\tau) = \frac{1}{12} + \frac{p}{(1-p)^2} + \sum_{d \geq 1} \sum_{k | d} k (p^k - 2 + p^{-k}) q^{d}.
\]
We write~$\wp'(x,\tau) = p \frac{d}{dp} \wp(x,\tau)$ for its derivative with respect to the first variable.
Consider the polynomial algebra
\[ \MQJac = \BC[\Theta, \A, G_2, \wp, \wp', G_4]. \]

\begin{prop}[{\cite{vIOP}}]
$\MQJac$ is a free polynomial ring on its generators,
and~$\QJac$ is equal to the subring of all polynomials which define holomorphic functions~$\BC \times \BH \to \BH$.
\end{prop}

The generators of $\MQJac$ are quasi-Jacobi forms (with poles and character \cite{vIOP})
of weight and index given in the following table. The algebra $\QJac$ is a graded subring of $\MQJac$.
$$\begin{array}{lll} 
\text{Generator} & \text{Weight} & \text{Index} \\\hline
\Theta & -1 & 1/2 \\
\A & 1  & 0\\
G_2 & 2 & 0\\
\wp & 2 & 0\\
\wp' & 3 & 0\\
G_4 & 4 & 0.
\end{array}$$
\begin{rmk}
By the well-known equation
\[
\wp'(x)^2 - 4 \wp(x)^3 + 20 \wp(x) G_4(\tau) + \frac{7}{3} G_6(\tau) = 0
\]
the generator $G_6$ is not needed as a generator of $\MQJac$.
\end{rmk}

\subsection{Differential and anomaly operators}
As explained in \cite[Sec.2]{OPix2}
the algebra $\QJac(\Gamma)$ is closed under 
the derivative operators
\[
D_{\tau} = \frac{1}{2\pi i} \frac{d}{d \tau} = q \frac{d}{dq}, \quad D_x = \frac{1}{2 \pi i} \frac{d}{dx} = \frac{d}{dz} = p \frac{d}{dp}.
\]
More precisely, these operators act by:
\[ D_{\tau} : \QJac_{k,m}(\Gamma) \to \QJac_{k+2, m}(\Gamma), \quad
D_x : \QJac_{k,m}(\Gamma) \to \QJac_{k+1,m}(\Gamma). \]

Similarly, we have {\em anomaly operators}.
These can be defined most directly as follows.
By Lemma~\ref{lemma:G2 A independent} every quasi-Jacobi form $f(x,\tau)$
can be uniquely written as a polynomial in $A$ and $G_2$ with coefficients weak Jacobi-forms.
We hence can take the formal derivative at these generators,
giving functions $\frac{d}{dG_2} f$ and $\frac{d}{d\A} f$.
If $F = \sum_{i,j} f_{i,j} \nu^i \alpha^j$ is the almost-holomorphic function with $f_{0,0} = f$,
then by \cite[Sec.2]{OPix2} one has 
\[ \frac{d}{dG_2} f = f_{1,0}, \quad \frac{d}{d \A} f = f_{0,1}. \]
This can be used to show that $\frac{d}{dG_2}$ and $\frac{d}{d \A}$ preserves the algebra of quasi-Jacobi forms,
more precisely:
\begin{lemma}[{\cite[Sec.2]{OPix2}}]
The formal derivation with respect to $\A$ and $G_2$ defines operators
\[
\frac{d}{dG_2} : \QJac_{k,m}(\Gamma) \to \QJac_{k-2, m}(\Gamma), \quad
\frac{d}{d\A} : \QJac_{k,m}(\Gamma) \to \QJac_{k-1,m}(\Gamma).
\]
We have the commutative diagrams:
\[
\begin{tikzcd}
\QJ_{k,m} \ar{d}[swap]{\frac{d}{dG_2}} & \ar{l}[swap]{\cong} \AHJ_{k,m} \ar{d}{\frac{d}{d\nu}} & & \QJ_{k,m} \ar{d}[swap]{\frac{d}{d \A}} & \ar{l}[swap]{\cong} \AHJ_{k,m} \ar{d}{\frac{d}{d \alpha}} \\
\QJ_{k-2,m} & \ar{l}[swap]{\cong} \AHJ_{k-2,m} & & \QJ_{k-1,m} & \ar{l}[swap]{\cong} \AHJ_{k-1,m}
\end{tikzcd}
\]
where the horizontal maps are the 'constant term' maps of Lemma~\ref{lemma:constant term maps}.
\end{lemma}

Let~$\mathrm{wt}$ and~$\mathrm{ind}$ be the operators which act on~$\QJac_{k,m}(\Gamma)$ by multiplication by the weight~$k$ and the index~$m$ respectively.
By \cite[(12)]{OPix2} we have the commutation relations:
\begin{equation} \label{eq:comm relations 1}
\begin{alignedat}{2}
\left[ \frac{d}{dG_2}, D_{\tau} \right] & = -2 \mathrm{wt}, & \qquad \qquad \left[\frac{d}{d \A}, D_x \right] & = 2 \mathrm{ind} \\
\left[\frac{d}{dG_2}, D_x \right] & = -2\frac{d}{dA}, &\qquad  \left[\frac{d}{dA}, D_{\tau}\right] & = D_x.
\end{alignedat}
\end{equation}

\begin{rmk}
These commutation relations are proven by checking them for almost-holomorphic Jacobi forms,
where they follow by a straightforward computation of commutators between derivative operators and operators of multiplication by variables.
In particular, the argument is not sensitive to the precise holomorphicity conditions we put on Jacobi forms,
for example the commutation relations \eqref{eq:comm relations 1} hold also for $\MQJac$ or any other ring of meromorphic quasi-Jacobi forms.
\end{rmk}

As explained in \cite{OPix2} knowing the holomorphic-anomaly equations of a quasi-Jacobi form
is equivalent to knowing their transformation properties under the Jacobi group.
Concretely, we have the following:
\begin{lemma}[\cite{OPix2}] \label{lemma:QJac transformation laws}
For any~$\phi(x,\tau) \in \QJac_{k,m}(\Gamma)$ we have
\begin{gather*}
\phi\left( \frac{x}{c \tau + d}, \frac{a \tau + b}{c \tau + d} \right) 
= (c \tau + d)^{k} e\left( \frac{c m x^2}{c \tau + d} \right)
\exp\left( - \frac{c \frac{d}{dG_2}}{4 \pi i (c \tau + d)}  + \frac{c x \frac{d}{d \A}}{c \tau+d} \right) \phi(x,\tau) \\
\phi(x + \lambda \tau + \mu, \tau)
= e\left( -m \lambda^2 \tau - 2 \lambda m x \right) \exp\left( - \lambda \frac{d}{d \A} \right) \phi(x,\tau).
\end{gather*}
for all $\binom{a\ b}{c\ d} \in \Gamma$ and $\lambda, \mu \in \BZ$.
\end{lemma}

\subsection{Elliptic transformation law}
Recall from Lemma~\ref{lemma:QJac transformation laws} the elliptic transformation law of quasi-Jacobi forms:
\begin{lemma} \label{elliptic transformation}
For any $f(p,q) \in \QJac_{k,m}$ and $\lambda \in \BZ$ we have
\[ f(pq^{\lambda},q) = q^{-\lambda^2 m} p^{-2 \lambda m} e^{- \lambda \frac{d}{d A} } f(p,q). \]
\end{lemma}

In particular, if we are given $f(p,q) \in \QJac_{k,m}$ such that $\frac{d}{dA} f = 0$, and
we let
\[ f(p,q) = \sum_{d \geq 0} \sum_{k \in \BZ} c(d,k) q^d p^k \] 
be its Fourier-expansion, then we have
\[
c(d-\lambda k + m \lambda^2, k-2 \lambda m) = c(d,k).
\]
Moreover, since $f(p^{-1}, q) = (-1)^{k} f(p,q)$ where $k$ is the weight of $f$, we have
\[ c(d,k) = (-1)^k c(d,-k). \]

We prove the following two useful lemmas, which serve as a partial converse.
\begin{lemma} \label{lemma:elliptic to Jac 1}
Let $m \geq 0$ and let $f(p,q) = \sum_{n \geq 0} \sum_{k \in \BZ} c(d,k) q^d p^k$ be a formal power series such that the following holds for all $d,k$ and $\lambda \in \BZ$:
\begin{gather}
\label{c1} c(d-\lambda k + m \lambda^2, k-2 \lambda m) = c(d,k), \\
c(d,k) = c(d,-k).
\end{gather}
Then there exists power series $f_i(q) \in \BC[[q]$ such that
\[ f(p,q) = \Theta^{2m}(p,q) \sum_{i=0}^{m} f_i(q) \wp(p,q)^{m-i}. \] 
\end{lemma}
\begin{proof}
A similar argument has appeared in \cite[Sec.4.2]{FM1}
but we recall it here for completeness.
The vector space of Laurent polynomials $g(p)$ such that $g(p^{-1}) = g(p)$
has a basis given by the set of polynomials
\[ (p^{1/2} - p^{-1/2})^{2k}, \quad k \geq 0. \]
Moreover, by the expansions of $\Theta$ and $\wp$
for every $i \in \{0, \ldots, m\}$ there exists $\alpha_j$ (all zero except for finitely many) such that:
\[
\wp(p,q)^{m-i} \Theta(p,q)^{2m} = (p^{1/2} - p^{-1/2})^{2i} + 
\sum_{j>i} \alpha_j (p^{1/2} - p^{-1/2})^{2j} + O(q)
\]
By an induction argument we can hence find $f_i(q) \in \BC[[q]]$ such that the function
\[ F(p,q) := f(p,q) - \Theta^{2m}(p,q) \sum_{i=0}^{m} f_i(q) \wp(p,q)^{m-i} \]
has the following property: for all $d \geq 0$ the $q^d$ coefficient of $F$ satisfies
\begin{equation} F_d(p) := [ F(p,q) ]_{q^d} = \sum_{\ell > m} b_{d,\ell} (p^{1/2}-p^{-1/2})^{2 \ell}. \label{aaaedfdfsd} \end{equation}

Let $a(d,k)$ be the coefficient of
$q^d p^k$ in $F(p,q)$.
Since $\Theta^{2m} \wp^{m-i}$ is a (quasi-) Jacobi
form of index $m$, its Fourier-coefficients satisfy \eqref{c1}.
Moreover, if the Fourier coefficients of a power series $h(p,q)$ satisfy \eqref{c1}, then the same holds for the Fourier coefficients of $h(p,q) \cdot r(q)$ for any power series in $q$.
This implies that we have
\begin{equation} 
a(d,k) = a(d-\lambda k + m \lambda^2, k-2 \lambda m)
\label{Masfsd2} \end{equation}
for all $d,k, \lambda \in \BZ$.
Assume $F(p,q)$ is non-zero and let
$d$ be the smallest integer such that
$F_d(p)$ is non-zero.
Since the sum in \eqref{aaaedfdfsd}
starts at $\ell=m+1$ we have
\[ a(d,k) \neq 0 \]
for some $k \geq m+1 \geq 0$.
But then by \eqref{Masfsd2} with $\lambda=1$
we obtain
\[ a(d,k) = a(d-k+m, k-2m) \neq 0 \,. \]
Since $d-k+m < d$ this contradicts the choice of $d$.
\end{proof}

\begin{lemma} \label{lemma:elliptic to Jac 2}
Let $m \geq 0$ and let $f(p,q) = \sum_{n \geq 0} \sum_{k \in \BZ} c(d,k) q^d p^k$ be a formal power series such that the following holds for all $d,k$ and $\lambda \in \BZ$:
\begin{gather*}
c(d-\lambda k + m \lambda^2, k-2 \lambda m) = c(d,k), \\
c(d,k) = -c(d,-k).
\end{gather*}
Then there exists power series $f_i(q) \in \BC[[q]$ such that
\[ f(p,q) = \Theta^{2m}(p,q) \wp'(p,q) \sum_{i=2}^{m} f_i(q) \wp(p,q)^{m-i}. \] 
\end{lemma}
\begin{proof}
The vector space of Laurent polynomials $g(p)$ such that $g(p^{-1}) = g(p)$
has the basis
\[ (p-p^{-1}) (p^{1/2} - p^{-1/2})^{2k}, \quad k \geq 0. \]
Moreover, for $i \leq m$ we have the expansions:
\[
\Theta^{2m} \cdot \wp' \cdot \wp^{m-i}  = 
(p-p^{-1}) \left( (p^{1/2} - p^{-1/2})^{2i-4} + 
\sum_{j>i-2} \alpha_j (p^{1/2} - p^{-1/2})^{2j} \right) + O(q)
\]
for some $\alpha_j$, of which all but finitely many are zero.
By induction we conclude that there exists
$f_i(q)$ such that
\[ F(p,q) = f(p,q) - \Theta^{2m}(p,q) \wp'(p,q) \sum_{i=2}^{m} f_i(q) \wp(p,q)^{m-i} \]
for all $d \geq 0$ satisfies
\begin{equation} F_d(p) := [ F(p,q) ]_{q^d} = (p-p^{-1}) \sum_{\ell > m-2} b_{d,\ell} (p^{1/2}-p^{-1/2})^{2 \ell}. \label{aaaedfdfsd2} \end{equation}

We argue now as before:
Let $a(d,k)$ be the coefficient of
$q^d p^k$ in $F(p,q)$.
We then still have that \eqref{Masfsd2}
as well as
\[ a(d,k) = -a(d,-k). \]
Assume $F(p,q)$ is non-zero and let
$d$ be the smallest integer such that
$F_d(p)$ is non-zero.
Since the sum in \eqref{aaaedfdfsd2}
starts at $\ell=m-1$ we have
$a(d,k) \neq 0$
for some $k \geq m \geq 0$.
But then by \eqref{Masfsd2} with $\lambda=1$
we obtain
\[ a(d,k) = a(d-k+m, k-2m) \neq 0 \,. \]
If $k>m$ this yields a contradiction as before,
and if $k=m$ then we obtain $a(d,k)=a(d,-k)$,
but since we also have $a(d,k) = -a(d,k)$ this gives
the contradiction $a(d,k)=0$.
\end{proof}

\subsection{The expansion in $z$}
Recall that we have set $z=2 \pi i x$ where $x \in \BC$ is the elliptic parameter.
To stress the dependence on $z$, we usually write $f(z)$ for a function $f(x)$ under this variable change.
We study here the $z$-expansions of quasi-Jacobi forms
for the full modular group $\SL_2(\BZ)$.
For that purpose recall the well-known expansion of the generators of $\MQJac$ in $z$,
see e.g. \cite{vIOP}:
\begin{gather*}
\Theta(z) =  z \exp\left(-2\sum_{k\geq 1} G_{k}(\tau) \frac{z^{k}}{k!}\right) \\
\A(z) = \frac{1}{z} - 2 \sum_{k \geq 1} G_{k}(\tau) \frac{z^{k-1}}{(k-1)!} \\
\wp(z) = \frac{1}{z^2} + 2 \sum_{k \geq 4} G_{k}(\tau) \frac{z^{k-2}}{(k-2)!} 
\end{gather*}
Consider the operator that takes the formal derivative with respect ot $G_2$ factorwise,
\[ \left( \frac{d}{dG_2} \right)_z : \QMod((z)) \to \QMod((z)), \]
that is
for $f = \sum_{r} f_r(\tau) z^r$ with $f_r \in \QMod$, we let
\[ \left( \frac{d}{dG_2} \right)_z f = \sum_{r} \frac{d f_r}{d G_2} z^r. \]
Consider the decomposition of $\MQJac$ according to weight $k$ and index $m$,
\[ \MQJac = \bigoplus_{k,m} \MQJac_{k,m}. \]
Then the following is immediate from the expansions above:
\begin{lemma} \label{lemma:z^r coefficient}
The coefficient of $z^r$ 
of any series $f \in \MQJac_{k,m}$ is a quasi-modular form of weight $r+k$.
Moreover, we have
\begin{equation} \left( \frac{d}{dG_2} \right)_z f = \frac{d}{dG_2} f  - 2 z \frac{d}{dA} f -2 z^2 m f.  \label{fiberwise ddg2} \end{equation}
\end{lemma}

We prove the following partial converses:
\begin{lemma} \label{lemma:z expansion 1}
Let $f_i(q) \in \BC[[q]]$ be power series, such that every $z^r$-coefficient
of
\[ f(p,q) = \Theta^{2m}(p,q) \sum_{i=0}^{m} f_i(q) \wp(p,q)^{m-i}. \]
is a quasi-modular form of weight $z^{r+s}$. Then
every $f_i(q)$ is quasi-modular of weight $s+2i$.
\end{lemma}
\begin{proof}
We have $\Theta^{2m} \wp^{m-i} = z^{2i} + O(z^{2i+2})$
so we can write $f_i(q)$ as a linear combination of the $z^r$-coefficients of $f(p,q)$ with coefficients quasi-modular forms (of the correct weight).
\end{proof}

\begin{lemma} \label{lemma:z expansion 2}
Let $f_i(q) \in \BC[[q]]$ be power series, such that every $z^r$-coefficient
of
\[ f(p,q) = \Theta^{2m}(p,q) \wp'(p,q) \sum_{i=2}^{m} f_i(q) \wp(p,q)^{m-i}. \] 
is a quasi-modular form of weight $z^{r+s}$. Then
every $f_i(q)$ quasi-modular of weight $s+2i-3$.
\end{lemma}
\begin{proof}
Similarly.
\end{proof}

\subsection{Hecke operators} \label{subsec:Hecke operators}
Let $m \geq 1$ and recall that
the $m$-th Hecke operator acts on Jacobi forms $\phi(x,\tau)$ of weight $k$ and index $m$
by
\begin{equation} (T_{(k,m),\ell} f)(x,\tau) =
\ell^{k-1} \sum_{A = \binom{a\ b}{c\ d} \in \SL_2(\BZ) \backslash M_{\ell}} (c \tau + d)^{-k} e\left( m \ell \frac{-c x^2}{ c\tau + d } \right) f\left( \frac{ \ell x}{c \tau+d}, \frac{a \tau+b}{c \tau+d} \right) \label{Hecke operator} \end{equation}
where $A$ runs over a set of representatives of the $\SL_2(\BZ)$-left cosets of
the set 
\[ M_{\ell} = \left\{ \begin{pmatrix} a & b \\ c & d \end{pmatrix} \middle| a,b,c,d \in \BZ, \ ad-bc = \ell \right\}. \]
As shown in \cite[I.4]{EZ}, the action of $T_{(k,m),\ell}$
is well-defined (i.e. independent of a set of representatives) and defines an operator\footnote{We only require Hecke operators
for the full modular group, so we restrict to $\Gamma = \SL_2(\BZ)$ here, i.e. omit $\Gamma$ from the notation.
This section generalizes also to arbitrary congruence subgroups.}
\[
T_{(k,m),\ell} : \Jac_{k,m} \to \Jac_{k,m \ell}.
\]

Since the argument in \cite{EZ} only involves the compatibilities of the slash-operators of the Jacobi forms
the proof carries over identically to almost holomorphic weak Jacobi forms.
Hence using formula \eqref{Hecke operator} we also obtain a well-defined operator:
\[ T_{(k,m),\ell} : \AHJ_{k,m} \to \AHJ_{k,m \ell^2}, \quad F \mapsto T_{(k,m),\ell} F. \]

Transporting to quasi-Jacobi forms using the 'constant term' map of Lemma~\ref{lemma:constant term maps} 
we hence obtain a Hecke operator on quasi-Jacobi forms
\[
T_{(k,m),\ell} : \QJac_{k,m} \to \QJac_{k,m \ell},
\]
defined by the commutativity of the diagram
\[
\begin{tikzcd}
\QJ_{k,m} \ar{d}[swap]{T_{(k,m),\ell}} & \ar{l}[swap]{\cong} \AHJ_{k,m} \ar{d}{T_{(k,m),\ell}} \\
\QJ_{k,m \ell} & \ar{l}[swap]{\cong} \AHJ_{k,m \ell}
\end{tikzcd}.
\]

The Hecke operator on quasi-Jacobi forms satisfies the following:
\begin{prop} \label{prop:Hecke quasi Jacobi}
If $f = \sum_{n,r} c(n,r) q^n p^r$ is the Fourier-expansion of a quasi-Jacobi form of weight $k$ and index $m$, then
\begin{equation} \label{Hecke formula}
T_{(k,m),\ell} f =
\sum_{n,r} \left( \sum_{a|(\ell,n,r)} a^{k-1} c\left( \frac{\ell n}{a^2}, \frac{r}{a} \right) \right) q^n p^r.
\end{equation}
Moreover,
\begin{equation} \label{Hecke rel1}
\begin{gathered}
\frac{d}{dG_2} T_{k,\ell} f = \ell \, T_{k-2,m} \frac{d}{dG_2} f \\
\frac{d}{d \A} T_{k,\ell} f = \ell \, T_{k-1,m} \frac{d}{d\A} f.
\end{gathered}
\end{equation}
where we write 
$T_{k,\ell} := T_{(k,m),\ell}$
since $T_{(k,m),\ell}$ does not depend on $m$.
\end{prop}
\begin{proof}
For $\binom{a\ b}{c\ d} \in M_{\ell}$ we have the transformation properties:
\begin{align*}
\nu\left( \frac{a \tau+b}{c \tau+d} \right)
& =
\frac{1}{\ell} \nu(\tau) | c \tau + d |^2 \\
& = \frac{1}{\ell} \left[ (c \tau+d)^2 \nu(\tau) + \frac{c (c \tau+d)}{4 \pi i} \right] \\
\alpha\left( \frac{ \ell x}{c \tau+d}, \frac{a \tau+b}{c \tau+d} \right) 
& =
( c \tau + d) \cdot \alpha(x,\tau) - c x
\end{align*}
Consider the weight $k$ index $m$ almost holomorphic weak Jacobi form
\[ F = \sum_{i,j} f_{i,j} \nu^i \alpha^j \]
with $f_{0,0} = f$.
With $J = c \tau + d$ and $\tilde{c} = c / 4 \pi i$ we obtain
\begin{multline} \label{Hecke1}
(T_{(k,m),\ell} F)(x,\tau)\\
= \ell^{k-1} 
\sum_{A,r,s} J^{-k} e\left( m \ell \frac{-c x^2}{ c\tau + d } \right) f_{r,s}\left( \frac{ \ell x}{c \tau+d}, \frac{a \tau+b}{c \tau+d} \right) \left( \frac{J \cdot (J \nu + \tilde{c})}{\ell} \right)^r 
( J \alpha - c z)^s.
\end{multline}
We specialize $A$ now to run over the set of representatives of $\SL_2(\BZ) \backslash M_{\ell}$
given by
\[ \begin{pmatrix} a & b \\ 0 & d \end{pmatrix}, \quad \ell = a \cdot d, \quad b = 0, \ldots d-1. \]
Then \eqref{Hecke1} becomes:
\[ (T_{(k,m),\ell} F)(x,\tau)
= \ell^{k-1}
\sum_{r,s \geq 0} \nu^r \alpha^s \left[ \frac{1}{\ell^r} \sum_{\substack{\ell = a \cdot d \\ b=0,\ldots, d-1}}  d^{-k+2r+s} f_{r,s}\left( az, \frac{a \tau+b}{d} \right) \right].
\]

Taking the $\nu^0 \alpha^0$ coefficient and inserting $f = \sum_{n,r} c(n,r) q^n p^r$ yields
\begin{align*}
T_{(k,m),\ell} f
& = 
\mathrm{Coeff}_{\nu^0 \alpha^0}\left( T_{(k,m),\ell} F \right)\\
& = 
\ell^{k-1} \sum_{\ell = a \cdot d} d^{-k} \sum_{b=0}^{d-1} f(az,  (a \tau+b)/d) \\
& = \sum_{\ell = a \cdot d} a^{k-1} \sum_{\substack{n,r \\ n \equiv 0 \text{ mod } d}} c(n,r) p^{ar} q^{na/d}.
\end{align*}
This gives the first claim.
The compatibility with the anomaly operators follows from
\begin{gather*}
\frac{d}{dG_2} T_{k,\ell} f = \mathrm{Coeff}_{\nu^1 \alpha^0}\left( T_{(k,m),\ell} F \right) 
= \ell \sum_{\ell = a \cdot d} a^{k-3} \sum_{\substack{n,r \\  n \equiv 0 \text{ mod } d}} c'(n,r) p^{ar} q^{n a/d} \\ 
\frac{d}{d \A} T_{k,\ell} f = \mathrm{Coeff}_{\nu^0 \alpha^1}\left( T_{(k,m),\ell} F \right)
= \ell \sum_{\ell = a \cdot d} a^{k-2} \sum_{\substack{n,r \\  n \equiv 0 \text{ mod } d}} c''(n,r) p^{ar} q^{n a/d}
\end{gather*}
where $c', c''$ are the Fourier coefficients of $f_{1,0}$ and $f_{0,1}$ respectively.
%
\end{proof}

By a straightforward computation using \eqref{Hecke formula}
one finds that for $f \in \QJac_{k,m}$:
\begin{equation} \label{Hecke rel2}
\begin{gathered}
T_{k+2,\ell} D_{\tau} f = \ell D_{\tau} T_{k,\ell} f \\
T_{k+1,\ell} D_z f = D_z T_{k,\ell} f
\end{gathered}
\end{equation}
Then equations \eqref{Hecke rel1} and \eqref{Hecke rel2} are compatible with the commutation relations \eqref{eq:comm relations 1}.

\subsection{Wrong-weight Hecke operators} \label{subsec:Hecke operators wrong}
For a formal power series $f = \sum_{d,r} c(d,r) q^d p^r$ 
we can define formally the $\ell$-th Hecke operator of weight $k$ by
\begin{equation} T_{k,\ell} f =
\sum_{n,r} \left( \sum_{a|(\ell,n,r)} a^{k-1} c\left( \frac{\ell n}{a^2}, \frac{r}{a} \right) \right) q^n p^r. \label{Hecke formal} \end{equation}

In Proposition~\ref{prop:Hecke quasi Jacobi} we have seen that
$T_{k,\ell}$ defines an operator
\[ T_{k,\ell} : \QJac_{k,m} \to \QJac_{k, m \ell}. \]
More generally, we can ask what happens if we apply $T_{k,\ell}$ to quasi-Jacobi forms $f$ of a weight $k'$ different from $k$?
This is answered by the following proposition:

Consider the congruence subgroup
\[ \Gamma_0(\ell) = \left\{ \begin{pmatrix} a & b \\ c & d \end{pmatrix} \in \SL_2(\BZ) \middle| c \equiv 0 \text{ mod } \ell \right\}. \]

\begin{prop} \label{prop:Hecke wrong weight}
For any $k,k',m$ the $\ell$-th formal Hecke operator defines a morphism
\[ T_{k,\ell} : \QJac_{k',m} \to \QJac_{k', m \ell}( \Gamma_0(\ell) ). \]
Moreover, for any $f \in \QJac_{k',m}(\SL_2(\BZ))$ we have
\begin{equation} \label{hecke relations wrong weight}
\begin{gathered}
\frac{d}{dG_2} T_{k,\ell} f = \ell \, T_{k-2,m} \frac{d}{dG_2} f \\
\frac{d}{d \A} T_{k,\ell} f = \ell \, T_{k-1,m} \frac{d}{d\A} f.
\end{gathered}
\end{equation}
\end{prop}

For the proof we will decompose the 'wrong-weight Hecke operator'
into ordinary Hecke operators and the scaling operators
$B_N$ for $N \geq 1$ defined on functions $f: \BC \times \BH \to \BC$ by
\[
(B_N f)(x,\tau) = f(Nx, N \tau).
\]

\begin{lemma} \label{lemma:Hecke shift}
If $f \in \QJac_{k,m}$ then $B_N f \in \QJac_{k, mN}(\Gamma_0(N))$,
and moreover
\begin{gather*}
\frac{d}{dG_2} B_N f = \frac{1}{N} \, B_N \frac{d}{dG_2} f \\
\frac{d}{d \A} B_N f = \frac{1}{N} \, B_N \frac{d}{d\A} f.
\end{gather*}
\end{lemma}

\begin{proof}
Let $F(x,\tau)$ be a almost-holomorphic weak Jacobi form of weight $k$ and index $m$.
Set 
\[ \hat{F}(x,\tau) = (B_N F)(x,\tau)= F(Nx, N \tau).\]
Then for $\binom{a\ b}{c\ d} \in \Gamma_0(N)$ and with $c = c' \cdot N$ we have
\begin{align*}
\hat{F}\left( \frac{x}{c \tau + d}, \frac{a \tau + b}{c \tau + d} \right)
& = F\left( \frac{Nx}{c \tau + d}, \frac{a N \tau + N b}{c \tau + d} \right) \\
& = F\left( \frac{Nx}{c' (N \tau) + d}, \frac{a N \tau + N b}{c' (N \tau) + d} \right) \\
& = (c' (N \tau) + d)^k e\left( \frac{m c' (N x)^2}{c' (N\tau) + d} \right) F(Nx, N \tau) \\
& = (c \tau + d)^k e\left( \frac{ (mN) c x^2}{c \tau + d} \right) \hat{F}(x,\tau)
\end{align*}
where we have used that $\binom{a\ Nb}{c'\ d} \in \mathrm{SL}_2(\BZ)$. 
Similarly, one proves that
\[ \hat{F}(x + \lambda \tau, \tau) = e(-m N ( \lambda^2 \tau + 2 \lambda x)) \hat{f}(x,\tau). \]
This shows that $\hat{F} \in \AHJ_{k, mN}(\Gamma_0(N))$,
and by taking the constant coefficient also $B_N f \in \QJac_{k, mN}(\Gamma_0(N))$.
To show the compatibility with the anomaly operators write
\[ F(x,\tau) = \sum_{i,j} f_{i,j} \nu^i \alpha^j. \]
Since we have $\nu(N \tau) = \nu(\tau)/N$, we get
\[
B_N F(x,\tau) = \sum_{i,j} \frac{1}{N^i} f_{i,j}(Nx, N \tau) \nu^i \alpha^j.
\]
Hence if $f=f_{0,0}$ we get
\[
\frac{d}{dG_2} B_N f = \mathrm{Coeff}_{\nu^1 \alpha^0}( B_N F(x,\tau) ) = \frac{1}{N} f_{1,0}(Nx, N_{\tau}) = \frac{1}{N} B_N \frac{d}{dG_2} f.
\]
The case for $\frac{d}{d \A}$ is similar.
\end{proof}

\begin{proof}[Proof of Proposition~\ref{prop:Hecke wrong weight}]
We follow ideas of \cite[Lemma 12]{BB}.
Given a power series
$f = \sum_{d,r} c(d,r) q^d p^r$ define the formal operator
\begin{gather*}
U_b f = \sum_{n,r} c(b \cdot n, r) q^n p^r.
\end{gather*}
A direct calculation starting from \eqref{Hecke formal} shows that
\[ T_{k,\ell} = \sum_{a \cdot b=\ell} a^{k-1} B_a U_b. \]

Recall the M\"obius function
\[ 
\mu(n) = 
\begin{cases} 
(-1)^g & \text{ if } n = p_1 \cdots p_g \text{ for distinct primes } p_i \\
0 & \text{ else }.
\end{cases}
\]
which satisfies $\sum_{d|n, d>0} \mu(d) = \delta_{n1}$.
For $s \in \BZ$ let $\Id_{s}$ be the function $\Id_{s}(a) = a^{s}$.
For functions $g,h$ define the Dirichlet convolution $(g \ast h)(\ell) = \sum_{\ell = a \cdot b} g(a) h(b)$ 
and the pointwise product $(g \cdot h)(a) = g(a) h(a)$.
Both of these are associative operations. We then have
\[ (\mu \cdot \Id_{k'-1}) \ast \Id_{k'-1}(a) = \delta_{a1} \]
and thus
\[ \Big( \Id_{k-1} \ast (\mu \cdot \Id_{k'-1}) \ast \Id_{k'-1}\Big) (a) = \Id_{k-1}. \]
After setting
\[ c_{k,k'}(e) = (\Id_{k-1} \ast (\mu \cdot \Id_{k'-1}))(e) \]
this yields
\begin{align} 
T_{k,\ell} & = \sum_{a \cdot b=\ell} \Big( \Id_{k-1} \ast (\mu \cdot \Id_{k'-1}) \ast \Id_{k'-1}\Big) (a) B_a U_b \notag \\
& = \sum_{a \cdot b=\ell} \sum_{e|a} c_{k,k'}(e) \left( \frac{a}{e} \right)^{k'-1} B_a U_b \notag \\
& = \sum_{e \cdot d=\ell} c_{k,k'}(e) B_{e} \sum_{d = b \cdot b'} (b')^{k'-1} B_{b'} U_b \notag \\
 & = \sum_{e \cdot d=\ell} c_{k,k'}(e) B_{e} T_{k',d} \label{hecke decomp} \\
%
\end{align}
where we used $B_{a} = B_{e} \cdot B_{a/e}$

Given $f \in \QJac_{k',m}$ we have $T_{k',d} f \in \QJac_{k', m d}$ by Proposition~\ref{prop:Hecke quasi Jacobi},
and hence 
\[ B_e T_{k',d} f \in \QJac_{k', m d} \in \QJac_{k', m d e}(\Gamma_0(e)) \]
by Lemma~\ref{lemma:Hecke shift}. 
Since for $e|\ell$ we have
\[ \QJac(\Gamma_0(e)) \subset \QJac(\Gamma_0(\ell)) \]
we obtain that
\[ T_{k,\ell} f = \sum_{e \cdot d=\ell} c_{k,k'}(e) B_{e} T_{k',d} f \in \QJac(\Gamma_0(\ell)). \]

For the second part, observe that
\[
c_{k,k'}(e) = \sum_{a \cdot b=e} a^{k-1} b^{k'-1} \mu(b) = e^2 c_{k-2,k'-2}(e).
\]
Hence by the second parts of Proposition~\ref{prop:Hecke quasi Jacobi} and Lemma~\ref{lemma:Hecke shift} we have
\begin{align*}
\frac{d}{dG_2} T_{k,\ell} f 
& = \sum_{e \cdot d=\ell} c_{k,k'}(e) \frac{d}{e} B_{e} T_{k',d} \frac{d}{d G_2} f \\
& = \ell \sum_{e \cdot d=\ell} c_{k-2,k'-2}(e) B_{e} T_{k',d} \frac{d}{d G_2} f \\
& = T_{k-2,\ell} \frac{d}{dG_2} f.
\end{align*}
\end{proof}

\begin{example}
Recall that $\Mod_2(\Gamma_0(2))$ is $1$-dimensional and is generated by
\[ F_2(\tau) = 1 + 24 \sum_{\substack{d|n \\ d \text{odd}}} d q^n. \]
Hence $\QMod_2(\Gamma_0(2))$ has the basis given by $F_2, G_2$.
One computes that
\begin{align*}
T_{k,2} G_2(\tau) 
& = 2^{k-1} B_2 G_2 + U_2 G_2 \\
& = 2^{k-1} \left( -\frac{1}{48} F_2 + \frac{1}{2} G_2 \right) + \left( \frac{1}{24} F_2 + 2 G_2 \right).
\end{align*}
Hence as predicted by Proposition~\ref{prop:Hecke wrong weight} we get:
\[
\frac{d}{d G_2} T_{k,2} G_2(\tau) = 2(1+2^{k-3}) = 2 \cdot T_{k-2,2}(1).
\]
\qed
\end{example}

In applications below we will consider quasi-Jacobi forms with a pole at $\tau = i \infty$, i.e.
which are of the form 
\[ f(x,\tau) = \frac{\phi(x,\tau)}{\Delta(\tau)^r} \]
for a quasi-Jacobi form $\phi$ and some $m \geq 1$.
Since the argument used to prove Proposition~\ref{prop:Hecke wrong weight}
also works when there are poles,
the results of Proposition~\ref{prop:Hecke wrong weight}
remain valid for these quasi-Jacobi forms as well.
The only modification concerns the order of poles:

\begin{prop} \label{prop:Hecke wrong weight with poles}
For any $k,k',m$ the $\ell$-th formal Hecke operator acts by
\[ T_{k,\ell} : \frac{1}{\Delta(\tau)} \QJac_{k'+12,m} \to \frac{1}{\Delta(\tau)^{\ell}} \QJac_{k'+12 \ell, m \ell}( \Gamma_0(\ell) ). \]
The relations \eqref{hecke relations wrong weight} hold identically.
\end{prop}
\begin{proof}
If $f(x,\tau) = \frac{\phi(x,\tau)}{\Delta(\tau)}$ is a weight $k$ index $m$ quasi-Jacobi for group $\SL_2(\BZ)$,
then the 'correct weight' Hecke transform $T_{k,\ell} f$ is also quasi-Jacobi for the full group $\SL_2(\BZ)$.
The poles of $T_{k,\ell} f$ are located at the single cusp $\tau = i \infty$, and here \eqref{Hecke formal} shows that the pole order is increased by $\ell$.
We hence obtain that $\Delta(\tau)^{\ell} T_{k,\ell} f$ is holomorphic quasi-Jacobi, i.e. lies in $\QJac_{k+12 \ell, m \ell}$.
Hence the claim holds if $k=k'$.
In the general case we use again the decomposition \eqref{hecke decomp},
the fact that $B_N$ is a ring homomorphism, and that for any $N \geq 1$ we have
(see e.g. \cite[Prop.17(a)]{Koblitz})
\[ B_N\left( \frac{1}{\Delta(\tau)} \right) \in \frac{1}{\Delta(\tau)^N} \Mod_{12(N-1)}(\Gamma_0(N)). \qedhere \]
\end{proof}

\subsection{Index 0 meromorphic Jacobi forms}
\label{subsec:index 0 meromorphic JAcobi forms}
Consider the algebra of index $0$ Jacobi forms,
\[
\MQJac_0 := \bigoplus_{k \geq 0} \MQJac_{k,0} = \BC[\A, G_2, \wp, \wp', G_4].
\]
The algebra $\MQJac_0$ is precisely the ring of index $0$ meromorphic Jacobi forms
with poles only at lattice points $x=a \tau +b$ for $a,b \in \BZ$, see \cite{Lib}.

Consider once more the Jacobi theta function
\[
\Theta(z) 
=  (p^{1/2}-p^{-1/2})\prod_{m\geq 1} \frac{(1-pq^m)(1-p^{-1}q^m)}{(1-q^m)^{2}}.
\]
which we view in this section as a function of $z=2\pi i x$ (and drop $\tau$ from notation).
Define functions $\A_n(z,\tau)$ for all $n \in \BZ$ by the expansion:
\begin{equation} \label{function}
\frac{\Theta(z+w)}{\Theta(z) \Theta(w)}
=
\sum_{n \geq 0} \frac{\A_n(z,\tau)}{n!} w^{n-1}.
\end{equation}
In particular $\A_0 = 1$, $\A_1 = \A$.
The function $\frac{\Theta(z+w)}{\Theta(z) \Theta(w)}$
is a meromorphic Jacobi
of lattice index $\begin{pmatrix} 0 & 1/2 \\ 1/2 & 0 \end{pmatrix}$,
which leads to the proof of the following:

\begin{thm}[{\cite{Zagier,Lib}}] \label{thm:2132}
(a) For all $n$ we have $\A_n \in \MQJac_{0,n}$ and
\[ \frac{d}{d G_2} \A_n = 0, \quad \frac{d}{d \A} \A_n = n \A_{n-1}. \]
(b) For all $n \geq 0$, we have the expansion
\[
\A_n(z,\tau) = B_n + \delta_{n,1}\frac{1}{2} \frac{p^{1/2}+p^{-1/2}}{p^{1/2} - p^{-1/2}} - n \sum_{k,d \geq 1} d^{n-1} (p^{k} + (-1)^n p^{-k}) q^{kd} .
\]
where the Bernoulli numbers $B_n$ are defined by $\frac{1}{2} \coth(z/2) = \sum_{n \geq 0} (B_n/n!) z^{n-1}$.
\end{thm}
\begin{proof}
The first part follows immediately from the transformation properties given in the theorem of \cite[Sec.3]{Zagier},
but see also \cite{Lib} for why the $\A_n$ lie in $\MQJac$, and \cite[Lemata 5 and 6]{Ob_SerreDerivative} for the holomorphic anomaly equation
(the functions $\A_n$ were called $J_n$ in {\em loc.cit.}).
Part (b) follows from the following expansion proven in \cite[Sec.3]{Zagier}:
\[ \frac{\Theta(z+w)}{\Theta(z) \Theta(w)}
=
\frac{1}{2} \left( \coth \frac{w}{2}  + \coth \frac{z}{2} \right)
-2 \sum_{n=1}^{\infty} \left( \sum_{d|n} \sinh( dw + \frac{n}{d} z ) \right) q^n. \qedhere \] 
\end{proof}

\begin{rmk}[Historical remark] The function \eqref{function}
already centrally appeared in work of Eisenstein on elliptic functions in the 1850's,
see
\cite{Weil} for a historical account,
\end{rmk}

\section{Cohomology and monodromy of the Hilbert scheme} \label{sec:cohomology}
\subsection{Overview}
Let $S$ be a K3 surface and let $S^{[n]}$ be the Hilbert scheme of $n$ points on $S$.
There are two basic structures on the cohomology of the Hilbert scheme.
The first is the Nakajima Heisenberg action (Section~\ref{subsec:Nakajima operators})
which gives a natural additive basis of the cohomology
and allow us to identify the curve classes on the Hilbert scheme (Section~\ref{sec:curve classes}).
The second is the Looijenga-Lunts-Verbitsky (LLV) Lie algebra (Section~\ref{subsec:LLV algebra}) which will appear in the statement of the holomorphic anomaly equations.
In Section~\ref{sec:Weight grading} we use the LLV algebra to define several gradings on the cohomology.
In Section~\ref{subsec:monodromy} we recall work of Markman on how the LLV algebra controls the monodromy.
For us important are two particular monodromy operators, which lead to the elliptic transformation property of the generating series:
These are discussed in detail in
Section~\ref{monodromy:involution} and~\ref{monodromy:shift}.
In particular, we describe how they act on the Nakajima basis.
%
%

\subsection{Nakajima operators} \label{subsec:Nakajima operators}
We follow the work \cite{Nak}, see also \cite{Groj}.
For any~$n,k \in \BN$, consider the closed subscheme:
$$
S^{[n,n+k]} = \left\{(I \supset I') \,\middle|\, I/I' \text{ is supported at a single }x \in S \right\} \subset S^{[n]} \times S^{[n+k]}
$$
endowed with projection maps
\begin{equation}
\label{eqn:diagram zk}
\begin{tikzcd}
& S^{[n,n+k]} \ar[swap]{dl}{p_-} \ar{d}{p_S} \ar{dr}{p_+} & \\
S^{[n]} & S & S^{[n+k]}
\end{tikzcd}
\end{equation}
which remember $I$, $x$, $I'$, respectively. 
For $\alpha \in H^{\ast}(S)$ and $k>0$ we define
the $k$-th Nakajima operator
by letting $S^{[n, n+k]}$ act as a correspondence, that is we define:
\[ \Fq_k(\alpha) : H^{\ast}(S^{[n]}) \to H^{\ast}(S^{[n+k]}) \]
\[ \Fq_k(\alpha) \gamma = p_{+ \ast}( p_{-}^{\ast}(\gamma) \cdot p_S^{\ast}(\alpha) ). \]
Similarly, we can go the other way around and define
$\Fq_{-k}(\alpha) : H^{\ast}(S^{[n+k]}) \to H^{\ast}(S^{[n]})$ by
\[
\Fq_{-k}(\alpha) \gamma = (-1)^k p_{- \ast}( p_{+}^{\ast}(\gamma) \cdot p_S^{\ast}(\alpha) ).
\]
We also set $\fq_0(\gamma) = 0$ for all $\gamma$.

Consider the direct sum
\[ H^{\ast}(\Hilb) = \bigoplus_{n \geq 0} H^{\ast}(S^{[n]}). \]
Because the correspondences above are defined for all $n$, we obtain operators
\[ \Fq_i(\alpha) : H^{\ast}(\Hilb) \to H^{\ast}(\Hilb). \]
By the main result of \cite{Nak} we have the commutation relations of the Heisenberg algebra
\begin{equation}
\label{eqn:heis op}
[\fq_k(\alpha), \fq_l(\beta)] 
=  k ( \alpha, \beta ) \Id_{\Hilb}.
\end{equation}
Moreover, $H^{\ast}(\Hilb)$ is generated by the operators $\Fq_k(\alpha)$ for $k>0$ from the vacuum vector
\[ \1 \in H^{\ast}(S^{[0]}) = \BQ. \]
In particular, the set of classes
\[ \Fq_{\lambda_1}(\gamma_{i_1}) \cdots \Fq_{\lambda_{\ell(\lambda)}}(\gamma_{i_{\ell(\lambda)}}) \1, \]
where $\lambda = (\lambda_j, \gamma_{i_j})$ runs over all partitions of size $n$
weighted by cohomology classes from a fixed basis $\{ \gamma_i \}_{i=1}^{24}$ of $H^{\ast}(S)$,
forms a basis of $H^{\ast}(S^{[n]},\BQ)$.

For homogeneous $\alpha_i \in H^{\ast}(S)$, the degree of a Nakajima cycle is
\begin{equation} \deg\big(   \Fq_{k_1}(\alpha_1) \cdots \Fq_{k_{\ell}}(\alpha_{\ell}) \1 \big) = n-\ell+\sum_{i} \deg(\alpha_i), \label{deg Nakajima cycle} \end{equation}
The {\em length} of a Nakajima cycle is defined to be the number of Nakajima factors:
\begin{equation} l\left( \Fq_{k_1}(\alpha_1) \cdots \Fq_{k_{\ell}}(\alpha_{\ell}) \1 \right) = \ell. \label{length Nakajima} \end{equation}

\subsection{Curve classes} \label{sec:curve classes}
For $n \geq 2$ the fiber of the Hilbert-Chow morphism $S^{[n]} \to \mathrm{Sym}^n(S)$ over a generic point in the discriminant
is isomorphic to $\p^1$ and has (co)homology class
\[ A = \Fq_2(\pt) \Fq_1(\pt)^{n-2} \1 \in H_2(S^{[n]},\BZ), \]
where $\pt \in H^4(S,\BZ)$ is the class of a point.
Similarly, given a class $\beta \in H_2(S, \BZ)$ we have an associated class
on the Hilbert scheme given by 
\[ \beta_{[n]} := \Fq_1(\beta) \Fq_1(\pt)^{n-1} \1 \in H_2(S^{[n]},\BZ). \]
If $\beta$ is the class of a curve $C \subset S$, then $\beta_{[n]}$ is the class of
the curve parametrizing subschemes consisting of $n-1$ distinct fixed points away from $C$ and a single free point on $C$.

By Nakajima's theorem \cite{Nak} (discussed in the last section), we have an isomorphism:
\begin{equation} H_2(S^{[n]},\BZ) \cong H^2(S,\BZ) \oplus \BZ A, \quad \beta_{[n]} + rA \mapsfrom (\beta,r).\label{H_2 iso} \end{equation}
Usually we simply write $\beta+rA$ for the class associated to $(\beta,r)$ on the Hilbert scheme.
If $n \leq 1$, we set $A=0$ and always assume that $r=0$; in case $n=0$ we also assume that $\beta=0$.

\subsection{The Looijenga-Lunts-Verbitsky algebra} \label{subsec:LLV algebra}
Let $X$ be an (irreducible) hyperk\"ahler variety of dimension $2n$.
The lattice $H^2(X,\BZ)$ is equipped with an integral and non-degenerate quadratic form,
called the Beauville-Bogomolov-Fujiki form \cite{Fuji}.
We will also view $H^{\ast}(X,\BZ)$ as a lattice using the Poincar\'e pairing.
Both pairings are extended to the $\BC$-valued cohomology groups by linearity.

The Looijenga-Lunts-Verbitsky Lie algebra of $X$ is defined as follows (see \cite{LL,V}).
For any $a \in H^2(X,\BQ)$ such that $(a,a) \neq 0$, consider the operator on cohomology which takes the cup product with $a$,
\[ e_a : H^{\ast}(X,\BQ) \to H^{\ast}(X,\BQ), x \mapsto a \cup x \]
Let $h$ be the Lefschetz grading operator which acts on $H^{2i}(X,\BZ)$ by multiplication by $(i-n)$.
Then there exists a unique operator 
\[ f_a : H^{\ast}(X,\BZ) \to H^{\ast}(X,\BZ) \]
such that the $\mathfrak{sl}_2$ commutation relations are satisfied:
\[ [e_a, f_a]=h, \quad [h,e_a] = e_a, \quad [h,f_a] = -f_a. \]
The LLV Lie algebra $\Fg(X)$ is defined as the
Lie subalgebra of $\End H^{\ast}(X,\BQ)$ generated by $e_a, f_a,h$ for all $a \in H^2(X,\BQ)$ as above.
By the central result of \cite{V} one has
\[ \Fg(X) = \so( H^2(X,\BQ) \oplus U_{\BQ} ) \]
where $U = \binom{0\ 1}{1\ 0}$ is the hyperbolic plane.

The degree zero part of $\Fg(X)$ decomposes as
\[ \Fg(X)_0 = \mathfrak{so}( H^2(X,\BQ)) \oplus \BQ h. \]
The summand $\mathfrak{so}( H^2(X,\BQ))$ is also called the reduced LLV algebra. 
Base changing to $\BC$ and integrating this yields the {\em LLV representation}
\begin{equation} \rho_{\mathrm{LLV}} : \SO(H^2(X,\BC)) \to \mathrm{GL}(H^{\ast}(X,\BC)). \label{LLV rep} \end{equation}
The LLV representation acts by degree-preserving orthogonal ring isomorphisms \cite[Prop. 4.4(ii)]{LL},
where orthogonal means with respect to the Poincar\'e pairing.

The Hilbert scheme of points $S^{[n]}$ on a K3 surface are irreducible hyperk\"ahler varieties \cite{Beauville}.
The LLV algebra was described here explicitly in the Nakajima basis in \cite{OLLV}.
We recall the explicit formulas, using the conventions of \cite{NOY}.
First recall the isomorphism
\begin{equation} V = H^2(S^{[n]}) \cong H^2(S) \oplus \BQ \cdot \delta \label{V description} \end{equation}
which can be obtained by dualizing \eqref{H_2 iso}.
In particular, $\delta$ is $-\frac{1}{2}$ times the class of the locus of non-reduced subschemes and satisfies $\delta \cdot A = 1$.
Moreover, for $\alpha \in H^2(S,\BQ)$ the associated divisor on the Hilbert scheme is $\frac{1}{(n-1)!}\Fq_1(\alpha) \Fq_1(1)^{n-1} \1$.
The Beauville--Bogomolov--Fujiki form is then the form on $V$
which extends the intersection form on $H^2(S)$ and satisfies
$$
(\delta, \delta) = 2 - 2n, \quad (\delta, A^1(S)) = 0.
$$
The LLV algebra is given by 
\[ \Fg(S^{[n]}) = \wedge^2 ( V \oplus U_{\BQ} ) \]
where the Lie bracket is defined for all $a,b,c,d \in V \oplus U_\BQ$ by
\[ [a \wedge b, c \wedge d] = (a,d) b \wedge c - (a,c) b \wedge d - (b,d) a \wedge c + (b,c) a \wedge d. \]

Consider for all $\alpha \in H^2(S,\BQ)$ the following operators:
\begin{gather}
e_{\alpha} = -\sum_{n > 0} \Fq_{n} \Fq_{-n} ( \Delta_{\ast} \alpha) \nonumber \\
\label{LLV_operators}
\begin{gathered}
e_{\delta} = -\frac{1}{6} \sum_{i+j+k=0} : \Fq_i \Fq_j \Fq_k ( \Delta_{123} ): \\
\widetilde{f}_{\alpha} = -\sum_{n > 0} \frac{1}{n^2} \Fq_{n} \Fq_{-n}( \alpha_1 + \alpha_2 )
\end{gathered} \\
\nonumber
\widetilde{f}_{\delta}
= -\frac{1}{6} \sum_{i+j+k=0} :\Fq_i \Fq_j \Fq_k \left( \frac{1}{k^2}  \Delta_{12} + \frac{1}{j^2} \Delta_{13} + \frac{1}{i^2} \Delta_{23} + \frac{2}{j k} c_1 + \frac{2}{i k} c_2 + \frac{2}{i j} c_3 \right): \,.
\end{gather}
Here $: - :$ is the normal ordered product defined by
$$: \Fq_{i_1} ... \Fq_{i_k}\!: \, = \, \Fq_{i_{\sigma(1)}} ... \Fq_{i_{\sigma(k)}}$$
where $\sigma$ is any permutation such that $i_{\sigma(1)} \geq ... \geq i_{\sigma(k)}$.
We define operators $e_{\alpha}$ and~$\widetilde{f}_\alpha$ for general $\alpha \in V$ by linearity in $\alpha$.
By \cite{Lehn} we have that $e_\alpha$ is precisely the operator of cup product with $\alpha$.
By \cite{OLLV}, if $(\alpha,\alpha) \neq 0$, the multiple~$\widetilde{f}_\alpha / (\alpha,\alpha)$ acts on cohomology as the Lefschetz dual of $e_\alpha$. 
Then, as shown in \cite{OLLV} the assignment
\begin{equation}
\label{eqn:action}
\begin{gathered}
\act: \Fg(S^{[n]}) \rightarrow \End H^{\ast}(S^{[n]}) \\
\forall \alpha \in V: \quad \act(e \wedge \alpha) = e_\alpha, \quad \act(\alpha \wedge f) = \widetilde{f}_\alpha
\end{gathered}
\end{equation}
induces a Lie algebra homomorphism, which is precisely the action of the LLV algebra.
The element~\mbox{$e \wedge f$} acts by
the Lefschetz grading operator
\begin{equation} 
\label{eqn:def h} 
h = \act(e \wedge f) = \sum_{k > 0} \frac{1}{k} \Fq_{k} \Fq_{-k}( \pt_2 - \pt_1 ).
\end{equation}

\subsection{Weight grading} \label{sec:Weight grading}
With the notation of the previous section,
consider vectors $W, F \in H^2(S,\BZ)$ which span a hyperbolic lattice,
that is which have intersection form $\binom{0\ 1}{1\ 0}$.
We associate three operators on $H^{\ast}(S^{[n]})$ to this pair:
\begin{enumerate}
\item[(i)] The Lefschetz dual operator (which will appear in the holomorphic anomaly equation for $\frac{d}{dG_2}$),
\[ U = \tilde{f}_F = \act(F \wedge f) = -\sum_{n > 0} \frac{1}{n^2} \Fq_n \Fq_{-n}( F_1 + F_2 ), \]
\item[(ii)] For any $\alpha \in V$ with $\alpha \perp \{ W, F \}$,
the degree-preserving operator 
\[ T_{\alpha} = [e_{\alpha}, U] = \act( \alpha \wedge F). \]
For the class $\delta \in V$, we have explicitly
\begin{equation} T_{\delta} = \frac{1}{2} \sum_{i+j+k=0} \frac{1}{i} : \Fq_i \Fq_j \Fq_k\big( (F_1 + F_2) \Delta_{23} \big) : . \label{T delta} \end{equation}
%
\item[(iii)] The weight grading operator
\begin{equation} \label{explicit formula}
\Wt = [e_{W}, U] = \act( e \wedge f + W \wedge F )
= \sum_{k > 0} \frac{1}{k} \Fq_{k} \Fq_{-k}( \pt_2 - \pt_1 + W_2 F_1 - W_1 F_2 )
\end{equation}
\end{enumerate}

The action of $x=e \wedge f + W \wedge F$ on $H^2(X,\BQ) \oplus U_{\BQ})$ is semisimple, so $x$ is a semisimple element of the LLV algebra.
Hence $H^{\ast}(S^{[n]})$ decomposes into eigenspaces under $\Wt$.
We can describe the eigenspaces quite explicitly:
We define a weight grading on $H^{\ast}(S)$ by
\[ \wt(\alpha)
=
\begin{cases}
1 & \text{ if } \alpha \in \{ W, \pt \} \\
-1 & \text{ if } \alpha \in \{ F, 1 \} \\
0 & \text{ if } \alpha \in \{ F, W, 1, \pt \}^{\perp}
\end{cases}
\]
This induces a grading of $H^{\ast}(S^{[n]})$ by setting
\begin{equation} \label{wt grading}
\wt(\gamma) = \sum_i \wt(\alpha_i) \quad \text{ for all } \quad \gamma=\prod_{i} \Fq_{k_i}(\alpha_i) \1,
\end{equation}
such that all $\alpha_i$ are $\wt$-homogeneous.
By the explicit formula \eqref{explicit formula}, a direct check shows that
\[ \Wt(\gamma) = \wt(\gamma) \gamma \]
for a $\wt$-homogeneous element $\gamma \in H^{\ast}(S^{[n]})$.

\begin{lemma} \label{lemma:T action}
The action of $\Wt$ on $H^{\ast}(S^{[n]})$ is semi-simple with 
eigenspace decomposition
\[ H^{\ast}(S^{[n]}) = \bigoplus_{\substack{d=-n \\ d \in \BZ}}^{n} V_d, \quad \Wt|_{V_d} = d \cdot \id_{V_d}. \]
The operators $T_{\alpha}$ (for $\alpha \perp \{ W, F\}$) and $U$ act with respect to this grading with weight $-1$ and $-2$ respectively, that is
\[ T_{\alpha} : V_d \to V_{d-1}, \quad U : V_d \to V_{d-2}. \]
\end{lemma}
\begin{proof}
The first claim follows since $\wt(\gamma)$ takes values in $\{ -n, \ldots, n \}$. The second claim follows from
\[ [\Wt,T_{\alpha}] = \act( [e \wedge f + W \wedge F, \alpha \wedge F] )
= \act( F \wedge \alpha ) = -T_{\alpha}. \]
\[ [\Wt,U] = \act( [e \wedge f + W \wedge F, F \wedge f] ) = \act( f \wedge F - F \wedge f ) = -2 U. \qedhere \]
\end{proof}

We have the following weight computation for the class
\[ U \in H^{\ast}(S^{[n]} \times S^{[n]}) \]
associated to the operator $U$ according to the conventions of Section~\ref{subsec:Conventions}.

\begin{lemma} \label{lemma: weight of U}
Consider a K\"unneth decomposition $U = \sum_{i} a_i \otimes b_i \in H^{\ast}(S^{[n]})^{\otimes 2}$ with $a_i, b_i$ homogeneous with respect to $\wt$.
Then for all $i$ we have
\[ \wt(a_i) + \wt(b_i) = -2. \]
\end{lemma}
\begin{proof}
This follows from
\begin{align*}
(\id \otimes \Wt + \Wt \otimes \id)(U) & = \Wt \circ U + U \circ \Wt^{t} \\
& = \Wt \circ U - U \circ \Wt \\
& = [\Wt, U] \\
& = -2 U.
\end{align*}
\end{proof}

The weight grading also interacts nicely with cup product:
\begin{lemma} \label{lemma:wt multiplicative}
The product $\gamma_1 \cdots \gamma_k$ 
of any $\wt$-homogeneous classes $\gamma_i$
is again $\wt$-homogeneous, and has weight
\[ \wt(\gamma_1 \cdots \gamma_k) = (k-1) n + \sum_i \wt(\gamma_i). \]
\end{lemma}
\begin{proof}
The grading operator $\tilde{h} = h+n \id$ is multiplicative, i.e. $\tilde{h}(xy) = \tilde{h}(x) y + x \tilde{h}(y)$.
Moreover, since the LLV representation \eqref{LLV rep} acts by ring isomorphisms,
\[ h_{WF} := \act(W \wedge F) = \frac{d}{dt}|_{t=0} \rho_{\mathrm{LLV}}(e^{t(W \wedge F)}) \]
is multiplicative. Hence
$\widetilde{\Wt} := \Wt + n \id = \tilde{h} + h_{WF}$
is multiplicative.
If we use this to compute $\Wt(\gamma_1 \cdots \gamma_k)$, we obtain the claim.
\end{proof}

\begin{rmk}
For $\gamma \in H^{\ast}(S^{[n]})$,
the modified degree function $\underline{\deg}(\gamma)$ of \cite[Sec.2.6]{ObMC} is related to the weight $\wt(\gamma)$ defined above
by $\underline{\deg}(\gamma) = n + \wt(\gamma)$.
\end{rmk}

\subsection{Monodromy} \label{subsec:monodromy}
\subsubsection{Monodromy group}
Let $X = S^{[n]}$. Let $\Mon(X)$ be the subgroup of $O(H^{\ast}(X,\BZ))$ generated by all monodromy operators,
and let $\Mon^2(X)$ be its image in $O(H^2(X,\BZ))$.
We let
\[ \mon : \Mon(X) \to O(H^{\ast}(X,\BZ)) \]
denote the monodromy representation.

By results of Markman (\cite[Thm.1.3]{MarkmanSurvey} and \cite[Lemma 2.1]{Markman3}) we have that
\begin{equation} \Mon(X) \cong \Mon^2(X) = \widetilde{O}^+(H^2(X,\BZ)) \label{abc2} \end{equation}
where the first isomorphism is the restriction map and
$\widetilde{O}^+(H^2(X,\BZ))$
is the subgroup of $O(H^2(X,\BZ))$ of orientation preserving lattice automorphisms which act by $\pm 1$ on the discriminant.\footnote{Let
$\CC = \{ x \in H^2(X,\BR) | \langle x, x \rangle > 0 \}$ be the positive cone. Then $\CC$ is homotopy equivalent to $S^2$. An automorphism is orientation preserving
if it acts by $+1$ on $H^2(\CC) = \BZ$.}
If $g \in \Mon^2(X)$, we let $\tau(g) \in \{ \pm 1 \}$ be the sign by which $g$ acts on the discriminant lattice. 
This defines a character
\[ \tau : \Mon^2(X) \to \BZ_2. \]

\subsubsection{Zariski closure} \label{subsec:zariski closure}
By \cite[Lemma 4.11]{Markman} if $n \geq 3$ the Zariski closure of the subgroup $\Mon(X) \subset O(H^{\ast}(X,\BC))$ is $O(H^2(X,\BC)) \times \BZ_2$.
The inclusion yields the representation
\begin{equation} \rho : O(H^2(X,\BC)) \times \BZ_2 \to O(H^{\ast}(X,\BC)) \label{repn} \end{equation}
which acts by degree-preserving orthogonal ring isomorphism.
There is a natural embedding
\[ \widetilde{O}^+(H^2(X,\BZ)) \to O(H^2(X,\BC)) \times \BZ_2,\, g \mapsto (g, \tau(g)) \]
under which $\rho$ restricts to the monodromy representation, that is:
\begin{equation} \mon(g) = \rho( g, \tau(g)) \quad \text{for all } g \in \Mon(X). \label{mon relation} \end{equation}
In case $n \in \{ 1 ,2 \}$ the Zariski closure of $\Mon(X)$ is $O(H^2(X,\BC))$. In this case, we define 
the representation \eqref{repn} by projection to $O(H^2(X,\BC))$ followed by the natural inclusion.

The representation $\rho$ is determined by and has the following properties:

\vspace{5pt}
\noindent
\textbf{Property 0.} For any $(g, \tau) \in O(H^2(X,\BC)) \times \BZ_2$ we have 
\[ \rho(g,\tau)|_{H^2(X,\BC)} = g. \]
\textbf{Property 1.} The restriction of $\rho$ to $\SO(H^2(X,\BC)) \times \{ 1 \}$ is the
integrated action of the Looijenga-Lunts-Verbitsky algebra \cite{LL, V},
\[ \rho|_{\SO(H^2(X,\BC)) \times \{ 0 \}} =  \rho_{\mathrm{LLV}}. \]
\textbf{Property 2.} We have 
\[ \rho(1, -1) = D \circ \rho(-\id_{H^2(X,\BC)}, 1), \]
where
$D$ acts on $H^{2i}(X,\BC)$ by multiplication by $(-1)^i$.\\[5pt]
\textbf{Property 3.} 
The action is equivariant with respect to the Nakajima operators:
For any $g \in O(H^2(X,\BC))$ such that $g(\delta) = \delta$, let $\tilde{g} = g|_{H^2(S,\BC))} \oplus \id_{H^0(S,\BZ) \oplus H^4(S,\BZ)}$. Then
\[ \rho(g, 1) \left( \prod_i \Fq_{k_i}(\alpha_i) 1 \right) = \prod_i \Fq_{k_i}( \tilde{g} \alpha_i ) 1. \]

Property 1 follows by \cite[Lemma 4.13]{Markman}.
Property 3 follows since
the Nakajima operator is naturally equivariant 
with respect to the action of the monodromy group $\Mon(S) = O(H^2(S,\BZ))^{+}$ (of deformations of the K3 surfaces), and this group is Zariski dense in $O(H^2(X,\BC))_{\delta}$.
Properties 0 follows by construction. Property 2 is implicit in \cite[Sec.4]{Markman}, compare also with \cite[Sec.1.1.2]{Markman}.

\subsubsection{Example 1: Involution} \label{monodromy:involution}
\label{mon:involution}
Consider the element $g \in \widetilde{O}^+(H^2(X,\BZ))$ given under the isomorphism \eqref{V description} by
\[ g|_{H^2(S,\BZ)} = \id, \quad g(\delta) = -\delta. \]
Indeed, this is orientation preserving (it fixes a slice of the positive cone) and acts by $-1$ on the discriminant lattice.
We want to describe the action of the corresponding monodromy operator of $X$ defined by the isomorphism \eqref{abc2}.


By Property 2 above we have that
\[ \mon(g) = D \circ \rho(-g, 1). \]
Since $-g$ fixes $\delta$, we obtain the equivariance with respect to the Nakajima operators in the sense of Property 3, that is, if we let
\[ \tilde{g} = \id_{H^0 \oplus H^4} \oplus -\id_{H^2(S,\BZ)} \]
then
\[ \rho(-g, 1) \left( \prod_i \Fq_{k_i}(\alpha_i) 1 \right) = \prod_i \Fq_{k_i}( \tilde{g} \alpha_i ) 1. \]
In particular, if all $\alpha_i$ are homogeneous, we see that
\[ \rho(-g, 1)( \Fq_{k_1}(\alpha_1) \cdots \Fq_{k_{\ell}}(\alpha_{\ell}) \1 ) = (-1)^{ \widetilde{\ell} } \Fq_{k_1}(\alpha_1) \cdots \Fq_{k_{\ell}}(\alpha_{\ell}) \1 \]
where $\widetilde{\ell} = | \{ i : \alpha_i \in H^2(S,\BQ) \} |$.
Using \eqref{deg Nakajima cycle} 
we conclude that
\[ \mathrm{mon}(g)\big( \Fq_{k_1}(\alpha_1) \cdots \Fq_{k_{\ell}}(\alpha_{\ell})\1 \big) =
(-1)^{n+\ell} \Fq_{k_1}(\alpha_1) \cdots \Fq_{k_{\ell}}(\alpha_{\ell})\1. \]

\subsubsection{Example 2: Shift}
\label{monodromy:shift}
The element $\delta \wedge F$ acts on $H^2(X,\BZ)$ by
\[
W \mapsto \delta, \quad \delta \mapsto (2n-2)F, \quad F \mapsto 0,
\quad (\delta \wedge F)|_{ \{ W, F, \delta \}^{\perp}} = 0.
\]
Let $T_{\delta} = \act(\delta \wedge F)$ as before and 
for any $\lambda \in \BZ$ consider the operator
\[
e^{\lambda T_{\delta}} : H^{\ast}(X,\BZ) \to H^{\ast}(X,\BZ).
\]

By a direct check
the operator $e^{\lambda (\delta \wedge F)} : H^2(X,\BZ) \to H^2(X,\BZ)$
is an isometry, which is orientation preserving, acts with $+1$ on the discriminant and has determinant $+1$.
By \eqref{abc2} it hence defines a monodromy operator of $X$.
Moreover, by \eqref{mon relation} and Property 1 we have
\[
\mon( e^{\lambda (\delta \wedge F)} ) = \rho_{\mathrm{LLV}}( e^{\lambda (\delta \wedge F)} )
= e^{\lambda T_{\delta}}.
\]
In particular, $e^{\lambda T_{\delta}}$ is a monodromy operator.

The action of $T_{\delta}$ is compatible with the identification of $H^2(X,\BQ)$ and $H_2(X,\BQ)$ under the Beauville-Bogomolov form.
Hence using $\delta = (2-2n) A$ under this identification,
one finds that $T_{\delta}$ acts on $H_2(X,\BZ)$ by
\[ W \mapsto (2-2n)A, \quad A \mapsto -F, \quad F \mapsto 0. \]
We conclude that
\begin{align*}
e^{\lambda T_{\delta}}(W + dF + rA) = W + \left(d-r \lambda + \lambda^2 (n-1) \right) F + (r -2 \lambda (n-1)) A.
\end{align*}

\subsection{Monodromies preserving the Hodge type of a curve class}
The Gromov-Witten invariants of $S^{[n]}$ in an effective curve class $\alpha \in H_2(S^{[n]})$ are invariant under
deformations which preserve the Hodge type of $\alpha$.
Consider two classes $\alpha, \alpha' \in H_2(S^{[n]})$ which are of Hodge type $(2n-1,2n-1)$ and which {\em pair positively with a K\"ahler class}.
If there is a monodromy operator $\varphi \in \Mon(S^{[n]})$ such that $h \alpha = \alpha'$,
then by the global Torreli theorem \cite{Ver, HuyTor}
there exists a monodromy of $S^{[n]}$ which induces $\varphi$ and which preserves the Hodge type of $\alpha$ along the deformation.
In this case we conclude that:
\[
\left\langle \taut ; \gamma_1, \ldots, \gamma_N \right\rangle^{S^{[n]}}_{g,\alpha} = 
\left\langle \taut ; \varphi(\gamma_1), \ldots, \varphi(\gamma_N) \right\rangle^{S^{[n]}}_{g,\varphi(\alpha)}.
\]

\begin{rmk}
The condition that $\alpha$ and $\varphi(\alpha)$ both pair positively with a K\"ahler class is necessary.
For example, the monodromy operator of Section~\ref{monodromy:involution} sends $A$ to $-A$,
but obviously does not preserve the Gromov-Witten invariants (since $-A$ is not effective).
\end{rmk}

\section{Constraints from the monodromy}
\label{sec:constraints from monodromy}

\subsection{Overview}
Let $S$ be an elliptic K3 surface with section $B$ and fiber class $F$ and define the class
\[ W = B+F. \]
Let $n \geq 2$ and consider the generating series
of Gromov-Witten invariants of $S^{[n]}$:
\begin{equation} \label{Fg}
F^{S^{[n]}}_{g}(\taut; \gamma_1, \ldots, \gamma_N) = \sum_{d \geq -1} \sum_{r \in \mathbb{Z}} \left\langle \taut ; \gamma_1, \ldots, \gamma_N \right\rangle^{S^{[n]}}_{g, W+dF+rA} q^d (-p)^r,
\end{equation}
Our goal in this section is to prove the following:

\begin{prop} \label{prop:constrains monodromy}
There exists unique power series $f_{i,j,s}(q) \in \BQ[[q]]$ such that
\[
F^{S^{[n]}}_{g}(\taut; \gamma_1, \ldots, \gamma_N) =
\frac{\Theta(p,q)^{2n-2}}{\Delta(q)}
\sum_{i=0}^{2n} \sum_{j=0}^{n-1} \sum_{s \in \{ 0 ,1 \}} f_{i,j,s}(q) \A(p,q)^i \wp(p,q)^{j} \wp'(p,q)^s
\]
Moreover, we have the following properties:
\begin{enumerate}
\item[(a)] In the ring $\frac{1}{\Delta(q)} \BQ[[q]][ \A, \wp, \wp', \Theta]$ we have
\[
\frac{d}{d\A} F^{S^{[n]}}_g(\taut; \gamma_1, \ldots, \gamma_N)
=
T_{\delta} F^{S^{[n]}}_g(\taut; \gamma_1, \ldots, \gamma_N) \]
where the right hand side is defined as in \eqref{T alpha action}.
\item[(b)] The series $F^{S^{[n]}}_{g}(\taut; \gamma_1, \ldots, \gamma_N)$ is a power series in $q$ with coefficients which are Laurent polynomials in $p$.
\item[(c)] If the $\gamma_i$ are written in the Nakajima basis (of length $l(\gamma_i)$ as defined in \eqref{length Nakajima}), then
\[ 
F^{S^{[n]}}_{g}(\taut; \gamma_1, \ldots, \gamma_N)(p^{-1})
=
(-1)^{N \cdot n + \sum_i l(\gamma_i)}
F^{S^{[n]}}_{g}(\taut; \gamma_1, \ldots, \gamma_N).
\]
\end{enumerate}
\end{prop}

The idea of the proof of Proposition~\ref{prop:constrains monodromy} is not difficult:
The monodromy operators described in
Section~\ref{monodromy:involution} and~\ref{monodromy:shift}
together with the invariance of Gromov-Witten invariants under deformations (which preserve the Hodge type of the curve class)
yield two basic identities on the generating series
$F^{S^{[n]}}_{g}(\taut; \gamma_1, \ldots, \gamma_N)$.
Up to correction terms coming from insertions of lower weight,
these identities are precisely the conditions 
given in Lemmata~\ref{lemma:elliptic to Jac 1} and~\ref{lemma:elliptic to Jac 2},
and hence up to the correction term force an expression of the series as a certain polynomial in $\wp$, $\Theta$ and $\wp'$.
To control the correction term, we argue by an induction on the order of the weight.
The correction term is then controlled by the $\A$-holomorphic anomaly equation,
and the claim follows by a formal argument.

\subsection{Proof of Proposition~\ref{prop:constrains monodromy}}
We split the proof in two parts:

\vspace{5pt}
\noindent
\textbf{Step 1.: The $p \mapsto p^{-1}$ symmetry}
We first prove the properties (b) and (c).
By Section~\ref{mon:involution}
there exists a monodromy $\mon(g)$ of $S^{[n]}$
which acts on cohomology by
\[
\Fq_{k_1}(\alpha_1) \cdots \Fq_{k_{\ell}}(\alpha_{\ell})\1
\mapsto 
(-1)^{n+\ell} \Fq_{k_1}(\alpha_1) \cdots \Fq_{k_{\ell}}(\alpha_{\ell})\1.
\]
In particular, it acts on $H_2(S^{[n]}, \BZ)$ by the identity on $H_2(S,\BZ)$ and sends $A$ to $-A$.
By deformation invariance of the Gromov-Witten invariants we obtain that:
\begin{align*}
\left\langle \taut ; \gamma_1, \ldots, \gamma_N \right\rangle^{S^{[n]}}_{g, W+dF+rA}
& = 
\left\langle \taut ; \mon(g) \gamma_1, \ldots, \mon(g) \gamma_N \right\rangle^{S^{[n]}}_{g, \mon(g)(W+dF+rA)} \\
& = 
(-1)^{N \cdot n + \ell_1 + \ldots + \ell_N}
\left\langle \taut ; \gamma_1, \ldots, \gamma_N \right\rangle^{S^{[n]}}_{g, W+dF-rA}.
\end{align*}
For any curve class $\beta \in H_2(S,\BZ)$ there exists an integer $r_{\beta}$ such that for all $r \geq r_{\beta}$ there are no curves in $S^{[n]}$ of class $\beta - r A \in H_2(S^{[n]})$.
Hence this equality proves (b) and (c).

\vspace{5pt}
\noindent
\textbf{Step 2.: The $p \mapsto p q^{\lambda}$ symmetry.}
We apply the deformation invariance with respect to the monodromy considered in Section~\ref{monodromy:shift}. It yields
\[
\left\langle \taut ; \gamma_1, \ldots, \gamma_N \right\rangle^{S^{[n]}}_{g, W+dF+rA}
=
\left\langle \taut ; e^{\lambda T_{\delta}} \gamma_1, \ldots, e^{\lambda T_{\delta}} \gamma_N \right\rangle^{S^{[n]}}_{g, W + \left(d-r \lambda + \lambda^2 (n-1) \right) F + (r -2 \lambda (n-1)) A}.
\]
By multiplying with $(-p)^{r-2 \lambda m} q^{d - r \lambda + \lambda^2 m}$, summing over $r$ and $d$ and replacing $\lambda$ by $-\lambda$, we obtain that:
\begin{equation} \label{key transformation property F_g}
F^{S^{[n]}}_{g}(\taut; \gamma_1, \ldots, \gamma_N)(pq^{\lambda}, q)
=
p^{-2 \lambda m} q^{-\lambda^2 m}
F^{S^{[n]}}_{g}(\taut; e^{- \lambda T_{\delta}} \gamma_1, \ldots, e^{-\lambda T_{\delta}} \gamma_N).
\end{equation}

We argue now the remaining claims by induction on the total weight of the insertions
\[ \sum_i \wt(\gamma_i) = L. \]
Assume that the claim of the proposition holds for all insertions $\gamma'_i$ with $\sum_i \wt(\gamma_i') < L$.
(Since we always have $\wt(\gamma_i) \geq -n$, the statement is true for $L<-nN$. This provides the base of the induction.)
Since $T_{\delta}$ decreases the weight by one (see Lemma~\ref{lemma:T action}), the series
\begin{equation} \label{deriv}
\sum_{i=1}^{N} 
F^{S^{[n]}}_{g}(\taut; \gamma_1, \ldots, \gamma_{i-1}, T_{\delta} \gamma_i, \gamma_{i+1}, \ldots, \gamma_N)
\end{equation}
satisfies the induction hypothesis and hence has all the desired properties. In particular, it is equal to $\Theta^{2n-2} \Delta(q)^{-1}$ times a polynomial in $A, \wp, \wp'$ with coefficients power series in $q$. Consider the integral with respect to $A$,
\[ \widetilde{F} = \sum_{i=1}^{N} \int
F^{S^{[n]}}_{g}(\taut; \gamma_1, \ldots, \gamma_{i-1}, T_{\delta} \gamma_i, \gamma_{i+1}, \ldots, \gamma_N) d A, \]
which is defined here formally as the right inverse to $\frac{d}{dA}$
with constant term in $A$ to be zero.
(In other words, $\int A^i dA = A^{i+1}/(i+1)$.)
By Lemma~\ref{elliptic transformation} and using the induction hypothesis to calculate $\frac{d}{dA}$
we obtain the transformation property:
\begin{align*}
p^{2 \lambda m} q^{\lambda^2 m}
\widetilde{F}(pq^{\lambda}, q)
& =
e^{-\lambda \frac{d}{dA}} \widetilde{F}(p,q) \\
& = \widetilde{F}(p,q)
- \lambda \frac{d}{dA} \widetilde{F} 
+ \frac{\lambda^2}{2} \left( \frac{d}{dA} \right)^2 \widetilde{F}
+ \ldots \\
& = \widetilde{F}(p,q) - F^{S^{[n]}}_{g}(\taut; \gamma_1, \ldots, \gamma_N) \\
& \quad \quad \quad + F^{S^{[n]}}_{g}(\taut; e^{-\lambda T_{\delta}} \gamma_1, \ldots, e^{-\lambda T_{\delta}} \gamma_N).
\end{align*}
Using this equation and \eqref{key transformation property F_g} we conclude that
\[ F(p,q) = F^{S^{[n]}}_{g}(\taut; \gamma_1, \ldots, \gamma_N) - \widetilde{F}(p,q) \]
satisfies
\[
p^{2 \lambda m} q^{\lambda^2 m}
F(pq^{\lambda}, q)
= F(p,q).
\]

Since $T_{\delta}$ is a cubic in Nakajima operators (see \eqref{T delta})
its action on a cohomology class changes the parity of the number of Nakajima factors in which it is written.
In particular, if $r=N \cdot n + \sum_i l(\gamma_i)$ is even,
then the function \eqref{deriv} is odd in $p$ by Step 1,
and, since $A(p,q)$ is odd in $p$, its integration with respect to $A$ is again even in $p$.
Similar arguments apply, if $r$ is odd.
We obtain that
\[ F(p^{-1}, q) = (-1)^{N \cdot n + \sum_i l(\gamma_i)} F(p,q). \]

Using Lemmata~\ref{lemma:elliptic to Jac 1} and~\ref{lemma:elliptic to Jac 2} (depending on the parity of $N n + \sum_i \ell_i$) 
we conclude that
\begin{equation} \label{F expression} 
F(p,q) = 
\begin{cases}
\Delta(q)^{-1} \Theta^{2m}(p,q) \wp'(p,q) \sum_{i=2}^{m} f_i(q) \wp(p,q)^{m-i} & \text{ if } Nn + \sum_{i} l(\gamma_i) \text{ is even} \\
\Delta(q)^{-1} \Theta^{2m}(p,q) \sum_{i=0}^{m} f_i(q) \wp(p,q)^{m-i} & \text{ if } Nn + \sum_i l(\gamma_i) \text{ is odd}.
\end{cases}
\end{equation}
for some power series $f_i(q) \in \BC[[q]]$.
This proves the main claim.

Since $F(p,q)$ is written without any $\A$ we have
\begin{multline*}
0 = \frac{d}{d\A} F(p,q) = \frac{d}{d \A} F^{S^{[n]}}_{g}(\taut; \gamma_1, \ldots, \gamma_N) - \frac{d}{d\A} \widetilde{F}(p,q) \\
=\frac{d}{d \A} F^{S^{[n]}}_{g}(\taut; \gamma_1, \ldots, \gamma_N) - 
\sum_{i=1}^{N} 
F^{S^{[n]}}_{g}(\taut; \gamma_1, \ldots, \gamma_{i-1}, T_{\delta} \gamma_i, \gamma_{i+1}, \ldots, \gamma_N)
\end{multline*}
that is we also have the holomorphic anomaly equation (part a) with respect to $\A$. \qed \\

The argument in Step 1 of the proof more generally shows the following:
\begin{lemma} \label{lemma:Laurent polynomial}
For any K3 surface $S$ and effective curve class $\beta \in H_2(S,\BZ)$ the series
\[ Z_{g,\beta}^{S^{[n]}}(\taut; \gamma_1, \ldots \gamma_N) := \sum_{r \in \mathbb{Z}} \left\langle \taut ; \gamma_1, \ldots, \gamma_N \right\rangle^{S^{[n]}}_{g,\beta+rA} (-p)^r \]
is a Laurent polynomial in $p$, and if the $\gamma_i$ are in the Nakajima basis, then
\[ 
Z_{g,\beta}^{S^{[n]}}(\taut; \gamma_1, \ldots \gamma_N)(p^{-1})
=
(-1)^{N \cdot n + \sum_i l(\gamma_i)}
Z_{g,\beta}^{S^{[n]}}(\taut; \gamma_1, \ldots \gamma_N).
\]
\end{lemma}

\section{Relative Gromov-Witten theory} \label{sec:Rel GW theory}
\subsection{Overview}
Let $X$ be a smooth projective divisor and let $D \subset X$ be a smooth divisor with connected components $D_i$ for $i=1,\ldots, N$.
In this section we consider the relative Gromov-Witten theory of the pair $(X,D)$ introduced by Li \cite{Li1,Li2}, see also \cite{ABPZ, MarkedRelative} for introductions.
In the first part we introduce the basic structures of the theory: moduli spaces, evaluation maps, psi classes, and rubber moduli spaces.
Then, we recall three basic equations that will be needed later on:
a splitting formula for the relative diagonal,
proven recently in \cite{ABPZ} (Proposition~\ref{prop:splitting relative diagonal});
a splitting formula for relative psi-classes (Proposition~\ref{prop:splitting relative psi class}), and finally we prove a new formula for the restriction
of relative Gromov-Witten classes
to the non-separating bounary divisor in the moduli space of curves
(Proposition~\ref{prop:restriction to boundary}).

%
%
%

\subsection{Moduli space}
Let $\beta \in H_2(X,\BZ)$ be a curve class and let $\vec{\lambda} = (\vec{\lambda}_{1}, \dots, \vec{\lambda}_N)$ be a tuple of ordered partitions $\lambda_i = (\lambda_{i,j})_{j=1}^{\ell}$ of size and length
\[ |\vec{\lambda}_i| = \sum_{j} \lambda_{i,j} = D_i \cdot \beta, \quad \ell(\lambda_i) = \ell. \]
Consider the moduli space of $r$-pointed genus $g$ degree $\beta$ relative stable maps from connected curves to the pair $(X,D)$
with ordered ramification profile $\vec{\lambda}_i$ along the divisor $D_i$,
\[ \Mbar_{g,r,\beta}((X,D), \vec{\lambda}) \]
By definition, an element of the moduli space is a map
$f : C \to X[k]$
where $X[k]$ is a target degeneration of $X$ along $D$ which satisfies a list of conditions
(finite automorphism, predeformability, no components mapping entirely mapped to the singular fibers, relative multiplicities as specified). The degree
of the map is $\pi_{\ast} f_{\ast}[C] = \beta$
where $\pi : X[k] \to X$ is the canonical map that contracts the expansion.

\subsection{Evaluation maps} \label{sec:eval maps}
For every boundary divisor $D_i$ we have relative evaluation maps
\[ \ev_{i,j}^{\text{rel}} \colon \Mbar_{g,r,\beta}((X,D), \vec{\lambda}) \to D_i, \quad j=1, \ldots, \ell(\vec{\lambda}_i) \]
which send a stable map to the $j$-th intersection point with the divisor $D_i$. 

We also have an interior evaluation map:
\[
\ev : \Mbar_{g,r,\beta}((X,D), \vec{\lambda}) \to (X,D)^r
\]
which takes values in the (smooth projective) moduli space $(X,D)^r$ of (ordered) tuples of $r$ points on the relative geometry $(X,D)$,
see \cite{KS} for a construction.
For example, as a variety
$(X,D)^1$ is isomorphic to $X$,
and $(X,D)^2$ is the blow-up $\mathrm{Bl}_{\sqcup_i D_i \times D_i}(X \times X)$.
We refer to \cite{PaPix_GWPT} and \cite[Sec.3.4]{ABPZ} for beautiful self-explaining figures illustrating the situation.
By forgetting points we have for any $I \subset \{ 1, \ldots, r \}$ contraction maps $p_I : (X,D)^r \to (X,D)^{|I|}$.
We can hence view classes on $\prod_i (X,D)^{a_i}$ with $\sum_i a_i = r$ (such as $X^r$) as defining cohomology classes on $(X,D)^r$
via pullback by the projections.
We write $\ev_I = p_I \circ \ev$.

The class of the locus in $(X,D)^2$ of incident points (the relative diagonal) is denoted by
\[ \Delta_{(X,D)}^{\text{rel}} \subset H^{\ast}((X,D)^2). \]

\subsection{Psi-classes}
There are cotangent line bundles at both interior and relative markings.
We let their first Chern classes be denoted, respectively, by
\[
\psi_i, \quad i=1,\ldots, r, \quad \psi_{i,j}^{\text{rel}} , \quad i=1, \ldots, N, \quad j=1, \ldots, \ell(\lambda_i).
\]

Let also $\BL_{D_i}$ be the cotangent line bundle associated to $D_i$ on the stack of target expansions $\CT$
as defined in \cite[1.5.2]{MP}.
The line bundle $\BL_{D_i}$ has a section which vanishes precisely at expansions corresponding to bubbling at $D_i$.
Let $\Psi_{D_i} = c_1(\BL_{D_i})$ and let
\[ \mathsf{q} : \Mbar_{g,r,\beta}((X,D), \vec{\lambda}) \to \CT \]
be the classifying map corresponding to the universal target over the moduli space.
The relative $\psi$-classes then satisfy the following well-known lemma:
\begin{lemma} \label{lemma:psi rel in terms Psi}
$\lambda_{i,j} \psi_{i,j}^{\rel} = \mathsf{q}^{\ast}( \Psi_i ) - \ev_{i,j}^{\text{rel} \ast }( c_1(N_{D_i/X}) )$,
\end{lemma}
\begin{proof}
See for example \cite[Proof of Lemma 12]{OPix2}.
\end{proof}

\subsection{Cohomology weighted partitions}
Consider a $H^{\ast}(D_i)$-weighted partition $\mu$ 
\begin{equation} \label{weighted partition} \big( (\mu_1, \delta_1) , \ldots, (\mu_{\ell}, \delta_{\ell} ) \big), \quad \delta_j \in H^{\ast}(D_i), \quad \mu_1 \geq \ldots \geq \mu_{\ell} \geq 1.  \end{equation}
We write $\ell = \ell(\mu)$ for the length and $|\mu| = \sum_i \mu_i$ for the size of the partition.
The \emph{partition underlying $\mu$} is the ordered partition 
\[ \vec{\mu} = (\mu_1, \ldots, \mu_{\ell}), \]
While the $\delta_i$ are arbitrary cohomology classes on $D_i$,
we often take them to be elements of a fixed basis $\CB$ of $H^{\ast}(D_i)$.
In this case we say $\mu$ is $\CB$-weighted.
Given a $\CB$-weighted partition $\mu$, the automorphism group $\Aut(\mu)$ consists of the permutation symmetries of $\mu$.

\subsection{Gromov-Witten invariants}
For $i \in \{ 1, \ldots, N \}$ consider $H^{\ast}(D_i)$-weighted partitions
\[ \lambda_i = ((\lambda_{i,j}, \delta_{i,j}))_{j=1}^{\ell(\lambda_i)} \]
and let $\vec{\lambda}_i$ be the partition underlying $\lambda_i$.
Fix also a class
\[ \gamma \in H^{\ast}((X,D)^n). \]
We define relative Gromov-Witten invariants by
integration over the virtual fundamental class \cite{Li2} of the moduli space:
\[
\big\langle \, \lambda_1, \ldots, \lambda_N \, \big| \, \gamma \big\rangle^{(X,D)}_{g,\beta} \\
:=
\int_{ [ \Mbar_{g,r,\beta}((X,D), \vec{\lambda}) ]^{\vir} }
\ev^{\ast}(\gamma) \prod_{i=1}^{N} \prod_{j=1}^{\ell(\lambda_i)} \ev^{\text{rel}}_{i,j}( \delta_{i,j} ) \,.
\]

We will also sometimes need to include $\psi$-classes in the integral.
A more general definition is hence the following.
Let $a_{i,j}$ and $b_{i}$ be arbitrary non-negative integers.
\begin{multline} \label{GWbracket}
\left\langle \, \left( \lambda_i \prod_{j=1}^{\ell(\lambda_i)} (\psi^{\text{rel}}_{i,j})^{a_{ij}} \right)_{i=1}^{N} \, \middle| \, (\tau_{b_1} \cdots \tau_{b_r})(\gamma) \right\rangle^{(X,D)}_{g,\beta} \\
:=
\int_{ [ \Mbar_{g,r,\beta}((X,D), \vec{\lambda}) ]^{\vir} }
\prod_{i=1}^{r} \psi_{i}^{b_i} \cdot  \ev^{\ast}(\gamma) \cdot \prod_{i=1}^{N} \prod_{j=1}^{\ell(\lambda_i)} (\psi^{\text{rel}}_{i,j})^{a_{ij}} \ev^{\text{rel}}_{i,j}( \delta_{i,j} ) \,.
\end{multline}
If all $b_i=0$ we will simply write $\gamma$ instead of $\tau_{b_1} \cdots \tau_{b_r}(\gamma)$.

The discussion above also work 
when we allow the source curve of our relative stable map to be disconnected.
More precisely, we let 
\[ \Mbar_{g,r,\beta}^{\bullet}((X,D), \vec{\lambda}) \]
denote the moduli space of relative stable maps to $(X,D)$ as above
except that we allow disconnected domain curves
and require the following condition:

($\bullet$) For any stable map $f : \Sigma \to (S \times C)[\ell]$ to a target expansion of the pair $(S \times C, S_{z})$,
the stable map $f$ has non-zero degree on every of its connected components.

We define Gromov-Witten invariants in the disconnected case completely parallel as in \eqref{GWbracket}.
The brackets on the left hand side will be denoted with a supscript $\bullet$, as in $\langle .. \rangle^{(X,D), \bullet}$.

\subsection{Rubber moduli space}
For any of the divisors $E \in \{ D_1, \ldots, D_N \}$ consider the projective bundle
\[ \p = \p(N_{E/X} \oplus \CO_{E}) \to E. \]
The projection has two canonical sections
$E_{0}, E_{\infty} \subset \p$
called the zero and infinite section with normal bundle $N_{E/\p} \cong N_{E/X}^{\vee}$
and $N_{E/X}$ respectively.
Let
\begin{equation} \label{rubber moduli space} \Mbar_{g,r,\alpha}^{\sim}((\p,E_0 \sqcup E_{\infty}), \vec{\lambda}) \end{equation}
be the moduli space of genus $g$ degree $\alpha \in H_2(E,\BZ)$ rubber stable maps with target $(\p, E_{0} \sqcup E_{\infty})$.
Elements of the moduli space are maps $f : C \to \p_{l}$, where $\p_l$ is a chain of $l$ copies of $\p$ with zero sections glued along infinite section of the next components,
satisfying a list of conditions. The degree of a rubber stable map
is fixed here to be $\pi_{E \ast} f_{\ast} [C] = \alpha$ where $\pi_E : \p_l \to E$ is the natural projection.
In the definition of \eqref{rubber moduli space} we let the source curve be connected.
If we allow disconnected domains and require condition ($\bullet$), we decorate the moduli space (and the invariants below) with the supscript $\bullet$.
As before we have evaluation maps at the relative markings denoted $\ev_{i,j}^{\text{rel}}$.
By evaluating the composition $\pi_E \circ f$ at the interior marked points we also have a well-defined interior evaluation map:
\[ \ev: \Mbar_{g,r,\alpha}^{\sim}((\p,E_0 \sqcup E_{\infty}), \vec{\lambda}) \to E^r \]
Given $H^{\ast}(E)$-weighted partitions $\lambda, \mu$ and $\gamma \in H^{\ast}(E^r)$ we define:
\[
\big\langle \, \lambda, \mu \, \big| \, \gamma \big\rangle^{(\p,E_0 \sqcup E_{\infty}), \sim}_{g,\alpha} \\
=
\int_{ [ \Mbar_{g,r,\alpha}^{\sim}((\p,E_0 \sqcup E_{\infty}), \vec{\lambda}) ]^{\vir} }
\ev^{\ast}(\gamma) \prod_{i=1}^{N} \prod_{j=1}^{\ell(\lambda_i)} \ev^{\text{rel}}_{i,j}( \delta_{i,j} ) \,.
\]

\subsection{Splitting formulas} \label{subsection:splitting formulas}
We state two splitting formulas that we will need later on.
Let $\iota : D \to X$ denote the inclusion.
We begin with the splitting of the relative diagonal. 
\begin{prop} \label{prop:splitting relative diagonal}
\begin{multline*}
\big\langle \, \lambda_1, \ldots, \lambda_N \, \big| \, \Delta_{(X,D)}^{\text{rel}} \big\rangle^{(X,D), \bullet}_{g,\beta} 
=
\big\langle \, \lambda_1, \ldots, \lambda_N \, \big| \, \Delta_X \big\rangle^{(X,D), \bullet}_{g,\beta}  \\
-
\sum_{i=1}^{N}
\sum_{\mu}
\sum_{\substack{ g_1+g_2=g + 1 - \ell(\mu) \\ \iota_{\ast} \alpha + \beta' = \beta}}
\frac{\prod_i \mu_i}{|\Aut(\mu)|}
\big\langle \, \lambda_1, \ldots, \underbrace{\mu}_{i\text{-th}}, \ldots , \lambda_N \, \big\rangle^{(X,D), \bullet}_{g_1,\beta'} 
\big\langle \, \lambda_i, \mu^{\vee} \, \big| \, \Delta_{D} \big\rangle^{(\p,D_{i,0} \sqcup D_{i,\infty}), \bullet, \sim}_{g_2,\alpha}.
\end{multline*}

In the above formula, $\mu$ runs over all cohomology weigted partitions $\mu = \{ (\mu_i, \gamma_{s_i}) \}$ of size $\beta \cdot D_i$,
with weights from a fixed basis $\{ \gamma_i \}$ of $H^{\ast}(D_i)$.
Moreover, we let $\mu^{\vee} = \{ (\eta_i, \gamma_{s_i}^{\vee}) \}$ be the dual partition,
with weights from the basis $\{ \gamma_i^{\vee} \}$ which is dual to $\{ \gamma_i \}$.
\end{prop}
\begin{proof}
This is a special case of \cite[Theorem 3.10]{ABPZ}.
\end{proof}

Next we explain how to remove the relative $\psi$-classes.
Again we only need a special case (the general case is similar),
and without loss of generality we can consider relative $\psi$-classes for the first component $D_1$.
\begin{prop} \label{prop:splitting relative psi class}
For any $j \in \{ 1, \ldots, \ell(\lambda_1) \}$,
\begin{multline*}
\lambda_{1,j} \big\langle \, \psi_{1,j}^{\text{rel}} \lambda_1, \ldots, \lambda_N \big\rangle^{(X,D), \bullet}_{g,\beta} 
=
- \big\langle \, \widehat{\lambda}_1, \lambda_2, \ldots, \lambda_N \, \big\rangle^{(X,D), \bullet}_{g,\beta}  \\
+
\sum_{\mu}
\sum_{\substack{ g_1+g_2=g + 1 - \ell(\mu) \\ \iota_{\ast} \alpha + \beta' = \beta}}
\frac{\prod_i \mu_i}{|\Aut(\mu)|}
\big\langle \, \lambda_1, \ldots, \underbrace{\mu}_{i\text{-th}}, \ldots , \lambda_N \, \big\rangle^{(X,D), \bullet}_{g_1,\beta'} 
\big\langle \, \lambda_i, \mu^{\vee} \, \big\rangle^{(\p,D_{1,0} \sqcup D_{1,\infty}), \bullet, \sim}_{g_2,\alpha},
\end{multline*}
where $\widehat{\lambda}_1$ is the weighted partition $\lambda_1$
but with $j$-th cohomology weight $\delta_{1j}$ replaced by $\delta_{1j} \cup c_1(N_{D_1/X})$.
Moreover, $\mu$ runs over the same data as in Proposition~\ref{prop:splitting relative diagonal}.
\end{prop}
\begin{proof}
This follows from Lemma~\ref{lemma:psi rel in terms Psi} and \cite{Li2},
compare also \cite[Lem.12]{OPix2}.
\end{proof}

\subsection{Boundary restriction}
We will also require the restriction of relative Gromov-Witten classes to the boundary. 
Consider the class in $H_{\ast}( \Mbar_{g,r,\beta}((X,D) )$ defined by
\begin{equation} \label{J class}
J^{(X,D)}_{g,\beta}(\lambda\, |\, \gamma)
= \ev^{\ast}(\gamma) \prod_{i=1}^{N} \prod_{j=1}^{\ell(\lambda_i)} \ev^{\text{rel}}_{i,j}( \delta_{i,j} ) \cdot [ \Mbar_{g,r,\beta}((X,D), \vec{\lambda}) ]^{\vir}.
\end{equation}
If there exists a forgetful morphism
\[ \tau : \Mbar_{g,r,\beta}((X,D), \vec{\lambda}) \to \Mbar_{g, n}, \]
where $n=r+\sum_i \ell(\vec{\lambda}_i)$ consider also the pushforward
\begin{equation} \label{I class}
I^{(X,D)}_{g,\beta}(\lambda\, |\, \gamma)
= \tau_{\ast}
\left( \ev^{\ast}(\gamma) \prod_{i=1}^{N} \prod_{j=1}^{\ell(\lambda_i)} \ev^{\text{rel}}_{i,j}( \delta_{i,j} ) \cdot [ \Mbar_{g,r,\beta}((X,D), \vec{\lambda}) ]^{\vir} \right).
\end{equation}
Let $u : \Mbar_{g-1,n+1} \to \Mbar_{g,n}$ be the natural gluing morphism.

\begin{prop} \label{prop:restriction to boundary}
\begin{align*}
& u^{\ast} I^{(X,D)}_{g,\beta}(\lambda_1, \ldots, \lambda_N)  
= I_{g-1,\beta}^{(X,D)}\left(\lambda_1, \ldots, \lambda_N \middle| \Delta_{(X,D)}^{\text{rel}} \right)  \\
& \quad +
\sum_{i=1}^{N}
\sum_{\substack{ m \geq 0 \\ g = g_1 + g_2 + m \\ \beta = \beta'+ \iota_{\ast} \alpha}}
\sum_{\substack{ b ,  b_1, \ldots, b_m \\ \ell , \ell_1, \ldots, \ell_m}}
\frac{\prod_{i=1}^{m} b_i}{m!} \Bigg\{ 
\\
& \qquad \begin{array}{r}  \xi_{\ast} j^{\ast} \Bigg[ J^{(X,D), \bullet}_{g_1, \beta'}\Big( \lambda_1, \ldots, \lambda_{i-1}, \Big( \underbrace{(b, \Delta_{D_i,\ell})}_{(n+1)\textup{-th}}, (b_j, \Delta_{D_i, \ell_j})_{j=1}^{m} \Big), \lambda_{i+1}, \ldots, \lambda_N \Big) \\
\boxtimes 
J^{(\p,D_{i,0} \sqcup D_{i,\infty}), \bullet, \sim}_{g_2, \alpha}\Big( \big( \underbrace{(b, \Delta_{D_i,\ell}^{\vee})}_{(n+2)\textup{-th}}, (b_j, \Delta^{\vee}_{D_i, \ell_j})_{j=1}^{m} \big) , \lambda_i \Big) \Bigg]
\end{array} \\
& \qquad \begin{array}{r} +  \xi_{\ast} j^{\ast} \Bigg[ J^{(X,D), \bullet}_{g_1, \beta'}\Big( \lambda_1, \ldots, \lambda_{i-1}, \Big( \underbrace{(b, \Delta_{D_i,\ell})}_{(n+2)\textup{-th}}, (b_j, \Delta_{D_i, \ell_j})_{j=1}^{m} \Big), \lambda_{i+1}, \ldots, \lambda_N \Big) \\
\boxtimes 
J^{(\p,D_{i,0} \sqcup D_{i,\infty}), \bullet, \sim}_{g_2, \alpha}\Big( \big( \underbrace{(b, \Delta_{D_i,\ell}^{\vee})}_{(n+1)\textup{-th}}, (b_j, \Delta^{\vee}_{D_i, \ell_j})_{j=1}^{m} \big) , \lambda_i \Big) \Bigg]
 \Bigg\}
\end{array}
\end{align*}
where
\begin{itemize}[itemsep=0pt]
\item '$(n+1)$-th' stands for labeling the corresponding marked points by $n+1$,
\item $b, b_1, \ldots, b_m$ run over all positive integers such that
$b + \sum_j b_j = \beta \cdot D_i$,
\item $\Delta_D = \sum_{\ell} \Delta_{D, \ell} \otimes \Delta_{D, \ell}^{\vee}$ is a K\"unneth decomposition of the diagonal of $D$.
\end{itemize}
Moreover, $j$ is the embedding of the (closed and open) component
\begin{multline*}
U \subset \Mbar_{g_1,\beta'}((X,D), (\vec{\lambda} \setminus \vec{\lambda}_i, (b ,b_1,\ldots, b_m) ) \times 
\Mbar_{g_2,\alpha}^{\bullet, \sim}((\p,D_{i,0} \sqcup D_{i,\infty}), \vec{\lambda}_i, (b,b_1,\ldots, b_m) )
\end{multline*}
parametrizing pairs $(f_1:C_1 \to X[k], p_i)$ and $(f_2 : C_2 \to \p_{\ell},p_i')$ such that
the curve, which is obtained by gluing $C_1$ to $C_2$ pairwise along the $m$ markings labeled by $b_i$, is connected.
And we let
\[ \xi : U \to \Mbar_{g-1,n+2} \]
is the map that forgets the maps $f_1,f_2$, glues together the curves $C_1, C_2$ pairwise along the markings labeled by $b_i$, and then contracts unstable components.
\end{prop}

A related formula for the restriction of the double ramification cycle 
to the divisor $\Mbar_{g-1,n+2} \to \Mbar_{g,n}$
was given (only with a sketch) by Zvonkine in \cite{Zvonkine}.

\begin{proof}
Let $\mathfrak{M}_{g,n}$ be the Artin stack of prestable curves, where $n=\sum_i \ell(\lambda_i)$.
We refer to \cite{BS1} for an introduction to the stack $\mathfrak{M}_{g,n}$.
The map $\tau$ factors as a morphism $\tilde{\tau}$ to $\mathfrak{M}_{g,n}$ followed by the stabilization map $\mathsf{st}:\mathfrak{M}_{g,n} \to \Mbar_{g,n}$.
Form the fiber diagram
\[
\begin{tikzcd}
M_1 \ar{r}{q} \ar{d}{\rho} & M_2 \ar{r} \ar{d}{\sigma} & \Mbar_{g,r,\beta}((X,D), \vec{\lambda}) \ar{d}{\widetilde{\tau}} \\
\mathfrak{M}_{g-1,n+2} \ar{r}{\tilde{q}} & W \ar{r}{u'} \ar{d}{\mathsf{st}} &  \mathfrak{M}_{g,n} \ar{d}{\mathsf{st}} \\
& \Mbar_{g-1,n+2} \ar{r}{u} & \Mbar_{g,n}.
\end{tikzcd}
\]

Consider also the gluing map on prestable curves
\[ \widetilde{u} = u' \circ \tilde{q} : \mathfrak{M}_{g-1,n+2} \to \mathfrak{M}_{g,n}. \]

We want to apply Proposition~\ref{prop:Schmitt} below.
Observe the following:
\begin{itemize}
\item $\mathfrak{M}_{g,n}$ is smooth and by \cite[Example 4]{BS1} has a good filtration by quotient stacks.
\item Since $u' : W \to \mathfrak{M}_{g,n}$ is representable and $\mathfrak{M}_{g,n}$ has affine stabilizers at geometric points \cite[Prop.3.1]{BS1},
by \cite[Proposition 3.5.5]{Kresch} and \cite[Proposition 3.5.9]{Kresch} $W$ has affine stabilizers at geometric points.
\item The gluing maps $u : \Mbar_{g-1,n+2} \to \Mbar_{g,n}$ and $\widetilde{u} : \mathfrak{M}_{g-1,n+2} \to \mathfrak{M}_{g,n}$
are both representable \cite[Lemma 2.2]{BS1}.
\item By \cite[Prop 3.13]{BS1} the map
\[ \tilde{q} : \mathfrak{M}_{g-1,n+2} \to W=\mathfrak{M}_{g,n} \times_{ \Mbar_{g,n}} \Mbar_{g-1,n+2} \]
is proper and birational. Since $\widetilde{u}$ is representable, $\tilde{q}$ is representable.
\item Since the domain and target of $\widetilde{u}$ is smooth, $\widetilde{u}$ is lci.
\item By \cite[Prop.3]{Beh2} the stabilization map $\mathsf{st} : \mathfrak{M}_{g,n} \to \Mbar_{g,n}$ is flat.
Since $u : \Mbar_{g-1,n+2} \to \Mbar_{g,n}$ is lci, and this is preserved by flat base change (see \cite[Tag 069I]{StacksProject}),
also $u'$ is lci.
\item The map $a: \Mbar_{g,r,\beta}((X,D), \vec{\lambda}) \to \mathfrak{M}_{g,n}$ is representable, since it is injective on stabilizers:
The group of automorphisms of $(C \to X[k], p_i)$ is a subgroup of the group of automorphisms of $(C,p_i)$.
\end{itemize}

By the above, $\widetilde{u}$ and $u$ are proper representable, so $M_1, M_2$ are proper DM stacks.

By Proposition~\ref{prop:Schmitt} below we obtain that
\[ (\iota')^{!} = q_{\ast} \widetilde{\iota}^{!} :
A_{\ast}(\Mbar_{g,r,\beta}((X,D), \vec{\lambda})) \to A_{\ast - 1}(M_2).
 \]

Consider the clas
\[ J := J^{(X,D)}_{g,\beta}(\lambda). \]
We obtain
\begin{align*}
u^{\ast} I^{(X,D)}_{g,\beta}(\lambda) 
& = (\mathsf{st} \circ \sigma)_{\ast} u^{!} J \\
& = (\mathsf{st} \circ \sigma)_{\ast} (u^{\prime})^{!} J \\
& = (\mathsf{st} \circ \sigma)_{\ast} q_{\ast} \tilde{u}^{!} J) \\
& = (q \circ \mathsf{st} \circ \sigma)_{\ast} \tilde{u}^{!} J \\
& = \pi_{\ast} \tilde{u}^{!} J
\end{align*}
where 
\[\pi = q \circ \mathsf{st} \circ \sigma : \mathfrak{M}_{g-1,n+2} \times_{\mathfrak{M}_{g,n}} \Mbar_{g,r,\beta}((X,D), \vec{\lambda}) \to \Mbar_{g,n+2}. \]
Hence we need to compute the refined pullback $\widetilde{u}^{!} J$.

The stack
\[ \mathfrak{M}_{g-1,n+2} \times_{\mathfrak{M}_{g,n}} \Mbar_{g,r,\beta}((X,D), \vec{\lambda}) \]
parametrizes relative stable maps $(f : C \to X[k],p_1,\ldots,p_r)$ together with a chosen non-separating nodal point $p \in C$
and two marking $p_{n_1}, p_{n+2}$ on the partial normalization $\widetilde{C} \to C$ at $p$.
By \cite[Sec.1.5]{ABPZ} we have a disjoint union (both components open and closed)
\[
\mathfrak{M}_{g-1,n+2} \times_{\mathfrak{M}_{g,n}} \Mbar_{g,r,\beta}((X,D), \vec{\lambda}) = \CP_{g,r,\beta}((X,D), \vec{\lambda})
\sqcup \CN_{g,r,\beta}((X,D), \vec{\lambda}).
\]

The component $\CP_{g,r,\beta}((X,D), \vec{\lambda})$ parametrizes relative stable maps where the marked point $p$
map to a {\em non-singular} point on some expanded degeneration $X[k]$ of $(X,D)$.
By \cite[Thm.3.2]{ABPZ} we have then
\[ 
\pi_{\ast}\left( \widetilde{u}^{!}(J)|_{\CP_{g,r,\beta}((X,D), \vec{\lambda})} \right)
= 
I_{g-1,\beta}^{(X,D)}\left(\lambda_1, \ldots, \lambda_N \middle| \Delta_{(X,D)}^{\text{rel}} \right).
\]

The other component $\CN_{g,r,\beta}((X,D), \vec{\lambda})$ parametrizes maps where $p$ maps to the singular locus,
and hence forces a splitting of the source curve $C$,
\[ C = C_1 \cup C_2, \]
where $f|_{C_1} : C_1 \to X[a]$ is a relative stable map to $(X,D)$ and $f|_{C_2} : C_2 \to \p_{\ell}$ maps entirely into a bubble of $D_i$ for some $i$.
The marked points $p_{n+1}, p_{n+2}$ have to lie on different components $C_i$,
hence there are two choices: $p_{n+1}$ can lie on $C_1$ and $p_{n+2}$ lies on $C_2$, or vice versa.
The curve $C$ is obtained by gluing $C_1,C_2$ along $p_{n+1}$ and $p_{n+2}$,
as well as along 'secondary' markings $q_i \in C_1$ and $q'_i \in C_2$ for $i=1, \ldots, m$.
The latter markings are called 'secondary' because they will be forgotten by pushforward along $\pi$ to $\Mbar_{g-1,n+2}$.
Let $b$ be the contact order of $f$ with the divisor at $p_{n+1}$, and let $b_i$ be the contact order at the $q_i$.

We consider the local structure of the component  $\CN_{g,r,\beta}((X,D), \vec{\lambda})$.
A local versal family for the gluing nodes of $C$ is given by $xy=s$ and $x_i y_i = s_i$ for $i=1,\ldots, m$.
Let $t$ be \'etale locally the coordinate defining the bubble splitting $X[a] \cup \p_{\ell}$. The coordinate $t$ is pulled back from the stack of target degeneration. Then 
the local analysis of \cite[Sec.4.4]{Li2} shows
that $t=s^{b}$ and $t=s_i^{b_i}$.
Hence $\CN_{g,r,\beta}((X,D), \vec{\lambda})$, which is cut out by $s=0$, is given by the equations $\{ s=0, s_i^{b_i} = 0 \}$.
On the other hand, the image stack of the glueing morphism
\begin{multline} \label{dfsg343} 
\Mbar_{g_1,\beta'}((X,D), (\vec{\lambda} \setminus \vec{\lambda}_i, (b ,b_1,\ldots, b_m) ) \times_{D^{m+1}} \\
\Mbar_{g_2,\alpha}^{\bullet, \sim}((\p,D_{i,0} \sqcup D_{i,\infty}), \vec{\lambda}_i, (b,b_1,\ldots, b_m) )
\xrightarrow{\ \xi \ } \CN_{g,r,\beta}((X,D), \vec{\lambda})
\end{multline}
is given by $\{ s = 0, s_i=0 \}$. Since the gluing morphism is finite of degree $|\Aut(\eta)|$,
by the arguments in \cite{Li2}, especially Lemma 3.12,
one obtains that the virtual class of $\CN_{g,r,\beta}((X,D), \vec{\lambda})$
is $\prod_{i=1}^{m} b_i / |\Aut(b_1, \ldots, b_m)|$ times the pushforward by $\xi$ of the natural virtual class on the domain of the map \eqref{dfsg343}.\footnote{The more modern viewpoint is to
work relative to the moduli space of stable maps to the universal target $(\BA^1/\BG_m, 0/\BG_m)$
as proposed in \cite{ACW}.
The moduli space $\Mbar_{g,n,d}(\BA^1/\BC^{\ast},0/\BG_m)$ is pure of expected dimension,
and the virtual class on $\Mbar_{g,r,\beta}(X,D)$ is the virtual pullback of the fundamental class on $\Mbar_{g,n,d}(\BA^1/\BC^{\ast},0/\BG_m)$.
The local argument above proves an equality of codimension $1$ classes in $\Mbar_{g,n,d}(\BA^1/\BC^{\ast},0/\BG_m)$.
The equality \eqref{N virtual} of virtual classes on $\Mbar_{g,r,\beta}(X,D)$ follows from this by virtual pullback (after matching the relative perfect obstruction theories).
See \cite[Proof of Thm 3.2]{ABPZ} for a similar case.
I thank P. Bousseau for discussions related to this point.}
In total one obtains:
\begin{multline}  \label{N virtual}
\left( \widetilde{u}^{!} [\Mbar_{g,r,\beta}(X,D) ]^{\vir} \right) |_{\CN_{g,r,\beta}((X,D), \vec{\lambda})}
= \\
\sum_{i=1}^{N} \sum_{\substack{ m \geq 0 \\ g = g_1 + g_2 + m \\ \beta' + \iota_{\ast} \alpha }}
\sum_{b ;  b_1, \ldots, b_m}
\frac{\prod_{i=1}^{m} b_i}{m!}
\xi_{\ast}
\Delta_{D^{m+1}}^{!} j^{\ast} \Big( [ \Mbar_{g_1,\beta'}((X,D), \vec{\lambda} \setminus \vec{\lambda}_i, (\underbrace{b}_{n+1} ,b_1,\ldots, b_m) ) ]^{\vir} \times \\
[ \Mbar_{g_2,\alpha}^{\bullet, \sim}((\p,D_{i,0} \sqcup D_{i,\infty}), \vec{\lambda}_i, (\underbrace{b}_{n+2},b_1,\ldots, b_m) ) ]^{\text{vir}} \Big] \\
+ \text{(same term with role of $(n+1)$ and $(n+2)$ interchanged)}.
\end{multline}
Here we have viewed $(b_1, \ldots, b_m)$ as a list of numbers and not as a partition,
so that the factor $1/ |\Aut(b_1, \ldots, b_m)|$ has to be replaced by $1/ m!$ to compensate for overcounting.

Pushing forward \eqref{N virtual} by $\pi$ completes the proof.
\end{proof}

\begin{example}
We consider a basic example (adapted from \cite{LiICTP}) which illustrates the local analysis in the last step of the proof above
in the case of a univeral target $(\CA, \CD) = (\BA^1/\BG_m, 0 / \BG_m)$.
The universal target was introduced in \cite{ACW}, see also \cite[Proof of Thm.3.2]{ABPZ}.
We let $w_0$ be the coordinate on the chart $\BA^1 \to \CA$.
Let $\CT^1 = \BA^1 / \BG_m$ be the stack of $1$-step target expansion of $(\CA,\CD)$.
The universal family of targets over $\CT^1$ is
\[ \widetilde{\BA^1[1]} = \mathrm{Bl}_{0}(\BA^1 \times \BA^1) \to \BA^1 \]
modulo a quotient by $\BG_m^3$.
Explicitly, if $t$ is the coordinate on $\BA^1$ (the chart of $\CT^1$), then
\[ \widetilde{\BA^1[1]} = \mathrm{Bl}_{0}(\BA^1 \times \BA^1) = V( w_0 z_1 = t w_1 ) \subset \BA^1_{w_0} \times \p^1 \times \BA^1_{t} \]
where $w_1, z_1$ are the homogeneous coordinates on $\p^1$.

Consider a family of degenerating curves $C = \BA^2 \to \BA^1_s$ given by $(x,y) \mapsto s=xy$,
and consider the commutative diagram:
\[
\begin{tikzcd}
C \ar{r}{f} \ar{d} & \BA^2_{w_0, Z} \subset \widetilde{\BA^1[1]} \ar{d} \\
\BA^1_s \ar{r} & \BA^1_{t}
\end{tikzcd}
\]
where we let $\BA^2_{w_0, Z}$ be the affine chart $\Spec( \BC[ w_0, Z ] ) \subset \widetilde{\BA^1[1]}$ for $Z = z_1/w_1$,
and the map $f$ is described by $x \mapsto w^r$ and $y \mapsto Z^r$.
Then the lower horizontal map is given by $t \mapsto s^r$,
that is the coordinate defining the bubble $t$ corresponds to the $r$-th power of the coordinate defining the node of $C$. \qed
\end{example}

\begin{prop}[Schmitt] \label{prop:Schmitt}
Consider the following data:
\begin{itemize}
\item Let $X,Y,Z$ be algebraic stacks locally of finite type over $\BC$ of pure dimension,
and assume that $Y$ has affine stabilizers at geometric points,
and that $Z$ is smooth and has a good filtration by finite type substacks\footnote{In the sense of \cite[Defn.A.2]{BS1} or \cite[Defn.5]{Oes},
i.e. there exists a collection $(\CU_{m})_{m \in \BN}$ of open substacks of finite type of $Z$
with $\CU_m \subset \CU_{\ell}$ for $m \leq \ell$ and such that $\dim( Z \setminus \CU_{m} ) < \dim(Z) - m$},
\item let $g : X \to Y$ be proper birational of DM type, let $f: Y \to Z$ be representable and lci of relative dimension $k$,
and assume that $h = g \circ f : X \to Z$ is representable and lci,
\item let $W$ be a finite type DM stack and let $a:W \to Z$ be a representable morphism.
\end{itemize}
Consider the fiber diagram 
\[
\begin{tikzcd}
U \ar{r}{\tilde{g}} \ar{d} & V \ar{r}{\tilde{f}} \ar{d} & W \ar{d}{a} \\
X \ar[bend right]{rr}{h} \ar{r}{g} & Y \ar{r}{f} & Z.
\end{tikzcd}
\]
Then we have
\[
f^{!} = \tilde{g}_{\ast} h^{!} : A_{\ast}(W) \to A_{\ast+k}(V).
\]
\end{prop}

\begin{proof}
We work with the Chow groups as introduced in \cite[App.A]{BS1} and \cite{Kresch}.
In particular, for any locally finite type algebraic stack $\CX$ over $\BC$ we define
\[ A_{\ast}(\CX) = \varprojlim_{i} A_{\ast}(\CU_i) \]
where $(\CU_i)_{i \in I}$ is a directed system of finite type open substacks of $\CX$ whose union is all of $\CX$,
and the Chow groups $A_{\ast}(\CU_i)$ are taken with $\BQ$-coefficients in the sense of Kresch \cite{Kresch}.
If $\CX$ is pure dimensional and admits a good filtration $(\CU_{m})_{m \in \BN}$  by finite type substacks then
\[ A_{\dim(\CX)-d}(\CX) = A_{\dim(\CX)-d}(\CU_m) \quad \text{for all } m>d. \]
In this case all functionalities of Kresch's Chow groups also apply to $A_{\ast}(\CX)$.

Kresch defines only a projective pushforward.
A proper pushforward along proper morphisms of DM type has been defined
in \cite[Theorem B.17]{BS1}, assuming that the target has affine stabilizers at geometric points,
or equivalently is stratified by quotient stacks \cite[Thm. 2.1.12]{Kresch}.
In particular, by our assumptions on $Y$ there exists a proper pushforward $g_{\ast}$.

Assume first that $W$ is a smooth finite type scheme.
Then since the source and target of $a$ are smooth, $a$ is lci.
By the commutativity of refined pullbacks \cite{Kresch},
and the compatibility \cite[Proposition B.18]{BS1} of proper pushforward (along the DM type morphism $g$) and refined Gysin pullback (along the representable morphism $a$)
we then have
\begin{multline} \label{W smooth}
f^{!} [W] = f^{!} a^{!} [Z] 
= a^{!} f^{!} [Z] 
= a^{!} [Y] 
\overset{(\ast)}{=} a^{!} g_{\ast} [X] \\
= \tilde{g}_{\ast} a^{!}[X] 
= \tilde{g}_{\ast} a^{!} h^{!}[Z] 
= \tilde{g}_{\ast} h^{!} a^{!}[Z] 
= \tilde{g}_{\ast} h^{!} [X],
\end{multline}
where (*) follows since $g$ is birational and hence of degree $1$, compare \cite[Prop.25]{BHPSS}.

In the general case, the Chow group of $W$ is generated by $\iota_{\ast} [ \widetilde{W} ]$, where $\widetilde{W}$ are smooth finite type schemes
and $\iota : \widetilde{W} \to W$ is proper and representable.
Form the fiber diagram:
\[
\begin{tikzcd}
\widetilde{U} \ar{r}{\dbtilde{g}} \ar{d}{\iota''} & \widetilde{V} \ar{r}{\dbtilde{f}} \ar{d}{\iota'} & \widetilde{W} \ar{d}{\iota} \\
U \ar{r}{\tilde{g}} \ar{d} & V \ar{r}{\tilde{f}} \ar{d} & W \ar{d}{a} \\
X \ar[bend right]{rr}{h} \ar{r}{g} & Y \ar{r}{f} & Z.
\end{tikzcd}
\]
With \eqref{W smooth} and using again
the compatibility \cite[Proposition B.18]{BS1} of proper pushforward and refined Gysin pullback (along the representable morphisms $f,h$) we find:
\[
f^{!} \iota_{\ast} [\widetilde{W}] = \iota'_{\ast} f^{!} [\widetilde{W}] = \iota'_{\ast} \dbtilde{g}_{\ast} h^{!} [\widetilde{W}]
=
\widetilde{g}_{\ast} \iota''_{\ast} h^{!} [\widetilde{W}] = \widetilde{g}_{\ast} h^{!} \iota_{\ast}[\widetilde{W}].
\]
\end{proof}

\section{Relative Gromov-Witten theory of $(K3 \times C, K3_z)$} \label{sec:Rel GW of K3xC}
\subsection{Overview}
Let $S$ be a smooth projective K3 surface,
let $C$ be a smooth curve and let $z=(z_1, \ldots, z_N)$ be a tuple of distinct points $z_i \in C$.
We specialize here to the relative Gromov-Witten theory of
the pair
\begin{equation} (S \times C, S_z), \quad S_{z} = \bigsqcup S \times \{ z_i \}: \label{relative geometry} \end{equation}
After introducing our notation for the relative Gromov-Witten invariants in Section~\ref{Section:Relative_Gromov_Witten_theory_of_P1K3},
we state in Section~\ref{sec:hilb/GW correspondence} the main input in this paper: the correspondence between relative invariants of $(S \times C, S_z)$ and the invariants of the Hilbert scheme of $S$ (Theorem~\ref{thm:GW/Hilb}).

Then we discuss further preliminaries.
In Section~\ref{subsec:Degeneration formula} we state the reduced degeneration formula. In Sections~\ref{subsec:Non-reduced invariants} and~\ref{subsec:reduced rubber invariants} we give basic evaluations of non-reduced invariants and reduced rubber invariants.

The curve classes on $S \times C$ will be denoted throughout by
\[ (\beta, n) = \iota_{\ast} \beta + n [C] \in H_2(S \times C, \BZ) \cong H_2(S,\BZ) \oplus \BZ [C]. \]

\subsection{Definition}
\label{Section:Relative_Gromov_Witten_theory_of_P1K3}
For $i \in \{ 1, \ldots, N \}$, consider $H^{\ast}(S)$-weighted partitions
\[ \lambda_i =  \big((\lambda_{i,j}, \delta_{i,j}) \big)_{j=1}^{\ell(\lambda_i)} \]
of size $n$ with underlying partition $\vec{\lambda}_i$.
Let also $\gamma \in H^{\ast}( (S \times C, S_z)^r )$ be a cohomology class.
If $\beta \neq 0$, define the partition function of reduced Gromov-Witten invariants
\begin{multline} \label{defn:ZGW}
Z^{(S \times C, S_z)}_{\GW, (\beta,n)}\left( \lambda_1, \ldots, \lambda_N \middle| (\tau_{k_1} \cdots \tau_{k_r})(\gamma) \right) 
=
(-1)^{(1-g(C)-N)n + \sum_i \ell(\lambda_i)}
z^{(2-2g(C)-N)n + \sum_i \ell(\lambda_i)} \\
\sum_{g \in \BZ} (-1)^{g-1} z^{2g-2}
\left\langle \, \lambda_1, \ldots, \lambda_N \, \middle| \, (\tau_{k_1} \cdots \tau_{k_r})(\gamma) \right\rangle^{(S \times C,S_z), \bullet}_{g, (\beta,n)}
\end{multline}
where the invariants on the right hand side are defined by integration over the \emph{reduced virtual fundamental class}
of the moduli space which is obtained by cosection localization \cite{KL}
from the surjective cosection constructed in \cite{MP_GWNL, MPT}.
If all $k_i=0$, we often just write $\gamma$ instead of $\tau_{k_1} \cdots \tau_{k_r}(\gamma)$.
Sometimes we will also include psi classes $\psi_{i,j}^{\mathrm{rel}}$ at the relative markings
where we follow the notation of \eqref{GWbracket}.
If $\beta = 0$, the series \eqref{defn:ZGW} is defined to vanish.

For any $(\beta,m)$ the moduli space $\Mbar^{\bullet}_{g,r,(\beta,n)}((S \times C,S_z), \vec{\lambda})$ carries also the ordinary or standard
 (i.e. non-reduced) virtual class.
By the existence of the non-trivial cosection it vanishes for all $\beta \neq 0$,
so it is only interesting for $\beta=0$.
In case $\beta=0$ we denote it by $[ - ]^{\std}$.
If we integrate over the 'standard' virtual class,
we decorate the corresponding Gromov-Witten bracket and the partition function $Z$ with a supscript $\std$.
The rest of the notation is unchanged.

We can associate to every $H^{\ast}(S)$-weighted partition a class on the Hilbert scheme:
\begin{defn} \label{defn:coh weighted to Nakajima}
The class in $H^{\ast}(S^{[n]})$ associated to a $H^{\ast}(S)$-weighted partition $\mu = \{ (\mu_i, \delta_{i}) \}$ of size $n$ is defined by
\begin{equation} \mu = \frac{1}{\prod_i \mu_i} \prod_{i} \Fq_{\mu_i}(\delta_i) \vacuum. \label{identification} \end{equation}
\end{defn}

We extend the Gromov-Witten bracket \eqref{GWbracket} for $(S \times C, S_z)$ and the partition functions $Z(..)$ by multilinearity in the entries $\lambda_i$.
Since the Gromov-Witten bracket is invariant under permutations of relative markings that preserve the ramification profile (i.e. under $\Aut(\vec{\lambda}_i)$),
the partition function $Z^{(S \times C, S_z)}_{\GW, (\beta,n)}\left( \lambda_1, \ldots, \lambda_N \middle| \gamma \right)$
only depends on the associated class $\lambda_i \in H^{\ast}(S^{[n]})$. Hence we obtain a morphism:
\[
Z^{(S \times C, S_z)}_{\GW, (\beta,n)}\left( -, \ldots, - \middle| \gamma \right) : H^{\ast}(S^{[n]})^{\otimes N} \to \BQ((z)).
\]

\subsection{Hilb/GW correspondence} \label{sec:hilb/GW correspondence}
Assume that $2g(C)-2+N>0$ so that $(C, z_1, \ldots, z_N)$
is a marked stable curve,
\[ \xi = [ (C, z_1, \ldots, z_N) ] \in \Mbar_{g,N}. \]
Given classes $\lambda_1, \ldots, \lambda_k \in H^{\ast}(S^{[n]})$
we define the generating series
\begin{equation} \label{Z Hilb series}
Z^{(S \times C, S_z)}_{\Hilb, (\beta,n)}\left( \lambda_1, \ldots, \lambda_N \right)
=
\sum_{r \in \mathbb{Z}}
(-p)^r
\int_{[ \Mbar_{g(C),N}(S^{[n]}, \beta+rA) ]^{\vir}}
\tau^{\ast}([\xi]) \prod_i \ev_i^{\ast}(\lambda_i).
\end{equation}
By Lemma~\ref{lemma:Laurent polynomial} the series \eqref{Z Hilb series} is a Laurent polynomial in $p$.

\begin{thm}[\cite{N1,N2,MarkedRelative}] \label{thm:GW/Hilb}
If $\beta \in H_2(S,\BZ)$ is primitive, then we have
\[ Z^{(S \times C, S_z)}_{\Hilb, (\beta,n)}\left( \lambda_1, \ldots, \lambda_N \right)
=
Z^{(S \times C, S_z)}_{\GW, (\beta,n)}\left( \lambda_1, \ldots, \lambda_N \right).
\]
under the variable change $p=e^{z}$.
\end{thm}
\begin{proof}
Denis Nesterov in \cite{N1,N2} showed that the left hand side is equal to
a partition function of relative Pandharipande-Thomas invariants
of $(S \times C, S_z$), see in particular \cite[Cor.4.5]{N2}.
The statement follows then from the GW/PT correspondence for $(S \times C, S_z)$ proven in \cite[Thm.1.2]{MarkedRelative} whenever $\beta$ is primitive.
\end{proof}

\begin{rmk} \label{rmk:GW/Hilb imprimitive}
If the multiple cover conjecture \cite[C2]{K3xE} holds for an effective curve class $\beta \in H_2(S,\BZ)$
then Theorem~\ref{thm:GW/Hilb} holds also for $\beta$, see \cite[Prop.1.4]{MarkedRelative}.
\end{rmk}

\subsection{Degeneration formula} \label{subsec:Degeneration formula}
We recall the reduced degeneration formula for reduced invariants.
Let $C \rightsquigarrow C_1 \cup_x C_2$ be a degeneration of $C$.
Let
\[ \{ 1, \ldots, N \} = A_1 \sqcup A_2 \]
be a partition of the index set of relative divisors,
and write $z(A_i) = \{ z_j | j \in A_i \}$.
We choose that the points in $A_i$ specialize to the curve $C_i$ disjoint from $x$.
Recall also the K\"unneth decomposition of the diagonal of the Hilbert scheme in the Nakajima basis:

\begin{lemma} In $H^{\ast}(S^{[n]} \times S^{[n]})$ we have
\begin{equation} \label{Kunneth diagonal}
\Delta_{S^{[n]}} = \sum_{\mu} (-1)^{n - \ell(\mu)}
\frac{\prod_i \mu_i}{|\Aut(\mu)|} \cdot \mu \boxtimes \mu^{\vee}.
\end{equation}
where $\mu$ runs over all cohomology weighted partitions $\mu = \{ (\mu_i, \gamma_{s_i}) \}$ with weights from a fixed basis $\CB = (\gamma_1, \ldots, \gamma_{24})$ of $H^{\ast}(S)$,
and $\mu^{\vee} = \{ (\eta_i, \gamma_{s_i}^{\vee}) \}$ is the dual partition.
\end{lemma}
\begin{proof}
For $\CB$-weighted partitions $\mu, \nu$ one has
$\int_{S^{{n]}}} \mu \cdot \nu^{\vee} = \delta_{\mu \nu} (-1)^{n+\ell(\mu)} |\Aut(\mu)|/ \prod_i \mu_i.$
\end{proof}

\begin{prop}
For any $\alpha_i \in H^{\ast}(S \times C)$ we have:
\begin{gather*}
Z^{(S \times C, S_z)}_{\GW, (\beta,n)}\left( \lambda_1, \ldots, \lambda_N \middle| \prod_{i} \tau_{k_i}(\alpha_i) \right) = 
\sum_{ \{ 1, \ldots, r \} = B_1 \sqcup B_2 } \Bigg( \\
\phantom{+}
Z^{(S \times C_1,S_{z(A_1),x})}_{(\beta,n)}\left( \prod_{i \in A_1} \lambda_i, \Delta_1 \middle| \prod_{i \in B_1} \tau_{k_i}( \alpha_i ) \right)
Z^{(S \times C_2,S_{z(A_2),x}), \std}_{(0,n)}\left( \prod_{i \in A_2} \lambda_i, \Delta_2 \middle| \prod_{i \in B_2} \tau_{k_i}(  \alpha_i ) \right) \phantom{\Bigg)}\\
+
Z^{(S \times C_1,S_{z(A_1),x}),\std}_{(0,n)}\left( \prod_{i \in A_1} \lambda_i, \Delta_1 \middle| \prod_{i \in B_1} \tau_{k_i}( \alpha_i ) \right)
Z^{(S \times C_2,S_{z(A_2),x})}_{(\beta,n)}\left( \prod_{i \in A_2} \lambda_i, \Delta_2 \middle| \prod_{i \in B_2} \tau_{k_i}(  \alpha_i ) \right) \Bigg)
\end{gather*}
where
$(\Delta_1, \Delta_2)$ stands for summing over the K\"unneth decomposition of the diagonal \eqref{Kunneth diagonal}. 
\end{prop}
\begin{proof}
The required modifications to the usual degeneration formula of Li \cite{Li1,Li2}
needed in the reduced case are discussed in \cite{MPT}.
We refer also to \cite[Sec.5.3]{MarkedRelative} for a discussion
of the matching of signs and exponents, and to \cite[Sec.8.1]{MarkedRelative}
for a conceptual explanation for the form of the equation.
\end{proof}

\subsection{Rubber invariants}
We will also need generating series of rubber invariants.
For any $\alpha_i \in H^{\ast}(S)$ define
\begin{multline*}
Z^{(S \times \p^1, S_{0,\infty}),\sim}_{\GW, (\beta,n)}\left( \lambda, \mu \middle| \prod_i \tau_{k_i}(\alpha_i) \right) \\
=
(-1)^{-n + \ell(\lambda) + \ell(\mu)}
z^{\ell(\lambda) + \ell(\mu)}
\sum_{g \in \BZ} (-1)^{g-1} z^{2g-2}
\left\langle \lambda, \mu \middle| \prod_i \tau_{k_i}(\alpha_i) \right\rangle_{g, (\beta,n)}^{(S \times \p^1, S_{0,\infty}),\bullet, \sim}
\end{multline*}
where the brackets on the right hand side are defined by integrating over the reduced virtual class of the moduli space of rubber stable maps to $(S \times \p^1, S \times \{0, \infty \})$.
The rubber invariants for the standard (non-reduced) virtual class are denoted by $\std$.

\subsection{Non-reduced invariants}
\label{subsec:Non-reduced invariants}
We state two explicit evaluations of non-reduced relative invariants:

\begin{prop}[\cite{BP}]  \label{prop:std evaluation}
For any cohomology weighted partitions $\lambda_1, \ldots, \lambda_N$ of size $n$,
\[
Z^{(S \times \p^1,S_{z}), \std}_{\GW, (0,n)}
\left( \lambda_1, \ldots, \lambda_N \right)
=
\int_{S^{[n]}} \lambda_1 \cup \ldots \cup \lambda_N.
\]
\end{prop}

Recall the class $\delta \in H^2(S^{[n]})$ from Section~\ref{sec:cohomology}.
\begin{prop}
 \label{prop:rubber std}
\[
Z^{\sim, \std}_{\GW, (0,n)}( \lambda, \mu )
= 
z \int_{S^{[n]}} \delta \cup \lambda \cup \mu
\]
\end{prop}
\begin{proof}
Consider first the connected rubber invariants $\langle \lambda, \mu \rangle^{\sim}_{g, (0,n)}$ (no $\bullet$ means connected).
By the stability of the moduli space we have $2g-2+\ell(\lambda) + \ell(\mu) > 0$.
Hence we can apply the product formula which shows that the invariant vanishes for $g \geq 2$.
If $g=1$ all the cohomology weights of $\lambda, \mu$ have to be of degree $0$,
hence $\deg(\lambda) + \deg(\mu) \leq 2n-2$. Since the moduli space is of virtual dimension $2n-1$, the integral vanishes.
This leaves $g=0$. Let $\lambda = (\lambda_i, \gamma_i)$ and $\mu = (\mu_i, \gamma_i')$.
We find
\[ \sum_i \deg(\gamma_i) + \sum_i \deg_i(\gamma_i') = 2. \]
On the other hands, by \eqref{deg Nakajima cycle} we have
\begin{gather*}
\deg(\lambda) = n - \ell(\lambda) + \sum_i \deg(\gamma_i) \\
\deg(\mu) = n- \ell(\mu) + \sum_i \deg(\gamma_i')
\end{gather*}
and moreover we can assume the dimension constraint:
\[ \deg(\lambda) + \deg(\mu) = 2n-1. \]
Inserting, we find $\ell(\lambda) + \ell(\mu) = 3$.
If we assume that $\lambda= ((\lambda_a, \gamma_a) (\lambda_b, \gamma_b))$ and $\mu = ((\mu_c, \gamma'_c))$,
then by the product formula we obtain
\begin{align*}
\langle \lambda, \mu \rangle^{\sim}_{g, (0,n)}
& =
\delta_{g0}
\int_{[ \Mbar_{0,3}(S,0) ]^{\std}} \pi^{\ast}( \DR_0(\lambda_a, \lambda_b, -\mu_c)) \ev_1^{\ast}(\gamma_a) \ev_2^{\ast}(\gamma_b) \ev_3^{\ast}(\gamma'_c) \\
& =
\delta_{g0}
\int_S \gamma_a \gamma_b \gamma'_c.
\end{align*}
where $\DR_g(a)$ is the double ramification cycle and we used that it is $=1$ in genus 0.

For the disconnected case, recall that all connected non-rubber invariants of $(S \times \p^1,S_{0,\infty})$
with only relative insertions vanish (see e.g. \cite[Lemma 2]{HAE}), except for the tube evaluation
\[
\int_{[ \Mbar^{\bullet}_{g}( S \times \p^1 / S_{0,\infty}, (0,n), ( (n), (n)) ) ]^{\vir}}
\ev_1^{\ast}(\gamma) \ev_2^{\ast}(\gamma') = \delta_{0g} \frac{1}{n} \int_{S} \gamma \cdot \gamma'.
\]
(This also proves Proposition~\ref{prop:std evaluation} in the case $N=2$.)
Moreover, in the disconnected series, we have one rubber term and the remaining terms are non-rubber,

We conclude that we must have $\ell(\lambda) = \ell(\mu) \pm 1$, otherwise all invariants vanish.
We assume that $\ell(\lambda) = \ell(\mu)+1$, the other case is parallel.
We find that
\begin{multline*}
Z^{\sim, \std}_{\GW, (0,n)}( \lambda, \mu ) 
= 
\sum_{g \in \BZ} (-1)^{g-1} 
(-1)^{-n + \ell(\lambda) + \ell(\mu)} z^{2g-2 + \ell(\lambda) + \ell(\mu)}
\Big\langle \lambda, \mu \Big\rangle_{g, (0,n)}^{\bullet, \sim} \\
= 
\sum_{g \in \BZ} (-1)^{g-1} 
(-1)^{-n + \ell(\lambda) + \ell(\mu)} z^{2g-2 + \ell(\lambda) + \ell(\mu)}
\sum_{\substack{ 1 \leq a,b \leq \ell(\lambda)  \\ a \neq b \\ 1 \leq c \leq \ell(\mu) }} 
\left( \delta_{g+\ell(\lambda'), 0} \int_{S} \gamma_{a} \gamma_{b} \gamma'_c \right) \\
\left( (-1)^{ |\lambda'| + \ell(\lambda')} \frac{1}{\prod_{i \neq a,b} \lambda_i \prod_{i \neq c} \mu_i} \int_{ S^{[ |\lambda'| ]} } \prod_{i \neq a,b} \Fq_{\lambda_i}(\gamma_i) \1 \cup \prod_{i \neq c} \Fq_{\mu_i}(\gamma'_i) \1 \right) 
\end{multline*}
where $\lambda'$ is the partition $\lambda$ without the parts $(\lambda_a, \gamma_a), (\lambda_b, \gamma_b)$. Since it is of length $\ell(\lambda') = \ell(\lambda)-2$,
we obtain:
\begin{multline*}
Z^{\sim, \std}_{\GW, (0,n)}( \lambda, \mu ) 
=
\frac{z}{\prod_i \lambda_i \prod_i \mu_i} \\
\sum_{\substack{ 1 \leq a,b \leq \ell(\lambda)  \\ a \neq b \\ 1 \leq c \leq \ell(\mu) }} 
(-1)^{\lambda_a + \lambda_b} \lambda_a \lambda_b \mu_c \left( \int_{S} \gamma_a \gamma_b \gamma'_c \right)
\int_{S^{[n-\lambda_a - \lambda_b]}}  \prod_{i \neq a,b} \Fq_{\lambda_i}(\gamma_i) \1 \cup \prod_{i \neq c} \Fq_{\mu_i}(\gamma'_i) \1.
\end{multline*}

On the other side, recall that the operator of cup product with $\delta$ can be explicitly described as a cubic in Nakajima operators \eqref{LLV_operators}.
For $i = j+k$, one obtains:
\[
\Big( \Fq_i(\gamma_i), e_{\delta}\, \Fq_j(\gamma_j) \Fq_k(\gamma_k) \1 \Big)
= (-1)^{j+k} i \cdot j \cdot k \int_S \gamma_i \gamma_j \gamma_k.
\]
where we write $(- , - )$ for the intersection pairing on $S^{[n]}$.
One finds that
$\int_{S^{[n]}} \delta \cup \lambda \cup \mu$
vanishes unless  $\ell(\lambda) = \ell(\mu) \pm 1$.
Assuming that $\ell(\lambda) = \ell(\mu)+1$ we compute:
\begin{multline*}
\int_{S^{[n]}} \delta \cup \lambda \cup \mu
=
\frac{1}{\prod_i \lambda_i \prod_i \mu_i} \\
\sum_{\substack{ 1 \leq a,b \leq \ell(\lambda)  \\ a \neq b \\ 1 \leq c \leq \ell(\mu) }} 
(-1)^{\lambda_a + \lambda_b} \left( \int_{S} \gamma_a \gamma_b \gamma'_c \right)
\int_{S^{[n-\lambda_a - \lambda_b]}}  \prod_{i \neq a,b} \Fq_{\lambda_i}(\gamma_i) \1 \cup \prod_{i \neq c} \Fq_{\mu_i}(\gamma'_i).
\end{multline*}
The claim follows by comparison.
\end{proof}

\subsection{Reduced rubber invariants}
\label{subsec:reduced rubber invariants}
The reduced rubber invariants can be expressed in terms of the non-reduced ones by rigidification.
This is the K3 surface analogue of \cite[Prop. 4.4]{M}:

\begin{prop} \label{prop:rubber in terms of non-rubber} For any $D \in H^2(S)$ and $\beta \neq 0$ we have
\begin{multline*}
(D \cdot \beta)
Z^{(S \times \p^1, S_{0,\infty}),\sim}_{\GW, (\beta,n)}( \lambda, \mu ) \\
= Z^{(S \times \p^1, S_{0,1,\infty})}_{\GW,(\beta,n)}\big(\lambda, \mu, D \big)
+ \left( \int_{S^{[n]}} \lambda \cdot \mu \right) 
Z^{(S \times \p^1,S_{0})}_{\GW, (\beta,n)}\big( (1,\pt)^n | \tau_0(\omega D) \big).
\end{multline*}
where $D = \frac{1}{(n-1)!} ( (1,D)(1,1)^{n-1})$.
\end{prop}
\begin{proof}[Proof of Proposition~\ref{prop:rubber in terms of non-rubber}]
Rigidification of the rubber as discussed in \cite[Prop. 4.3]{M} (or \cite{MP}, or \cite[Prop.3.12]{MarkedRelative}) implies:
\begin{align*}
(\beta \cdot D) \big\langle \lambda, \mu \big\rangle_{g, (\beta,n)}^{(S \times \p^1, S_{0,\infty}),\sim}
& = 
\big\langle \tau_0(D) \big| \lambda, \mu \big\rangle_{g, (\beta,n)}^{(S \times \p^1, S_{0,\infty}),\sim} \\
& = 
\big\langle \tau_0(\omega D) \big| \lambda, \mu \big\rangle_{g, (\beta,n)}^{(S \times \p^1, S_{0,\infty})}.
\end{align*}
For the disconnected rubber invariants we hence obtain that:
\begin{align*}
& (\beta,D) Z^{(S \times \p^1, S_{0,\infty}),\sim}_{\GW, (\beta,n)}( \lambda, \mu ) \\
= & \sum_{g \in \BZ} (-1)^{g-1} 
(-1)^{-n + \ell(\lambda) + \ell(\mu)} z^{2g-2 + \ell(\lambda) + \ell(\mu)}
\sum_{\substack{\lambda = \lambda' \sqcup \lambda'' \\
\mu = \mu' \sqcup \mu''}}
\Big\langle \tau_0(\omega D) \Big| \lambda'', \mu'' \Big\rangle_{g + \ell(\lambda'), (\beta,n)}^{(S \times \p^1, S_{0,\infty})}
 \\
& \quad \cdot \left( (-1)^{ |\lambda'| + \ell(\lambda')} \frac{1}{\prod_{\lambda_i \in \lambda'} \lambda_i \prod_{i \in \mu'} \mu_i} \int_{ S^{[ |\lambda'| ]} } \prod_{\lambda_i \in \lambda'} \Fq_{\lambda_i}(\gamma_i) \1 \cup \prod_{\mu_i \in \mu'} \Fq_{\mu_i}(\gamma'_i) \1 \right) \\
= &
Z^{(S \times \p^1, S_{0,\infty})}_{\GW, (\beta,n)}\left( \tau_0(\omega D) \middle| \lambda, \mu \right).
\end{align*}
We now apply the degeneration formula which gives:
\begin{multline}
\label{sdasf}
Z^{(S \times \p^1, S_{0,\infty})}_{\GW, (\beta,n)}\left( \tau_0(\omega D) \middle| \lambda, \mu \right) \\
= \sum_{\nu} 
Z^{(S \times \p^1, S_{0,1,\infty})}_{\GW, (\beta,n)}\left( \lambda, \mu, \nu \right) (-1)^{|\nu| + \ell(\nu)} \frac{\prod_i \nu_i}{|\Aut(\nu)|}
Z^{(S \times \p^1, S_{0}), \std}_{\GW, (0,n)}\left( \tau_0(\omega D) \middle| \nu^{\vee} \right) \\
+ 
\sum_{\nu} 
Z^{(S \times \p^1, S_{0,1,\infty}), \std}_{\GW, (0,n)}\left( \lambda, \mu, \nu \right) (-1)^{|\nu| + \ell(\nu)} \frac{\prod_i \nu_i}{|\Aut(\nu)|}
Z^{(S \times \p^1, S_{0})}_{\GW, (\beta,n)}\left( \tau_0(\omega D) \middle| \nu^{\vee} \right).
\end{multline}
We have the straightforward evaluation
\[
Z^{(S \times \p^1, S_{0}), \std}_{\GW, (0,n)}\left( \tau_0(\omega D) \middle| \nu^{\vee} \right)
=
\begin{cases}
\int_{S} \gamma D & \text{ if } \nu = (1,\gamma)(1,\pt)^{n-1} \\
0 & \text{ if } \nu = (2,\pt) (1,\pt)^{n-2}
\end{cases}
\]
which gives us
\[
\sum_{\nu}  (-1)^{|\nu| + \ell(\nu)} \frac{\prod_i \nu_i}{|\Aut(\nu)|}
Z^{(S \times \p^1, S_{0}), \std}_{\GW, (0,n)}\left( \tau_0(\omega D) \middle| \nu^{\vee} \right) \cdot \nu = \frac{1}{(n-1)!} ( (1,D)(1,1)^{n-1}) = D.
\]
Moreover, in the second summand on the right of \eqref{sdasf} we must have $\nu = (1,1)^{n}$ for dimension reasons.
Using Proposition~\ref{prop:std evaluation} the claim follows.
\end{proof}

For primitive $\beta$ the second term on the right of the proposition is known:
\begin{prop}[\cite{K3xP1}] \label{K3xP1 lemma}
If $\beta \in H_2(S,\BZ)$ is primitive, then
\[
Z^{(S \times \p^1, S_{0})}_{\GW, (\beta,n)}( (1,\pt)^n | \tau_0(\omega D) ) = (\beta,D) \mathrm{Coeff}_{q^{\beta^2/2}}\left( \frac{\mathbf{G}^n(z,q)}{\Theta^2(z,q) \Delta(q)} \right),
\]
where $\mathbf{G}(z,q) = -\Theta(z,\tau)^2 D_z^2 \log(\Theta(z,\tau))$ with $D_z = \frac{d}{dz}$ and $q=e^{2\pi i \tau}$.
\end{prop}

\section{Holomorphic anomaly equations: $(K3 \times C, K3_z)$}
\label{sec:HAE on K3}

\subsection{Overview}
In this section we prove that the natural
generating series of Gromov-Witten invariants of
$(S \times C, K3_z)$
for an elliptic K3 surface $S$ in primitive classes,
are quasi-modular forms and satisfy a holomorphic anomaly equation (Theorem~\ref{thm:rel HAE}).
The idea is straightforward: We apply the product formula in Gromov-Witten theory
and use the corresponding results from the Gromov-Witten theory of K3 surfaces
which were proven in \cite{MPT,HAE}.

The details require some work: 
First in Section~\ref{subsec:Preliminaries HAE K3xC} we introduce a special set of disconnected invariants labeled by '$\sharp$'
which is well-adapted to the holomorphic anomaly equation.
In Section~\ref{subsec:quasi-modularity K3} and Section~\ref{subsec:HAE K3}
we recall the quasi-modularity and holomorphic anomaly equation for K3 surfaces in this convention. In Section~\ref{subsec:conclusion HAE K3xC} we then state and prove Theorem~\ref{thm:rel HAE} using the product formula and by a careful application of the splitting formulas and the new boundary restriction formulas introduced in Section~\ref{sec:Rel GW theory}.

\subsection{Preliminaries} \label{subsec:Preliminaries HAE K3xC}
To state the holomorphic anomaly equations we will need another convention
for disconnected Gromov-Witten invariants.
Let $\pi : X \to B$ be an elliptic fibration and let 
\[ \Mbar_{g,n}^{\sharp}(X,\beta) \]
be the moduli space of stable maps $f : C \to X$ from possibly disconnected
curves of genus $g$ in class $\beta$, with the following requirement:

($\sharp$) For every connected component $C' \subset C$ 
at least one of the following holds:
\begin{enumerate}
    \item[(i)] $\pi \circ f|_{C'}$ is non-constant, or 
    \item[(ii)] $C'$ has genus $g'$ and carries $n'$ markings with $2g'-2+n' > 0$.
\end{enumerate}

Parallel definitions apply to relative targets $(X,D)$ admitting an elliptic fibration to a pair $(B,A)$,
moduli spaces of rubber stable maps, etc. We will denote the invariants defined from moduli satisfying condition ($\sharp$) by a supscript $\sharp$.

\subsection{Quasi-modularity} \label{subsec:quasi-modularity K3}
Let $\pi : S \to B \cong \p^1$ be an elliptic K3 surface with a section,
let $B,F$ denote the class of the section and a fiber respectively,
and set $W = B+F$.
For any tautological class $\taut \in \tau^{\ast} R^{\ast}(\Mbar_{g,n})$ (or $\taut = 1$ in the unstable cases $2g-2+N \leq 0$) and $\gamma_i \in H^{\ast}(S)$ 
consider (or recall from \eqref{Fg}) the generating series
\[
F^{S}_{g}(\taut; \gamma_1, \ldots, \gamma_N) = \sum_{d \geq -1} \left\langle \taut ; \gamma_1, \ldots, \gamma_N \right\rangle^{S}_{g, W+dF} q^d.
\]
\begin{thm}[{\cite{MPT},\cite[Sec.4.6]{BOPY}}] \label{thm:MPT}
For $\wt$-homogeneous classes $\gamma_i \in H^{\ast}(S)$, we have
\[ F^{S}_{g}(\taut; \gamma_1, \ldots, \gamma_N) \in \frac{1}{\Delta(q)} \QMod_{s} \]
for $s = 2g + N + \sum_i \wt(\gamma_i)$.
\end{thm}

Consider the generating series of disconnected invariants (for the $\sharp$-condition):
\[ 
F^{S, \sharp}_{g}(\taut; \gamma_1, \ldots, \gamma_N) 
=
\sum_{d \geq -1} q^d
\int_{[ \Mbar_{g,N}^{\sharp}(S,W+dF) ]^{vir}}
\pi^{\ast}(\taut)
\prod_{i=1}^{N} \ev_i^{\ast}(\gamma_i).
\]
\begin{cor} \label{cor:MPT}
For $\wt$-homogeneous classes $\gamma_i \in H^{\ast}(S)$, we have
\[ F^{S, \sharp}_{g}(\taut; \gamma_1, \ldots, \gamma_N) \in \frac{1}{\Delta(q)} \QMod_{s} \]
for $s = 2g + N + \sum_i \wt(\gamma_i)$.
\end{cor}
\begin{proof}
Recall that the standard virtual class satisfies
\[
[ \Mbar_{g,n}(S,0) ]^{\std}
=
\begin{cases}
[ \Mbar_{0,n} \times S] & \text{ if } g=0 \\
c_2(S) \cap [ \Mbar_{1,n} \times S] & \text{ if } g=1 \\
0 & \text{ if } g \geq 2.
\end{cases}
\]
If an invariant
\[ \int_{ [ \Mbar_{g,n}(S,0) ]^{\std} } \pi^{\ast}(\taut) \prod_i \ev_i^{\ast}(\gamma_i) \]
is to contribute, we must have:
\begin{itemize}[itemsep=0pt]
\item $g=0$ and $\sum_i \wt(\gamma_i) = 2-n$
\item $g=1$ and $\sum_i \wt(\gamma_i) = -n$.
\end{itemize}
In both cases we have
\[ -2 + 2g + n + \sum_i \wt(\gamma_i) = 0. \]

Now, if a connected components of $\Mbar^{\sharp}_{g,N}(S, \beta)$ contributes non-trivially to the disconnected Gromov-Witten invariant,
then by a second-cosection argument the component
must parametrize stable maps $f : C \to S$ which are non-constant only on one component $C'$.
Let $g', N'$ be the genus and number of markings on $C'$.
The above computation shows that
\[ 2g+N+\sum_{i=1}^{N} \wt(\gamma_i) = 2g'+N'+\sum_{j=1}^{N'} \wt(\gamma_{i_j}) \]
where $i_j$ are the indices of marked points on $C'$.
The claim hence follows from Theorem~\ref{thm:MPT}.
\end{proof}

\subsection{Holomorphic anomaly equation}
\label{subsec:HAE K3}
We state the holomorphic anomaly equation for K3 surfaces in primitive classes.

\begin{thm}[\cite{HAE}]
\begin{align*}
\frac{d}{d G_2} F^S_{g}(\taut; \gamma_1, \ldots, \gamma_r)
= & 
F^S_{g-1}(\taut'; \gamma_1, \ldots, \gamma_r, \Delta_{B} ) \\
& + 2 \sum_{\substack{g=g_1 + g_2 \\ \{ 1, \ldots, r \} = A \sqcup B }} F^S_{g_1}(\taut_1 ; \gamma_A, \Delta_{B,1}) F_{g_2}^{S,\std}(\taut_2 ; \gamma_B , \Delta_{B,2} ) \\
& - 2 \sum_{i=1}^{r} F^S_{g}( \psi_i \cdot \taut ; \gamma_1, \ldots, \gamma_{i-1}, \pi^{\ast} \pi_{\ast} \gamma_i, \gamma_{i+1}, \ldots, \gamma_r) \\
& - \sum_{a,b} (g^{-1})_{ab} T_{e_a} T_{e_b} F^S_{g}(\taut; \gamma_1, \ldots, \gamma_r)
\end{align*}
where we follow the notation of Conjecture~\ref{Conj:HAE} and moreover
\begin{itemize}
\item $\Delta_{B,1}, \Delta_{B,2}$ stands for summing over the K\"unneth decomposition of the diagonal class
$\Delta_B \in H^{\ast}(B \times B)$, and we have suppressed the pullback to $S \times S$,
\end{itemize}
\end{thm}

This immediately yields the following for the series of disconnected invariants (compare with \cite[Sec.3.2]{OPix2} for a similar case).
\begin{cor} \label{cor:HAE}
\begin{align*}
\frac{d}{dG_2} F^{S, \sharp}_{g}(\taut; \gamma_1, \ldots, \gamma_r) 
& = F^{S, \sharp}_{g-1}(\taut'; \gamma_1, \ldots, \gamma_r, \Delta_{B})  \\
& - 2 \sum_{i=1}^{r} F^{S, \sharp}_{g}(\psi_i \cdot \taut ; \gamma_1, \ldots,
\gamma_{i-1}, \pi^{\ast} \pi_{\ast} \gamma_i, \gamma_{i+1}, \ldots, \gamma_r) \\
& - \sum_{a,b} (g^{-1})_{ab} T_{e_a} T_{e_b} F^{S, \sharp}_{g}(\taut; \gamma_1, \ldots, \gamma_r).
\end{align*}
\end{cor}
%
\begin{example}
Instead of the proof (which is straightforward) let us consider a concrete example
that highlights all the main points. Consider the series
\[ F_0^S(W,F,F) = F_0^{S, \sharp}(W,F,F) = q \frac{d}{dq} \frac{1}{\Delta(q)}. \]
We compute in three different ways the $G_2$-derivative. First directly:
\[
\frac{d}{d G_2} F_0^S(W,F,F) = \left[ \frac{d}{dG_2}, q \frac{d}{dq} \right] \frac{1}{\Delta(q)} = -2 \cdot (-12) \frac{1}{\Delta(q)} = 24 \frac{1}{\Delta(q)}.
\]
Second, by the holomorphic anomaly equations for the connected series:
\[
\frac{d}{dG_2} F_0^S(W,F,F)
=
2 \cdot 2 \cdot F_0^S(F, U_1) F_0^{S,\std}(U_2, W, F)
+ 20 F_0^S(F,F,F) = 24 \frac{1}{\Delta(q)}
\]
where the extra factor $2$ comes from choosing which of the two $F$'s goes to the two factors.
Third, by the disconnected holomorphic anomaly equation:
\[
\frac{d}{dG_2} F_0^{S,\sharp}(W,F,F)
= F_{-1}^{S,\sharp}(W,F,F, \Delta_{B})
-2 F_0^{S,\sharp}( \psi \cdot 1,F,F)
+ 20 F_0^{S,\sharp}(F,F,F) 
= (6 - 2 + 20) \frac{1}{\Delta(q)} \]
where we have used that:
\begin{gather*}
F_{-1}^{S,\sharp}(W,F,F, \Delta_{B})
=
2 F_{-1}^{S,\sharp}(W,F,F, F, 1)
=
6 F_{0}^S(F,F) F_0^{\std}(W,F,1) = 6 \frac{1}{\Delta(q)}.\\
-2 F_0^{S,\sharp}( \psi \cdot 1,F,F)
=
-2 \left( 24 \int_{\Mbar_{1,1}} \psi_1 \right) F_0^S(F,F).
\end{gather*} \qed
\end{example}

\subsection{Relative geometry $(S \times C, S_z)$} \label{subsec:conclusion HAE K3xC}
Consider the relative geometry:
\begin{equation} (S \times C, S_{z}), \quad z=(z_1, \ldots, z_N), \quad S_{z} = \bigsqcup_i S \times \{ z_i \}, \label{rel geometry2} \end{equation}
where we assume that that the pair $(C,z_1, \ldots, z_N)$ is stable, i.e. $2g-2+N>0$.
Define the generating series of relative invariants satisfying the $\sharp$ condition:
\[
F_{g}^{(S \times C, S_{z}), \sharp}(\lambda_1, \ldots, \lambda_N | \gamma )
=
\sum_{d \geq -1} q^d 
\big\langle \, \lambda_1, \ldots, \lambda_N \big| \gamma \big\rangle^{(S \times C, S_{z}), \sharp}_{g, (W+dF,n)}.
\]
where $\lambda_i$ are $H^{\ast}(S)$-weighted partitions and $\gamma \in H^{\ast}( (S \times C,S_z)^r)$.
Similarly, we have the corresponding series of reduced rubber invariants, see Section~\ref{sec:Rel GW of K3xC}.

We also also require the non-reduced invariants:
\[
F_{g}^{(S \times C, S_{z}), \sharp, \std}(\lambda_1, \ldots, \lambda_N | \gamma )
=
\big\langle \, \lambda_1, \ldots, \lambda_N \big| \gamma \big\rangle^{(S \times C, S_{z}), \sharp, \std}_{g, (0,n)}.
\]

\begin{thm} \label{thm:rel HAE}
(a) For cohomology weighted partitions $\lambda_i = (\lambda_{i,j}, \delta_{i,j})$ where $\delta_{i,j} \in H^{\ast}(S)$ are  $\wt$-homogeneous
we have
\[ F_{g}^{(S \times C, S_{z}), \sharp}(\lambda_1, \ldots, \lambda_N ) \in \frac{1}{\Delta(q)} \QMod_s. \]
where $s = 2g + \sum_{i=1}^{N} \ell(\lambda_i) + \sum_{i,j} \wt(\delta_{i,j})$.

(b) We have the holomorphic anomaly equation
\begin{align*}
& \frac{d}{dG_2} F_{g}^{(S \times C, S_{z}), \sharp}(\lambda_1, \ldots, \lambda_N ) \\
= \ & 
F_{g-1}^{(S \times C, S_{z}), \sharp}(\lambda_1, \ldots, \lambda_N| \Delta^{\rel}_{(B \times C, B_{z})} ) \\
& + 2 \sum_{i=1}^{N} \sum_{\substack{ m \geq 0 \\ g = g_1 + g_2 + m }}
\sum_{\substack{ b , b_1, \ldots, b_m \\ \ell , \ell_1, \ldots, \ell_m}}
\frac{\prod_{i=1}^{m} b_i}{m!}
 \\
&
\left(
\begin{array}{l}
F_{g_1}^{(S \times \p^1,S_{0,\infty}),\sim, \sharp,\std}\left( \lambda_i, \big( (b, \Delta_{B,\ell}) , (b_i, \Delta_{S, \ell_i})_{i=1}^{m}\big)  \right) \\
\qquad \qquad \cdot F_{g_2}^{(S \times C , S_{z}), \sharp}\big( \lambda_1, \ldots, \lambda_{i-1},  \big( (b, \Delta_{B, \ell}^{\vee}), (b_i, \Delta^{\vee}_{S, \ell_i})_{i=1}^{m} \big), \lambda_{i+1}, \ldots \lambda_N \big) \\
+ F_{g_1}^{(S \times \p^1,S_{0,\infty}), \sim, \sharp}\left( \lambda_i, \big( (b, \Delta_{B,\ell}) , (b_i, \Delta_{S, \ell_i})_{i=1}^{m}\big)  \right) \\
\qquad \qquad \cdot F_{g_2}^{(S \times C , S_{z}), \sharp, \std}\big( \lambda_1, \ldots , \lambda_{i-1}, \big( (b, \Delta_{B, \ell}^{\vee}), (b_i, \Delta^{\vee}_{S, \ell_i})_{i=1}^{m} \big), \lambda_{i+1}, \ldots, \lambda_N \big)
\end{array}\right) \\
& - 2 \sum_{i=1}^{N} \sum_{j=1}^{\ell(\lambda_i)} 
F_{g}^{(S \times C, S_{z}), \sharp}(\lambda_1, \ldots, \lambda_{i-1}, \psi_{i,j}^{\rel} \cdot \lambda_i^{(j)}, \lambda_{i+1}, \ldots \lambda_N ) \\
& - \sum_{a,b} (g^{-1})_{ab} T_{e_a} T_{e_b} F_{g}^{(S \times C, S_{z}), \sharp}(\lambda_1, \ldots, \lambda_N ).
\end{align*}
Here, the $b, b_1, \ldots, b_m$ run over all positive integers such that
$b+ \sum_i b_i = n$, and the $\ell, \ell_i$ run over the splitting of the diagonals of $B$ and $S$ respectively:
\[ 
\Delta_B = \sum_{\ell} \Delta_{B,\ell} \otimes \Delta_{B, \ell}^{\vee}, \quad \quad
 \forall i \colon\  \Delta_{S} = \sum_{\ell_i} \Delta_{S, \ell_i} \otimes \Delta_{S, \ell_i}^{\vee}.
\]
Moreover, $\lambda_{i}^{(j)}$ is the weighted partition $\lambda_{i}$ but with $j$-th weight $\delta_{ij}$ replaced by $\pi^{\ast} \pi_{\ast}( \delta_{ij})$.
\end{thm}
\begin{proof}
Consider a stable map $f : \Sigma \to (S \times C)[k]$ parametrized by the moduli space
$\Mbar^{\sharp}_{g,(W+dF,n)}((S \times C,S_z), \vec{\lambda})$.
In order for the connected component of the moduli space containing $f$ to contribute non-trivially to the Gromov-Witten invariant,
there must be  precisely one connected component $\Sigma' \subset \Sigma$ where $f$ is of non-zero degree over the K3 surface $S$.
Moreover, we claim that $f|_{\Sigma'}$ in this case is also of non-zero degree over $C$.
Indeed if not, then the remaining components yields a factor of
\[ \big\langle \, \lambda_1, \ldots, \lambda_N \big\rangle^{(S \times C, S_{z}), \sharp, \std}_{g', (0,n)}, \]
which have to vanish for dimension reasons (since the standard virtual class is dimension one less than the reduced virtual class
and the degree of the insertions $\lambda_1, \lambda_2, \lambda_3$ are chosen to sum up to the degree of the reduced virtual class).
Since $(C, z)$ is stable, it follows that $\Sigma'$ satisfies $2g(\Sigma') -2 + n(\Sigma') > 0$, so its stabilization is well-defined.
Similarly, 
if $\Sigma' \subset \Sigma$ is a connected component whose degree over the K3 surface $S$ is trivial, then either $2g(\Sigma')-2+n(\Sigma) > 0$ by assumption of the moduli space,
or the degree over $C$ is non-trivial. In the latter case by the stability of $(C,z)$ we have that $\Sigma'$ has again at least $N$ special points and genus $\geq g(C)$, hence $\Sigma'$ and its markings defines a stable curve.
Note also since we have no interior markings, there are no contributions from contracted genus $g \geq 2$ components.
Let $\Mbar^{\sharp, \mathrm{contr}}_{g,(W+dF,n)}( (S \times C,S_z), \vec{\lambda})$ be the union of connected components
which have a non-trivial contribution,
where we have written $\vec{\lambda} = (\vec{\lambda}_1, \ldots, \vec{\lambda}_N)$. 
We have shown that there exists a commutative diagram:
\[
\begin{tikzcd}
\Mbar^{\sharp, \mathrm{contr}}_{g,(W+dF,n)}( (S \times C,S_z), \vec{\lambda}) \ar{r}{q} \ar{d} &  \Mbar^{\sharp}_{g, \sum_i \ell(\lambda_i)}(S, W+dF) \ar{d}\\
\Mbar'_{g,n}( (C,z), \vec{\lambda}) \ar{r}{\pi}
& \Mbar_{g,n}'
\end{tikzcd} 
\]
where $\Mbar'_{g,n}(X,D)$ is the moduli space of disconnected relative stable maps where each connected component of the source is stable,
and $\Mbar'_{g,n}$ is simply the moduli space of disconnected stable curves (where each connected component is stable).

Recall from \eqref{I class} the class
\[
I^{(C, z), \prime}_{g,n}( \vec{\lambda}\, |\, \gamma)
=
\pi_{\ast}
\left( \ev^{\ast}(\gamma) [ \Mbar_{g,r,\beta}((C,z), \vec{\lambda}) ]^{\vir} \right).
\]
Then applying the product formula of \cite{Beh, LQ} we conclude that:
\[
F_{g}^{(S \times C, S_{z}), \sharp}(\lambda_1, \ldots, \lambda_N )
=
F_{g}^{S, \sharp}\Big( I^{(C, z), \prime}_{g,n}( \vec{\lambda}) ; \prod_{i=1}^{\ell(\lambda_i)} \prod_j \delta_{i,j} \Big)
\]
The first claim hence follows from Corollary~\ref{cor:MPT}.

For the second claim we apply the holomorphic anomaly equation of Corollary~\ref{cor:HAE}.
Let
\[ \iota : \Mbar'_{g-1,n+2} \to \Mbar'_{g,n} \]
be the morphism that glues the $(n+1)$-th and $(n+2)$-th marked point.
By an application of Proposition~\ref{prop:restriction to boundary} 
we then have:
\begin{align*}
&  \iota^{\ast} I^{(C, z), \prime}_{g,n}( \vec{\lambda} ) 
=
I_{g-1,\beta}^{(C,z),\prime}\left( \vec{\lambda} \middle| \Delta_{(C,z)}^{\text{rel}} \right) 
 +
\sum_{i=1}^{N}
\sum_{\substack{ m \geq 0 \\ g = g_1 + g_2 + m }}
\sum_{\substack{ b ,  b_1, \ldots, b_m \\ \ell, \ell_1, \ldots, \ell_m}} \frac{\prod_{i=1}^{m} b_i}{m!} \cdot 
\\
& \quad \quad \quad \begin{array}{r} 
\Bigg\{ \xi_{\ast} j^{\ast} \Bigg[
J^{(C,z), \bullet}_{g_1, \beta'}\Big( \lambda_1, \ldots, \lambda_{i-1}, \big( b, b_1, \ldots, b_m \big), \lambda_{i+1}, \ldots, \lambda_N \Big) \\
\boxtimes 
J^{(\p^1, \{ 0, z_i \} ), \bullet, \sim}_{g_2, \alpha}\Big( \big( b, b_1, \ldots, b_m \big) , \lambda_i \Big)
\Bigg] + (\text{reversed}) \Bigg\}
\end{array} 
\end{align*}
where (reversed) stands for the same term as before but with the role of the markings $(n+1)$ and $(n+2)$ reversed,
and the rest of the notation is as in Proposition~\ref{prop:restriction to boundary} (except that we do not require the glued curve to be connected).
Since only the $(n+1, n+2)$-th marked points are not glued, we exclude precisely those components of the moduli space
where there is a totally ramified morphism from a genus $0$ component to rubber $(\p^1, 0 \sqcup \infty)$ which is ramified over $0$ by some relative marking $\lambda_{i,j}$ and over $\infty$ by $b$ (corresponding to the marking labeled $n+1$ or $n+2$).
Applying the product formula in reverse we hence find that
\[ F_{g-1}^{S, \sharp}\Big( \iota^{\ast} I^{(C, z), \prime}_{g,n}( \vec{\lambda} )  ; \prod_{i=1}^{\ell(\lambda_i)} \prod_j \delta_{i,j} \Big) \]
accounts for precisely the first two terms on the right of part (b) of Theorem~\ref{thm:rel HAE},
except for the components where we have a contribution from a totally ramified map to a bubble attached to the marking $b$.

The second term on the right of Corollary~\ref{cor:HAE} is
\[ - 2 \sum_{i=1}^{r} F^{S, \sharp}_{g}(I^{(C, z), \bullet}_{g,n}( \vec{\lambda} )  ; \gamma_1, \ldots,
\gamma_{j-1}, \psi_i \cdot \pi^{\ast} \pi_{\ast} \gamma_j, \gamma_{j+1}, \ldots, \gamma_r) \]
where we write $(\gamma_1, \ldots, \gamma_r) = ( \delta_{ij} )_{i,j}$.
Again we apply the product formula in reverse. For that we need to compare the psi-classes $\psi_i$ on the domain and target of the morphism:
\[ \Mbar^{\sharp, \mathrm{contr}}_{g,(W+dF,n)}( (S \times C,S_z), \vec{\lambda}) \xrightarrow{q} \Mbar^{\sharp}_{g, \sum_i \ell(\lambda_i)}(S, W+dF). \]
Precisely, we have
\[ q^{\ast}(\psi_{i,j}) = \psi_{i,j}^{\text{rel}} - D \]
where $D$ is the virtual boundary divisor parametrizing splittings of maps $f : C \to X[k]$
where the relative marking $\lambda_{i,j}$ lies on a genus $0$ component mapping entirely into the bubble, such that the underlying curve is contracting after forgetting the map to $(C,z)$.
We hence obtain precisely the third tirm in part (b) of the claim,
plus the contribution we were missing in the first $2$ terms.

Finally, the third term in Corollary~\ref{cor:HAE} yields precisely part (b) in our claim.
This completes the proof.
\end{proof}

\begin{rmk}
The holomorphic anomaly equation of Theorem~\ref{thm:rel HAE} is a version (for reduced virtual classes) of the holomorphic anomaly equation
conjectured for the relative Gromov-Witten theory of elliptic fibrations in Conjecture D of \cite{OPix2}.
The form in \cite{OPix2} is more natural, but requires more notation (for once it is defined on the cycle level).
Theorem~\ref{thm:rel HAE} is then a special case of the following statement:
if the holomorphic anomaly equation (in the form of \cite[Conjecture B]{OPix2}) holds for an elliptic fibration $S \to B$,
then for any relative pair $(X,D)$ the holomorphic anomaly equation holds for the elliptic fibration $S \times X \to B \times X$ relative to $S \times D \to B \times D$ (in the form of \cite[Conjecture D]{OPix2}). \qed
\end{rmk}

\section{Holomorphic anomaly equations: Primitive case} \label{sec:HAE primitive case}
\subsection{Overview}
Let $S \to B$ be an elliptic K3 surface and
recall the generating series
\[
F^{S^{[n]}}_{g}(\taut; \gamma_1, \ldots, \gamma_N) = \sum_{d \geq -1} \sum_{r \in \mathbb{Z}} \left\langle \taut ; \gamma_1, \ldots, \gamma_N \right\rangle^{S^{[n]}}_{g, W+dF+rA} q^d (-p)^r,
\]
where $W=B+F$  and $B,F$ are the section and fiber class. 
The following is the conjectural quasi-Jacobi form property and holomorphic anomaly equation in the special case of primitive classes.
We follow parallel notation as in Conjecture~\ref{Conj:HAE}.

\begin{conj} \label{conj:fg HAE prim}
(a) For $\wt$-homogeneous classes $\gamma_i$, we have
\[ F^{S^{[n]}}_{g}(\taut; \gamma_1, \ldots, \gamma_N) \in 
\frac{1}{\Delta(q)}
\QJac_{k, n-1} \]
where $k=n(2g-2+N) + 2 + \sum_i \wt(\gamma_i)$.

(b) Assuming part (a) we have
\begin{align*}
\frac{d}{d G_2} F^{S^{[n]}}_{g}(\taut; \gamma_1, \ldots, \gamma_N)
= & 
F^{S^{[n]}}_{g-1}(\taut'; \gamma_1, \ldots, \gamma_N, U) \\
& + 2 \sum_{\substack{g=g_1 + g_2 \\ \{ 1, \ldots, N \} = A \sqcup B }} F^{S^{[n]}}_{g_1}(\taut_1 ; \gamma_A, U_1) F_{g_2}^{S^{[n]},\std}(\taut_2 ; \gamma_B , U_2 ) \\
& - 2 \sum_{i=1}^{N} F^{S^{[n]}}_{g}( \psi_i \cdot \taut ; \gamma_1, \ldots, \gamma_{i-1}, U(\gamma_i), \gamma_{i+1}, \ldots, \gamma_N) \\
& - \sum_{a,b} (g^{-1})_{ab} T_{e_a} T_{e_b} F^{S^{[n]}}_{g}(\taut; \gamma_1, \ldots, \gamma_N).
\end{align*}
\end{conj} 

\vspace{5pt}
In this section we prove the following:
\begin{thm} \label{thm:prim HAE}
Conjecture~\ref{conj:fg HAE prim} holds
in case $g=0$ and $N \leq 3$.
\end{thm}

The proof below proceeds in three steps.
Afte reducing to $N=3$ and $\taut=1$, 
the GW/Hilb correspondence (Theorem~\ref{thm:GW/Hilb}) implies the following basic statement, see \eqref{xx4tewt}:
\begin{equation*}
F^{S^{[n]}}_{0}(\lambda_1, \lambda_2, \lambda_3)
=
\sum_{g \in \BZ} z^{2g-2 -n + \sum_i l(\lambda_i) }
(-1)^{g-1+\sum_i l(\lambda_i)} F_g^{(S \times \p^1,S_{0,1,\infty}), \sharp}\left( \lambda_1, \lambda_2, \lambda_3 \right)
\end{equation*}
under the variable change $p=e^z$.
By Theorem~\ref{thm:rel HAE} of Section~\ref{sec:HAE on K3} we know that each
$F_g^{(S \times \p^1,S_{0,1,\infty}), \sharp}\left( \lambda_1, \lambda_2, \lambda_3 \right)$ is a quasi-modular form satisfying a holomorphic anomaly equation.
Moreover, the Hilbert scheme series $F^{S^{[n]}}_{0}(\lambda_1, \lambda_2, \lambda_3)$ on the left
satisfies
the structure described in Proposition~\ref{prop:constrains monodromy}.
Our main work is then to turn these two inputs
into the quasi-Jacobi form property and the holomorphic anomaly equation for quasi-Jacobi forms for the Hilbert scheme series.
In Section~\ref{subsec:QJ prop} we first discuss that the left hand side is a quasi-Jacobi form.
Then in Section~\ref{subsec:reduction} we reduce the holomorphic anomaly equation
for the left hand side to an identity of the correspondence $z$-series.
This is done by using Lemma~\ref{lemma:z^r coefficient}
on the comparison of the $G_2$-holomorphic anomaly equation for quasi-Jacobi  forms
with the factorwise $G_2$-holomorphic anomaly equation on the $z$-expansion.
Finally, the required identity is checked in Section~\ref{subsec:conclusion HAE}
in a longer and technical 4-step argument.

\subsection{Quasi-Jacobi form property}
\label{subsec:QJ prop}
We start with the quasi-Jacobi form part of Theorem~\ref{thm:prim HAE}.
\begin{prop} \label{prop:quasi-modularity}
Assume that $g=0$ and $N \leq 3$. For $\wt$-homogeneous classes $\gamma_i$, we have
\[ F^{S^{[n]}}_{g}(\taut; \gamma_1, \ldots, \gamma_N) \in 
\frac{1}{\Delta(q)}
\QJac_{s, n-1} \]
where $s=n(2g-2+N) + 2 + \sum_i \wt(\gamma_i)$.
\end{prop}

\begin{proof}
For $g=0$ and $N\leq 3$ we can take $\taut=1$.
By using the divisor equation the claim for $N \in \{ 0, 1, 2 \}$ reduces to case $N=3$. 
Consider three $H^{\ast}(S)$-weighted partitions,
\[ \lambda_i = \left( \lambda_{ij}, \delta_{ij} \right)_{j}, \quad i=1,2,3 \]
We argue in three steps.

\vspace{5pt}
\noindent
\textbf{Step 1.} Under the variable change $p=e^{z}$ the $z^r$ coefficient in $\Delta(q) \cdot F^{S^{[n]}}_{0}(\lambda_1, \lambda_2, \lambda_3)$ 
is a quasi-modular form of weight $r+n+2 + \sum_i \wt(\lambda_i)$.

\begin{proof}[Proof of Step 1]
By Theorem~\ref{thm:GW/Hilb} under the variable change $p=e^{z}$ we have
\[
F^{S^{[n]}}_{0}(\lambda_1, \lambda_2, \lambda_3)
=
\sum_{d \geq -1} Z^{(S \times \p^1,S_{0,1,\infty})}_{\GW, (W+dF,n)}\left( \lambda_1, \lambda_2, \lambda_3 \right) q^{d}
\]

Since $(\p^1, 0, 1, \infty)$ is stable and there are no interior markings, we have the inclusion
\[ \Mbar^{\sharp}_{g,(W+dF,n)}((S \times \p^1,S_{0,1,\infty}), \vec{\lambda}) \subset \Mbar^{\bullet}_{g,(W+dF,n)}((S \times \p^1,S_{0,1,\infty}), \vec{\lambda}) \]
and moreover, every connected component in the complement does not contribute to the Gromov-Witten invariant
since the obstruction theory will admit an extra cosection coming from stable maps with two components of the domain curve of non-trivial degree over $S$.
Hence we find that
\begin{equation} \label{xx4tewt} 
F^{S^{[n]}}_{0}(\lambda_1, \lambda_2, \lambda_3)
=
\sum_{g \in \BZ} z^{2g-2 -n + \sum_i l(\lambda_i) }
(-1)^{g-1+\sum_i l(\lambda_i)} F_g^{(S \times \p^1,S_{0,1,\infty}), \sharp}\left( \lambda_1, \lambda_2, \lambda_3 \right).
\end{equation}
By Theorem~\ref{thm:rel HAE} (a) the series
$\Delta(q) \cdot F_g^{(S \times \p^1,S_{0,1,\infty}), \sharp}\left( \lambda_1, \lambda_2, \lambda_3 \right)$
is a quasi-modular form of weight
\[ 2g + \sum_{i=1}^{3} \ell(\lambda_i) + \sum_{i,j} \wt(\delta_{i,j}). \]
Hence under $p=e^{z}$
the $z^r$ coefficient of $\Delta(q) F^{S^{[n]}}_{0}(\lambda_1, \lambda_2, \lambda_3)$ is a quasi-modular form of weight $r+s$ where
\[ s = \left( 2g+ \sum_i \ell(\lambda_i) + \sum_{i,j} \wt(\delta_{ij}) \right)
-
\left( 2g-2 - n + \sum_i \ell(\lambda_i) \right)
=
n+2 + \sum_i \wt(\lambda_i).
\]
\end{proof}

\vspace{5pt}
\noindent
\textbf{Step 2.} $\Delta(q) \cdot F^{S^{[n]}}_{0}(\lambda_1, \lambda_2, \lambda_3) \in \MQJac_{s, n-1}$, where $s=n+2 + \sum_i \wt(\lambda_i)$.

\begin{proof}[Proof of Step 2]
We argue by induction on the total weight of the insertions
\[ \sum_i \wt(\lambda_i) = L. \]
We assume that the claim holds for all insertions $\lambda'_i$ with $\sum_i \wt(\lambda_i') < L$.
By induction and Lemma~\ref{lemma:T action} 
we have
\begin{equation*}
\sum_{i=1}^{3} F^{S^{[n]}}_{0}(\lambda_1, \ldots, \lambda_{i-1}, T_{\delta} \lambda_i, \lambda_{i+1}, \ldots, \lambda_3) \  \in \MQJac_{s-1, n-1}.
\end{equation*}
As in Step 2 of the proof of Proposition~\ref{prop:constrains monodromy} we consider the integral with respect to $A$,
\[ \widetilde{F} = \sum_{i=1}^{3} \int
F^{S^{[n]}}_{0}(\lambda_1, \ldots, \lambda_{i-1}, T_{\delta} \lambda_i, \lambda_{i+1}, \ldots, \lambda_3) dA \]
which lies in $\MQJac_{s,n-1}$. Consider also the difference:
\[ F(p,q) = F^{S^{[n]}}_{0}(\lambda_1, \lambda_2, \lambda_3) - \widetilde{F}(p,q). \]
Then as shown in \eqref{F expression} there exists power series $f_i(q) \in \BQ[[q]]$ such that
\[ 
F(p,q) = 
\begin{cases}
\Delta^{-1}(q) \Theta^{2m}(p,q) \wp'(p,q) \sum_{i=2}^{m} f_i(q) \wp(p,q)^{m-i} & \text{ if } 3n + \sum_{i=1}^{3} \ell(\lambda_i) \text{ is even} \\
\Delta^{-1}(q) \Theta^{2m}(p,q) \sum_{i=0}^{m} f_i(q) \wp(p,q)^{m-i}
& \text{ if } 3n + \sum_{i=1}^{3} \ell(\lambda_i) \text{ is odd}.
\end{cases}
\]

By Step 1 (for the term $F^{S^{[n]}}_{0}(\lambda_1, \lambda_2, \lambda_3)$) and by 
Lemma~\ref{lemma:z^r coefficient} (for $\widetilde{F} \in \MQJac_{s,n-1}$) every $z^r$ coefficient of $F(p,q)$ is a quasi-modular form of weight $r+s$.
By Lemma~\ref{lemma:z expansion 1} or Lemma~\ref{lemma:z expansion 2} (depending on the parity of $3n + \sum_{i=1}^{3} \ell(\lambda_i)$)
the claim follows.
\end{proof}

\vspace{5pt}
\noindent
\textbf{Step 3.} $\Delta(q) \cdot F^{S}_{0}(\lambda_1, \lambda_2, \lambda_3) \in \QJac_{s, n-1}$, where $s=n+2 + \sum_i \wt(\lambda_i)$.

\begin{proof}[Proof of Step 3]
The function $F(z,\tau) = \Delta(q) \cdot F^{S}_{0}(\lambda_1, \lambda_2, \lambda_3)$
defines a meromorphic function $\BC \times \BH \to \BC$
which is holomorphic away from the lattice points $\frac{z}{2 \pi i} = \lambda \tau + \mu$ for all $\lambda, \mu \in \BZ$.

By Proposition~\ref{prop:constrains monodromy}(b) the expansion of $z$ around $z=0$ takes the form
\[
F(z,\tau) = \sum_{k \geq 0} f_k(\tau) z^k
\]
where $f_k(\tau)$ are quasi-modular forms.
This shows that $F(z,\tau)$ is holomorphic at $z=0$.

To check the other lattice points, we apply Lemma~\ref{elliptic transformation} which yields the transformation
\begin{align*} F(z + 2 \pi i (\lambda \tau + \mu),\tau) & = q^{-\lambda^2 m} p^{-2 \lambda m} e^{- \lambda \frac{d}{d A} } F(z,\tau)
\end{align*}
By Proposition~\ref{prop:constrains monodromy}(a) (the behaviour under $\frac{d}{dA}$) this equals:
\begin{align*}
& = q^{-\lambda^2 m} p^{-2 \lambda m}
\Delta(q) \cdot F^{S}_{0}(\taut; e^{-\lambda T_{\delta}} \lambda_1, e^{- \lambda T_{\delta}} \lambda_2, e^{-\lambda T_{\delta}} \lambda_3)
\end{align*}
Since $T$ is nilpotent there are only finitely many terms on the right hand side. Hence by Proposition~\ref{prop:constrains monodromy}(b) again,
the right hand side is holomorphic at $z=0$.
\end{proof}

\end{proof}

\subsection{Reduction} \label{subsec:reduction}
Recall the operator that takes the $G_2$-derivative of a power series in $z$ with coefficients quasi-modular forms factorwise:
\[ \left( \frac{d}{dG_2} \right)_z : \QMod((z)) \to \QMod((z)). \]
After having shown part (a) of Conjecture~\ref{conj:fg HAE prim}
we now reduce part (b) to a statement about the $z$-series of the $3$-point function:
\begin{prop} \label{prop:reduction}
Part (b) of Conjecture~\ref{conj:fg HAE prim} holds for $g=0$ and $N \leq 3$
if for any cohomology weighted partitions $\lambda_1, \lambda_2, \lambda_3$ we have
\begin{equation} \label{eq: prop:reduction}
\begin{aligned}
\left( \frac{d}{dG_2} \right)_z F_0^{S^{[n]}}( \lambda_1, \lambda_2, \lambda_3 )
& = 2 \Big( F_0^{S^{[n]}}(\lambda_1, U(\lambda_2 \lambda_3)) 
- F_0^{S^{[n]}}(U \lambda_1, \lambda_2 \lambda_3) \\
& \quad
+ F_0^{S^{[n]}}(\lambda_2, U(\lambda_1 \lambda_3)) 
- F_0^{S^{[n]}}( U\lambda_2, \lambda_1 \lambda_3) \\
& \quad
+ F_0^{S^{[n]}}(\lambda_3, U(\lambda_1 \lambda_2)) 
- F_0^{S^{[n]}}( U \lambda_3, \lambda_1 \lambda_2) \Big) \\
& \phantom{=} - \sum_{a,b} (G^{-1})_{ab} T_{e_a} T_{e_b} F_0^{S^{[n]}}(\lambda_1, \lambda_2, \lambda_3) \\
& \phantom{=} -2 z \left( F_0^{S^{[n]}}(T_{\delta} \lambda_1, \lambda_2, \lambda_3)
+ F_0^{S^{[n]}}(\lambda_1, T_{\delta} \lambda_2, \lambda_3)
+ F_0^{S^{[n]}}(\lambda_1, \lambda_2, T_{\delta} \lambda_3) \right) \\
& \phantom{=} -2 (n-1) z^2 F_0^{S^{[n]}}( \lambda_1, \lambda_2, \lambda_3 ).
\end{aligned}
\end{equation}
\end{prop}
\begin{proof} 
Part (b) is Conjecture~\ref{conj:fg HAE prim} is compatible under the
divisor equations, string equation, and restriction to boundary.
This can be proven parallel to \cite[Sec.2]{HAE} or \cite[Sec.3]{BB}.
Hence to prove part (b) it suffices to consider the case $\alpha=1$, $g=0$, $N=3$,
i.e. to prove the holomorphic anomaly equation for $F_0^{S^{[n]}}(\lambda_1, \lambda_2, \lambda_3)$.

One has that
\begin{multline*}
 \sum_{ \{ 1, \ldots, 3 \} = A \sqcup B } F_{0}^{S^{[n]}}(1 ; \lambda_A, U_1) F_{0}^{S^{[n]}, \std}(1 ; \lambda_B , U_2 ) \\
 =
F_0^{S^{[n]}}(\lambda_1, U(\lambda_2 \lambda_3)) + F_0^{S^{[n]}}(\lambda_2, U(\lambda_1 \lambda_3)) + F_0^{S^{[n]}}(\lambda_3, U(\lambda_1 \lambda_2)),
 \end{multline*}
and by expressing $\psi_i$ as boundary we also get:
\[
F_0^{S^{[n]}}( \psi_1; U \lambda_1, \lambda_2, \lambda_3 ) = F_0^{S^{[n]}}( U \lambda_1, \lambda_2 \lambda_3 ).
\]
Hence the equation that we need to prove is:
\begin{equation} \label{w0-sdif}
\begin{aligned}
\frac{d}{dG_2} F_0^{S^{[n]}}( \lambda_1, \lambda_2, \lambda_3 )
& = 2 \Big( F_0^{S^{[n]}}(\lambda_1, U(\lambda_2 \lambda_3)) 
- F_0^{S^{[n]}}(U \lambda_1, \lambda_2 \lambda_3) \\
& \quad
+ F_0^{S^{[n]}}(\lambda_2, U(\lambda_1 \lambda_3)) 
- F_0^{S^{[n]}}( U\lambda_2, \lambda_1 \lambda_3) \\
& \quad
+ F_0^{S^{[n]}}(\lambda_3, U(\lambda_1 \lambda_2)) 
- F_0^{S^{[n]}}( U \lambda_3, \lambda_1 \lambda_2) \Big) \\
& \phantom{=} - \sum_{a,b} (G^{-1})_{ab} T_{e_a} T_{e_b} F_0^{S^{[n]}}(\lambda_1, \lambda_2, \lambda_3)
\end{aligned}
\end{equation}

We now apply the variable change $p=e^{z}$ and view $F_0( \lambda_1, \lambda_2, \lambda_3 )$ as a power series in $z$ with coefficients quasi-modular forms.
Since $F_0( \lambda_1, \lambda_2, \lambda_3 )$ are quasi-Jacobi forms of index $n-1$
by Lemma~\ref{lemma:z^r coefficient}, we have the following relation of Jacobi and factorwise $G_2$-derivative:
\begin{multline*} \left( \frac{d}{dG_2} \right)_z F_0^{S^{[n]}}( \lambda_1, \lambda_2, \lambda_3 ) = \frac{d}{dG_2} F_0^{S^{[n]}}( \lambda_1, \lambda_2, \lambda_3 )  \\
- 2 z \frac{d}{dA} F_0^{S^{[n]}}( \lambda_1, \lambda_2, \lambda_3 ) 
-2 z^2 (n-1) F_0^{S^{[n]}}( \lambda_1, \lambda_2, \lambda_3 ).  \end{multline*}
By Proposition~\ref{prop:constrains monodromy} we have that
\[ \frac{d}{d\A} F_0^{S^{[n]}}(\lambda_1, \lambda_2, \lambda_3)
=
F_0^{S^{[n]}}(T_{\delta} \lambda_1, \lambda_2, \lambda_3)
+ F_0^{S^{[n]}}(\lambda_1, T_{\delta} \lambda_2, \lambda_3)
+ F_0^{S^{[n]}}(\lambda_1, \lambda_2, T_{\delta} \lambda_3).
\]
Expressing the left hand side in \eqref{w0-sdif} in terms of $\left( \frac{d}{dG_2} \right)_z$ then yields the claim.
\end{proof}

\subsection{Conclusion} \label{subsec:conclusion HAE}
We aim to prove the holomorphic anomaly equation \eqref{eq: prop:reduction},
which by Proposition~\ref{prop:reduction} gives us the remaining part of Theorem~\ref{thm:prim HAE}.
We start with the expression given in \eqref{xx4tewt},
\begin{equation}
F^{S^{[n]}}_{0}(\lambda_1, \lambda_2, \lambda_3)
=
\sum_{g \in \BZ} z^{2g-2 -n + \sum_i l(\lambda_i) } (-1)^{g-1+\sum_i l(\lambda_i)} F_g^{(S \times \p^1,S_{0,1,\infty}), \sharp}\left( \lambda_1, \lambda_2, \lambda_3 \right).
\end{equation}
We will compute the factorwise $G_2$-derivative $\left( \frac{d}{dG_2} \right)_z$
using the holomorphic anomaly equation given in Theorem~\ref{thm:rel HAE},
and then match all the terms with the right hand side of \eqref{eq: prop:reduction}.

We analyse all the terms appearing the right hand side of Theorem~\ref{thm:rel HAE}.
We need to analyze four terms, which we do in a sequence of lemmata.

\begin{lemma}[\textbf{Term 1}]
\begin{multline*}
\sum_{g \in \BZ} z^{2g-2 -n + \sum_i l(\lambda_i) } (-1)^{g-1+\sum_i l(\lambda_i)} 
F_{g-1}^{(S \times \p^1,S_{0,1,\infty}), \sharp}\left( \lambda_1, \lambda_2, \lambda_3 \middle| \Delta^{\rel}_{(B \times C, B_{z})} \right) \\
= (2-2n) z^2 F_0^{S^{[n]}}( \lambda_1, \lambda_2, \lambda_3 )
\end{multline*}
\end{lemma}
\begin{proof}
Let $\beta = \beta_d = W+dF$.
By the splitting formula of Proposition~\ref{prop:splitting relative diagonal} applied in the reduced case we have
\begin{multline*}
\left\langle \lambda_1, \lambda_2, \lambda_3 \middle| \Delta^{\rel}_{(B \times C, B_{z})} \right\rangle_{g, (\beta,n)}^{(S \times \p^1,S_{0,1,\infty}), \sharp}
=
\left\langle \lambda_1, \lambda_2, \lambda_3 \middle| \Delta_{B \times C} \right\rangle_{g, (\beta,n)}^{(S \times \p^1,S_{0,1,\infty}), \sharp} \\
- \sum_{\substack{i \in \{ 1, 2, 3 \}, \mu \\ g_1+g_2=g + 1 - \ell(\mu)}}
\frac{\prod_i \mu_i}{|\Aut(\mu)|}
\left(
\begin{array}{c}
\big\langle \, \lambda_1, \ldots, \underbrace{\mu}_{i\text{-th}}, \ldots , \lambda_N \, \big\rangle^{(S \times \p^1,S_{0,1,\infty}), \sharp, \std}_{g_1,(0,n)} 
\big\langle \, \lambda_i, \mu^{\vee} \, \big| \, \Delta_{D} \big\rangle^{(S \times \p^1, S_{0,\infty}), \sharp, \sim}_{g_2,(\beta,n)} + \\
\big\langle \, \lambda_1, \ldots, \underbrace{\mu}_{i\text{-th}}, \ldots , \lambda_N \, \big\rangle^{(S \times \p^1,S_{0,1,\infty}), \sharp}_{g_1,(\beta,n)} 
\big\langle \, \lambda_i, \mu^{\vee} \, \big| \, \Delta_{D} \big\rangle^{(S \times \p^1, S_{0,\infty}), \sharp, \sim, \std}_{g_2,(0,n)}
\end{array}\right)
\end{multline*}

To analyze the first term above we now use the K\"unneth decomposition
\[ \Delta_{B \times C} = \Delta_{B} \cdot \Delta_{C} = (\omega_1 + \omega_2) (F_1 + F_2). \]
The moduli space
$\Mbar_{1,1,(0,0)}((S \times \p^1,S_{0,1,\infty}), \varnothing)$ is naturally isomorphic to $\Mbar_{1,1} \times S \times \p^1$
with virtual class given by
\[ e( H^1(\Sigma, f^{\ast}(T_{(S \times \p^1,S_{0,1,\infty})}^{\text{log}}) )) = e(T_{(S \times \p^1,S_{0,1,\infty})}^{\text{log}}) - \lambda_1 c_2( T_{(S \times \p^1,S_{0,1,\infty})}^{\text{log}}) \]
where we used the log tangent bundle
\[ T_{(S \times \p^1,S_{0,1,\infty})}^{\text{log}} = T_S \oplus T_{(\p^1, \{ 0, 1, \infty \})}^{\text{log}} = T_S \oplus \CO_{\p^1}(-1). \]
It follows that
\[ \langle \tau_0(\alpha) \rangle_{g=1, (0,0)}^{(S \times \p^1, S_{0,1,\infty})} =
\begin{cases}
0 & \text{ if } \alpha \in \{ 1, F \} \\
-1 & \text{ if } \alpha = \omega
\end{cases}
\]
Observe that under the ($\sharp$) convention we can have genus $1$ components that are contracted,
but since we only have two interior markings there can be no contracted genus $0$ component.
Moreover, genus $\geq 2$ contracted component are ruled out since the K3 virtual class vanishes.
Hence applying the divisor equation yields:
\begin{align*}
& \left\langle \lambda_1, \lambda_2, \lambda_3 \middle| \Delta_{B \times C} \right\rangle_{g, (\beta,n)}^{(S \times \p^1,S_{0,1,\infty}), \sharp} \\
=\ & 2 \left\langle \lambda_1, \lambda_2, \lambda_3 \middle| F, \omega \right\rangle_{g, (\beta,n)}^{(S \times \p^1,S_{0,1,\infty}), \sharp} \\
=\ & 2 \langle \tau_0(\omega) \rangle_{g=1, (0,0)}^{(S \times \p^1, S_{0,1,\infty})} \left\langle \lambda_1, \lambda_2, \lambda_3 \right\rangle_{g, (\beta,n)}^{(S \times \p^1,S_{0,1,\infty}), \bullet} \\
& \quad \quad + 2 \left( \int_{(\beta,n)} \omega \right)  \left\langle \lambda_1, \lambda_2, \lambda_3 \right\rangle_{g, (\beta,n)}^{(S \times \p^1,S_{0,1,\infty}), \bullet} \\
=\ & (-2 + 2n) \left\langle \lambda_1, \lambda_2, \lambda_3 \right\rangle_{g, (\beta,n)}^{(S \times \p^1,S_{0,1,\infty}), \bullet}.
\end{align*}

On the other hand, we have $\Delta_D = F_1 + F_2$, so we find
\[
\big\langle \, \lambda_i, \mu^{\vee} \, \big| \, \Delta_{D} \big\rangle^{(S \times \p^1, S_{0,\infty}), \sharp, \sim}_{g_2,(\beta,n)}
=
2 \big\langle \, \lambda_i, \mu^{\vee} \, \big| \, \tau_0(1) \tau_0(F) \big\rangle^{(S \times \p^1, S_{0,\infty}), \sharp, \sim}_{g_2,(\beta,n)}
\]
If the marked point carrying $\tau_0(1)$ lies on a component of a curve which remains stable after forgetting the marking,
i.e. where on the corresponding connected component of the moduli space the morphism forgetting the marking is welldefined,
then since the integrand is pulled back from the forgetful morphism, the contribution vanishes.
Alternatively, $\tau_0(1)$ lies on a contracted genus $1$ component, which yields the contribution
\[ \langle \tau_0(1) \rangle_{g=1, (0,0)}^{(S \times \p^1, S_{0,\infty}), \sim} \left\langle \lambda_i, \mu^{\vee} | \tau_0(F) \right\rangle_{g, (\beta,n)}^{(S \times \p^1,S_{0,\infty}), \bullet} \]
where since $\tau_0(1)$ stabilizes the rubber action, the second factor is non-rubber(!).
The first factor is non-zero, but the second factor vanishes by the product formula and the general vanishing (see e.g. \cite[Lemma 2]{HAE})
\[ \pi_{\ast} [ \Mbar_{g,r}(\p^1, \mu, \nu) ]^{\text{vir}} = 0 \]
for $\pi$ the forgetful morphism to $\Mbar_{g,r+\ell(\mu)+\ell(n)}$ whenever $2g-2+r+\ell(\mu) + \ell(\nu) > 0$.
The case where the rubber carries the standard virtual class is similar.

In summary we obtain that:
\[
\left\langle \lambda_1, \lambda_2, \lambda_3 \middle| \Delta^{\rel}_{(B \times C, B_{z})} \right\rangle_{g, (\beta,n)}^{(S \times \p^1,S_{0,1,\infty}), \sharp}
=
2 (n-1) \left\langle \lambda_1, \lambda_2, \lambda_3 \right\rangle_{g, (\beta,n)}^{(S \times \p^1,S_{0,1,\infty}), \bullet}. \]
Replacing $g$ by $g-1$, and summing over the genus, then yields
\begin{align*}
& \sum_{g \in \BZ} z^{2g-2 -n + \sum_i l(\lambda_i) }
(-1)^{g-1+\sum_i l(\lambda_i)} 
\left\langle \lambda_1, \lambda_2, \lambda_3 \middle| \Delta^{\rel}_{(B \times C, B_{z})} \right\rangle_{g-1, (\beta,n)}^{(S \times \p^1,S_{0,1,\infty}), \sharp} \\
& = z^2 (-1) \sum_{g \in \BZ} z^{2(g-1)-2 -n + \sum_i l(\lambda_i) }
(-1)^{(g-1)-1+\sum_i l(\lambda_i)} 
2 (n-1) \left\langle \lambda_1, \lambda_2, \lambda_3 \right\rangle_{g, (\beta,n)}^{(S \times \p^1,S_{0,1,\infty}), \bullet} \\
& = -2 (n-1) z^2 Z^{(S \times \p^1,S_{0,1,\infty})}_{\GW,(\beta,n)}( \lambda_1, \lambda_2, \lambda_3 ) \\
& = -2 (n-1) z^2 Z^{(S \times \p^1,S_{0,1,\infty})}_{\Hilb, (\beta,n)}( \lambda_1, \lambda_2, \lambda_3 )
\end{align*}
where we used the triangle of correspondences in the last step.
Summing now over the curve class $\beta_d$ then completes the lemma.
\end{proof}

For the second term we need first some preparation.

\begin{lemma} \label{lemma:Udecomp}
The class $U \in H^{\ast}(S^{[n]})$ has K\"unneth decomposition:
\[
U = \sum_{m \geq 0} \sum_{\substack{ b ; b_1, \ldots, b_m \\ \ell ; \ell_1, \ldots, \ell_m}}
(-1)^{m+n+1} \frac{\prod_{i=1}^{m} b_i}{m!}
\big( (b, \Delta_{B,\ell}) , (b_i, \Delta_{S, \ell_i})_{i=1}^{m}\big) \boxtimes
\big( (b, \Delta_{B, \ell}^{\vee}), (b_i, \Delta^{\vee}_{S, \ell_i})_{i=1}^{m} \big)
\]
where the $b, b_1, \ldots, b_m$ run over all positive integers such that
$b+ \sum_i b_i = n$, and the $\ell, \ell_i$ run over the splitting of the diagonals of $B$ and $S$ respectively:
\[ 
\Delta_B = \sum_{\ell} \Delta_{B,\ell} \otimes \Delta_{B, \ell}^{\vee}, \quad \quad
 \forall i \colon\  \Delta_{S} = \sum_{\ell_i} \Delta_{S, \ell_i} \otimes \Delta_{S, \ell_i}^{\vee}.
\]
\end{lemma}
\begin{proof}
Let $\Fq_i, \Fq_i'$ denote the Nakajima operators acting on the first and second copy of $S^{[j]} \times S^{[j]}$ respectively.
Then we have
\begin{align*}
U 
& = -\sum_{b > 0} \frac{1}{b^2}\Fq_b \Fq_{-b}( F_1 + F_2 ) \\
& = -\sum_{b > 0} \sum_{\lambda : |\lambda| = n-b} \frac{1}{b^2} (-1)^b \Fq_{b} \Fq_{b}'(F_1 + F_2) \Delta_{S^{[n-b]}} \\
& = -\sum_{b > 0} \sum_{\lambda : |\lambda| = n-b} \frac{1}{b^2} (-1)^b \Fq_{b} \Fq_{b}'(F_1 + F_2) \frac{(-1)^{|\ell|+\ell(\lambda)}}{ |\Aut(\lambda)| \prod_i \lambda_i } \Fq_{\lambda_1} \Fq_{\lambda_1}'(\Delta_S) \cdots \Fq_{\lambda_{\ell(\lambda)}} \Fq_{\lambda_{\ell(\lambda)}}'(\Delta_S) \\
& = \sum_{ m \geq 0} \sum_{b ; b_1, \ldots, b_m}
(-1)^{n+m+1} \frac{1}{m!} \frac{1}{b^2} \frac{1}{\prod_i b_i} \Fq_{b} \Fq_{b}'(F_1 + F_2) \Fq_{b_1} \Fq_{b_1}'(\Delta_S) \cdots \Fq_{b_m} \Fq_{b_m}'( \Delta_S ).
\end{align*}
By using Definition~\ref{defn:coh weighted to Nakajima} to rewrite this in terms of weighted partitions
yield the claim.
\end{proof}

\begin{lemma}[\textbf{Term 2a}]
\begin{multline} \label{2a:lhs}
2 \sum_{g \in \BZ} z^{2g-2 -n + \sum_i l(\lambda_i) } (-1)^{g-1+\sum_i l(\lambda_i)} 
\sum_{\substack{ m \geq 0 \\ g = g_1 + g_2 + m }}
\sum_{\substack{ b ; b_1, \ldots, b_m \\ \ell ; \ell_1, \ldots, \ell_m}}
\frac{\prod_{i=j}^{m} b_j}{m!} \\
F_{g_1}^{\sim, \sharp,\std}\left( \lambda_1, \big( (b, \Delta_{B,\ell}) , (b_i, \Delta_{S, \ell_i})_{i=1}^{m}\big)  \right)
F_{g_2}^{(S \times \p^1 , S_{0,1,\infty}), \sharp}\big( \big( (b, \Delta_{B, \ell}^{\vee}), (b_i, \Delta^{\vee}_{S, \ell_i})_{i=1}^{m} \big), \lambda_2, \lambda_3 \big) \\
= 2 z F_0^{S^{[n]}}( U( \delta \cdot \lambda_1 ), \lambda_2, \lambda_3 )
\end{multline}
\end{lemma}
\begin{proof}
By Lemma~\ref{lemma:Udecomp}, by a careful matching of the signs and $z$ factors,
and by observing that since we have no interior markings the ($\sharp$) convention
yields the same invariant as the ($\bullet$) convention, the left hand side in \eqref{2a:lhs}, equals
\[ 2 \sum_{d} q^d Z^{(S \times \p^1, S_{0,\infty}), \sim, \std}_{(0,n)}(\lambda_1, U_1) Z^{(S \times \p^1, S_{0,1,\infty})}_{(W+dF,n)}(U_2, \lambda_2, \lambda_3), \]
where we write $U_1, U_2$ for summing over the K\"unneth factors of the class $U \in H^{\ast}(S^{[n]} \times S^{[n]})$.
By Proposition~\ref{prop:rubber std} and Theorem~\ref{thm:GW/Hilb} the above then becomes:
\begin{align*}
= & 2 \sum_{d} q^d z \left( \int_{S^{[n]}} \delta \cdot \lambda_1 \cdot U_1 \right) Z^{(S \times \p^1, S_{0,1,\infty})}_{(W+dF,n)}(U_2, \lambda_2, \lambda_3) \\
= & 2 \sum_{d} q^d z Z^{(S \times \p^1, S_{0,1,\infty})}_{(W+dF,n)}(U( \delta \delta_1) , \lambda_2, \lambda_3) \\
= & 2 z F_0^{S^{[n]}}(U( \delta \cdot \lambda_1) , \lambda_2, \lambda_3).
\end{align*}
\end{proof}

\begin{lemma}[\textbf{Term 2b}]
\begin{multline*}
2 \sum_{g \in \BZ} z^{2g-2 -n + \sum_i l(\lambda_i) } (-1)^{g-1+\sum_i l(\lambda_i)} 
\sum_{\substack{ m \geq 0 \\ g = g_1 + g_2 + m }}
\sum_{\substack{ b ; b_1, \ldots, b_m \\ \ell ; \ell_1, \ldots, \ell_m}}
\frac{\prod_{i=j}^{m} b_j}{m!} \\
F_{g_1}^{\sim, \sharp}\left( \lambda_1, \big( (b, \Delta_{B,\ell}) , (b_i, \Delta_{S, \ell_i})_{i=1}^{m}\big)  \right)
F_{g_2}^{(S \times \p^1 , S_{0,1,\infty}), \sharp, \std}\big( \big( (b, \Delta_{B, \ell}^{\vee}), (b_i, \Delta^{\vee}_{S, \ell_i})_{i=1}^{m} \big), \lambda_2, \lambda_3 \big) \\
= 2 F_0^{S^{[n]}}( \lambda_1, U( \lambda_2 \lambda_3) )
+ 2 \left( \int_{S^{[n]}} \lambda_1 U(\lambda_2 \lambda_3) \right) \frac{\mathbf{G}(p,q)^n}{\Theta^2(p,q) \Delta(q)}
\end{multline*}
\end{lemma}
\begin{proof}
With similar reasoning as for Term 2a and using Proposition~\ref{prop:std evaluation} and Proposition~\ref{prop:rubber in terms of non-rubber} this becomes
\begin{align*}
& 2 \sum_{d} q^d Z^{(S \times \p^1, S_{0,\infty}), \sim}_{(W+dF,n)}(\lambda_1, U_1) Z^{(S \times \p^1, S_{0,1,\infty}), \std}_{(0,n)}(U_2, \lambda_2, \lambda_3) \\
= & 2 \sum_{d} q^d Z^{(S \times \p^1, S_{0,\infty}), \sim}_{(W+dF,n)}(\lambda_1, U_1) \int_{S^{[n]}} U_2 \cdot \lambda_2 \cdot \lambda_3  \\
= & 2 \sum_{d} q^d Z^{(S \times \p^1, S_{0,\infty}), \sim}_{(W+dF,n)}(\lambda_1, U(\lambda_2 \lambda_3)).
\end{align*}
Let $D(F) = \frac{1}{(n-1)!} ( (1,F)(1,1)^{n-1}) \in H^2(S^{[n]})$.
Employing Proposition~\ref{prop:rubber in terms of non-rubber} and the evaluation of Proposition~\ref{K3xP1 lemma} we get
\begin{align*}
= & 2 \sum_{d} q^d Z^{(S \times \p^1, S_{0,\infty}), \sim}_{(W+dF,n)}(\lambda_1, U(\lambda_2 \lambda_3), D(F)) 
+ 2 \left( \int_{S^{[n]}} \lambda_1 \cdot U(\lambda_2 \lambda_3) \right) \frac{\mathbf{G}(z,q)^n}{\Theta^2(z,q) \Delta(q)} \\
= & 2 F_0^{S^{[n]}}( \lambda_1, U(\lambda_2 \lambda_3) ) + 2 \left( \int_{S^{[n]}} \lambda_1 \cdot U(\lambda_2 \lambda_3) \right) \frac{\mathbf{G}(p,q)^n}{\Theta^2(p,q) \Delta(q)}
\end{align*}
as desired.
\end{proof}

\begin{lemma}[\textbf{Term 3}]
Let $\lambda_{i}^{(j)}$ be the weighted partition $\lambda_{i}$ but with $j$-th weight $\delta_{ij}$ replaced by $\pi^{\ast} \pi_{\ast}( \delta_{ij})$.
Then we have
\begin{multline*} - 2 
\sum_{g \in \BZ} z^{2g-2 -n + \sum_i l(\lambda_i) } (-1)^{g-1+\sum_i l(\lambda_i)}
\sum_{j=1}^{\ell(\lambda_1)} F_{g}^{(S \times \p^1, S_{0,1,\infty}), \sharp}( \psi_{1,j}^{\rel} \cdot \lambda_1^{(j)}, \lambda_2, \lambda_3 )  \\
= \ 
-2 z F_0^{S^{[n]}}( \delta \cdot U(\lambda_1), \lambda_2, \lambda_3 ) 
 -2 F_0^{S^{[n]}}( U(\lambda_1), \lambda_2 \cdot \lambda_3 )  \\
-2 \left( \int_{S^{[n]}} U(\lambda_1) \lambda_2 \lambda_3 \right) \frac{\mathbf{G}(p,q)^n}{\Theta^2(p,q) \Delta(q)}
\end{multline*}
\end{lemma}
\begin{proof}
We employ the splitting formula for the relative $\psi$-class given in Proposition~\ref{prop:splitting relative psi class}.
The left hand side term becomes:
\begin{align*}
& -2 \sum_{d} \sum_{j=1}^{\ell(\lambda_1)} \frac{1}{\lambda_{1,i}} q^d Z^{(S \times \p^1, S_{0,\infty}), \sim}_{(W+dF,n)}(\lambda_1^{(i)}, \Delta_1) 
Z^{(S \times \p^1, S_{0,1,\infty}), \std}_{(0,n)}(\Delta_2, \lambda_2, \lambda_3) \\
& -2 \sum_{d} \sum_{j=1}^{\ell(\lambda_1)} \frac{1}{\lambda_{1,i}} q^d Z^{(S \times \p^1, S_{0,\infty}), \sim, \std}_{(0,n)}(\lambda_1^{(i)}, \Delta_1) 
Z^{(S \times \p^1, S_{0,1,\infty})}_{(W+dF,n)}(\Delta_2, \lambda_2, \lambda_3) 
\end{align*}
where $\Delta_1, \Delta_2$ stands for summing over the K\"unneth decomposition of the diagonal in $(S^{[n]})^2$.

Observe that $U$ acts on a $H^{\ast}(S)$ weighted partition $\lambda = \big( (\lambda_j, \delta_j) \big)_{j=1}^{\ell}$ by
\[ U \lambda = \sum_{j=1}^{\ell(\lambda)} \frac{1}{\lambda_{j}} \big( (\lambda_1, \delta_1) \cdots \underbrace{(\lambda_i, \pi^{\ast} \pi_{\ast}(\gamma_i)}_{i\text{-th}} \cdots (\lambda_{\ell}, \delta_{\ell}) \big) \]
Hence with Proposition~\ref{prop:std evaluation} and Proposition~\ref{prop:rubber std} the above becomes
\begin{align*}
= & -2 \sum_{d} q^d Z^{(S \times \p^1, S_{0,\infty}), \sim}_{(W+dF,n)}( U(\lambda_1), \Delta_1) \int_{S^{[n]}} \Delta_2 \lambda_2 \lambda_3 \\
& -2 z \sum_{d} q^d \left( \int_{S^{[n]}} U(\lambda_1) \delta \Delta_1 \right) Z^{(S \times \p^1, S_{0,1,\infty})}_{(W+dF,n)}(\Delta_2, \lambda_2, \lambda_3) \\
= & -2 \sum_{d} q^d Z^{(S \times \p^1, S_{0,\infty}), \sim}_{(W+dF,n)}( U(\lambda_1), \lambda_2 \lambda_3) 
- 2 z \sum_{d} q^d Z^{(S \times \p^1, S_{0,1,\infty})}_{(W+dF,n)}(\delta \cdot U(\lambda_1), \lambda_2, \lambda_3) \\
= & -2 F_0^{S^{[n]}}( U(\lambda_1), \lambda_2 \lambda_3) -2 \left( \int_{S^{[n]}} U(\lambda_1) \lambda_2 \lambda_3 \right) \frac{\mathbf{G}(p,q)^n}{\Theta^2(p,q) \Delta(q)}
- 2 z F_0^{S^{[n]}}( \delta \cdot U(\lambda_1), \lambda_2, \lambda_3 ).
\end{align*}
\end{proof}

\begin{lemma}[\textbf{Term 4}]
\[ - \sum_{a,b} (G^{-1})_{ab} T_{e_a} T_{e_b} F_{g}^{(S \times \p^1, S_{0,1,\infty}), \sharp}(\lambda_1, \lambda_2, \lambda_3 )
=
- \sum_{a,b} (G^{-1})_{ab} T_{e_a} T_{e_b} F_{0}^{S^{[n]}}(\lambda_1, \lambda_2, \lambda_3 ).
 \]
\end{lemma}
\begin{proof}
Since there are no interior markings the ($\sharp$) condition yields the same invariants as the ($\bullet$) condition.
Hence the claim is just the application of Theorem~\ref{thm:GW/Hilb}.
\end{proof}

\begin{proof}[Proof of Theorem~\ref{thm:prim HAE}]
Part (a) was proven in Proposition~\ref{prop:quasi-modularity}.
For Part (b) it suffices to prove the equality in Proposition~\ref{prop:reduction}.
We start with equation \eqref{xx4tewt}.
We compute $\left( \frac{d}{d G_2} \right)_{z}$ of the left hand side of \eqref{xx4tewt}
by applying the holomorphic anomaly equation for $(S \times \p^1, S_{0,1,\infty})$ stated in Theorem~\ref{thm:rel HAE}.
This holomorphic anomaly equation produces four terms.
These four terms are precisely the terms labeled 1, 2a, 2b, 3, 4 in the above lemmata (up to permutation).
Summing these four terms together yields:
\begin{align*}
\left( \frac{d}{dG_2} \right)_z F^{S^{[n]}}_{0}(\lambda_1, \lambda_2, \lambda_3)
& =(2-2n) z^2 F_0^{S^{[n]}}( \lambda_1, \lambda_2, \lambda_3 ) \\
& + 2 z F_0^{S^{[n]}}( U( \delta \cdot \lambda_1 ), \lambda_2, \lambda_3 ) \\
& + 2 F_0^{S^{[n]}}( \lambda_1, U( \lambda_2 \lambda_3) ) + 2 \left( \int_{S^{[n]}} \lambda_1 U(\lambda_2 \lambda_3) \right) \frac{\mathbf{G}(p,q)^n}{\Theta^2(p,q) \Delta(q)} \\
& -2 z F_0^{S^{[n]}}( \delta \cdot U(\lambda_1), \lambda_2, \lambda_3 )  
-2 F_0^{S^{[n]}}( U(\lambda_1), \lambda_2 \cdot \lambda_3 )  \\
& \quad \quad -2 \left( \int_{S^{[n]}} U(\lambda_1) \lambda_2 \lambda_3 \right) \frac{\mathbf{G}(p,q)^n}{\Theta^2(p,q) \Delta(q)} \\
& - \sum_{a,b} (G^{-1})_{ab} T_{e_a} T_{e_b} F_{0}^{S^{[n]}}(\lambda_1, \lambda_2, \lambda_3 ) +  (...)
\end{align*}
where $(...)$ stands for the terms where the role of $\lambda_1$ is played by $\lambda_2$ and $\lambda_3$ in the four middle terms.
The above is precisely the right hand side in Proposition~\ref{prop:reduction} if we observe two basic facts:
First, the operator $U$ is symmetric (since the adjoint of $\Fq_n(\alpha)$ is $(-1)^n \Fq_{-n}(\alpha)$):
\[
\int_{S^{[n]}} U(\lambda) \cdot \mu = \int_{S^{[n]}} \lambda \cdot U(\mu),
\]
hence the $\mathbf{G}^n$ terms cancel. And second, we have
\[ T_{\delta} = [ e_{\delta}, U ], \quad \text{ hence } \quad T_{\delta} \lambda = \delta \cdot U(\lambda) - U( \delta \cdot \lambda ). \qedhere \]
\end{proof}

\section{Holomorphic anomaly equations: Imprimitive case}
\label{sec:HAE imprimitive case}
\subsection{Overview}
Let $g,N$ be fixed. For the elliptic K3 surface $S \to \p^1$ 
recall the generating series
\[
F_{g,\ell}(\taut; \gamma_1, \ldots, \gamma_N) = 
\sum_{d = -\ell}^{\infty} \sum_{r \in \BZ} \left\langle \taut ; \gamma_1, \ldots, \gamma_N \right\rangle^{S^{[n]}}_{g, \ell W+dF+rA} q^d (-p)^r,
\]
where we have dropped the supscript $S^{[n]}$ on the left.

We show that the quasi-Jacobi form property and the holomorpic anomaly equation for the primitive series $F_{g,1}$ (Conjecture~\ref{conj:fg HAE prim})
together with the Multiple Cover Conjecture (Conjecture~\ref{MCConjecture})
implies both claims for the general series $F_{g,\ell}$.
More precisely we have:
\begin{prop} \label{prop:Implication}
If Conjecture \ref{MCConjecture} and
Conjecture~\ref{conj:fg HAE prim}
hold for all $g',N'$ such that either $g'<g$ or ($g'=g$ and $N'<N$),
then Conjecture~\ref{Conj: Quasi Jacobi property} and Conjecture~\ref{Conj:HAE} holds for $g,N$.
\end{prop}

Using Proposition~\ref{prop:Implication}
we obtain the proof of the main theorem of the paper:

\begin{proof}[Proof of Theorem~\ref{thm:Main result}]
If $g=0$ and $N \leq 3$, then Conjecture~\ref{MCConjecture} holds by Theorem~\ref{MCConjecture for g=0 N<=3},
and Conjecture~\ref{conj:fg HAE prim} was proven in Theorem~\ref{thm:prim HAE}.
Hence the claim follows from Proposition~\ref{prop:Implication}.
\end{proof}

The proof of Proposition~\ref{prop:Implication}
is purely formal: if the multiple cover formula holds,
then $F_{g,\ell}$ is obtained from $F_{g,1}$ by applying the Hecke operator.
The statement then follows from results about Hecke operators on quasi-Jacobi forms
(Section~\ref{subsec:Hecke operators}) and basic properties of the operators appearing in the holomorphic anomaly equation.

\subsection{Proof}
\begin{proof}[Proof of Proposition~\ref{prop:Implication}]
Recall the formal $\ell$-th weight $k$ Hecke operator $T_{k,\ell}$ defined in \eqref{Hecke formal}.
If the Multiple Cover Conjecture holds, then for all $\ell > 0$ we have
\[ F_{g,\ell}(\taut; \gamma_1, \ldots, \gamma_N) =  
\ell^{\sum_i (\deg(\gamma_i) - n - \wt(\gamma_i))}
T_{k,\ell} 
F_{g,1}(\taut; \gamma_1, \ldots, \gamma_N)
\]
where $k = n(2g-2+N) + \sum_i \wt(\gamma_i)$.
Assuming part (a) of Conjecture~\ref{conj:fg HAE prim} we have
\[
F_{g,1}(\taut; \gamma_1, \ldots, \gamma_N) \in \frac{1}{\Delta(\tau)} \QJac_{k', n-1}
\]
where $k' = n(2g-2+N) + \sum_i \wt(\gamma_i) - 10$.
Hence by Proposition~\ref{prop:Hecke wrong weight with poles} (describing the action of Hecke operators of weight $k$ on weight $k'$ forms) we find that:
\[
F_{g,\ell}(\taut; \gamma_1, \ldots, \gamma_N) \in \frac{1}{\Delta(\tau)^{\ell}} \QJac_{k'+12 \ell, (n-1) \ell}(\Gamma_0(\ell)).
\]
that is Conjecture~\ref{Conj: Quasi Jacobi property} holds.

To prove Conjecture~\ref{Conj:HAE}, the multiple cover conjecture and the relations \eqref{hecke relations wrong weight} give:
\begin{align*}
\frac{d}{d G_2} F_{g,\ell}(\taut; \gamma_1, \ldots, \gamma_N)
& = \ell^{e(\gamma_1,\ldots,\gamma_N)} \frac{d}{d G_2} T_{k,\ell} F_{g,1}(\taut; \gamma_1, \ldots, \gamma_N) \\
& = 
\ell^{e(\gamma_1,\ldots,\gamma_N)+1}
T_{k-2,\ell} 
\frac{d}{d G_2} F_{g,1}(\taut; \gamma_1, \ldots, \gamma_N). 
\end{align*}
where 
\begin{gather*} k = k(g,N, \gamma_1, \ldots, \gamma_N) := n(2g-2+N) + \sum_i \wt(\gamma_i), \\
e(\gamma_1,\ldots,\gamma_N) = \sum_i (\deg(\gamma_i) - n - \wt(\gamma_i)).
\end{gather*}
Assuming part (b) of Conjecture~\ref{conj:fg HAE prim} this equals
\[
= \ell^{e(\gamma_1,\ldots,\gamma_N)+1}
T_{k-2,\ell} 
\left[
\begin{array}{l} 
F_{g-1,1}(\taut; \gamma_1, \ldots, \gamma_N, U) \\
+ 2 
\sum_{\substack{g=g_1 + g_2 \\ \{ 1, \ldots, N \} = A \sqcup B }} F_{g_1,1}(\taut_1 ; \gamma_A, U_1) F_{g_2}^{\std}(\taut_2 ; \gamma_B , U_2 ) \\
- 2 \sum_{i=1}^{N} F_{g,1}( \psi_i \cdot \taut ; \gamma_1, \ldots, \gamma_{i-1}, U \gamma_i, \gamma_{i+1}, \ldots, \gamma_N) \\
- \sum_{a,b} (g^{-1})_{ab} T_{e_a} T_{e_b} F_{g,1}(\taut; \gamma_1, \ldots, \gamma_N)
\end{array}
\right]
\]
By Lemmata~\ref{lemma:e evals} and \ref{lemma:k evals} below
we can apply Conjecture~\ref{MCConjecture} to this term in reverse, e.g.
\begin{align*} 
\ell^{e(\gamma_1,\ldots,\gamma_N)+1}
T_{k-2,\ell} F_{g-1,1}(\taut; \gamma_1, \ldots, \gamma_N, U) 
& = F_{g-1, \ell}(\taut; \gamma_1, \ldots, \gamma_N, U),
\end{align*}
or the exceptional case
\begin{multline*}
\ell^{e(\gamma_1,\ldots,\gamma_N)+1}
T_{k-2,\ell} F_{g,1}(\taut; \ldots , T_{e_a} \gamma_i, \ldots, T_{e_b} \gamma_j, \ldots, ) \\
=
\frac{1}{\ell} F_{g,\ell}(\taut; \ldots , T_{e_a} \gamma_i, \ldots, T_{e_b} \gamma_j, \ldots, ), \end{multline*}
et cetera. As a result we obtain precisely the right hand side for the $\frac{d}{d G_2}$-holomorphic anomaly equation in
Conjecture~\ref{Conj:HAE}.

Similarly, by Proposition~\ref{prop:constrains monodromy} we have
\[
\frac{d}{d \A} F_{g,1}(\taut; \gamma_1, \ldots, \gamma_N) 
=
T_{\delta} F_{g,1}(\taut; \gamma_1, \ldots, \gamma_N).
\]
Hence by the relations \eqref{hecke relations wrong weight} we have
\begin{align*}
\frac{d}{d \A} 
 F_{g,\ell}(\taut; \gamma_1, \ldots, \gamma_N)
& = \ell^{\sum_i (\deg(\gamma_i) - n - \wt(\gamma_i))} \frac{d}{d\A} T_{k,\ell} \frac{d}{d \A} F_{g,1}(\taut; \gamma_1, \ldots, \gamma_N) \\
& = \ell \cdot \ell^{\sum_i (\deg(\gamma_i) - n - \wt(\gamma_i))} T_{k-1,\ell} \frac{d}{d \A} F_{g,1}(\taut; \gamma_1, \ldots, \gamma_N) \\
& = \ell \cdot \ell^{\sum_i (\deg(\gamma_i) - n - \wt(\gamma_i))} T_{k-1,\ell} T_{\delta} F_{g,1}(\taut; \gamma_1, \ldots, \gamma_N) \\
& = T_{\delta} F_{g,\ell}(\taut; \gamma_1, \ldots, \gamma_N)
\end{align*}
where we used that $T_{\delta}$ is of weight $-1$ (Lemma~\ref{lemma:T action}).
\end{proof}

\begin{lemma} \label{lemma:k evals}
If $U=\sum_i a_i \otimes b_i$ is a $\wt$-homogeneous K\"unneth decomposition of $U \in H^{\ast}(S^{[n]})^{\otimes 2}$,
then for every $i$ we have
\begin{gather*}
k(g-1,N+2, \gamma_1, \ldots, \gamma_N,a_i,b_i) = k(g,N, \gamma_1, \ldots, \gamma_N,a_i,b_i) - 2 \\
k(g_1,|A|+1,\gamma_A, a_i) = k(g,N,\gamma_1,\ldots,\gamma_N)-2, \quad \text{ if } F_{g_2}^{S^{[n]}, \std}(\taut_2 ; \gamma_B , b_i ) \neq 0 \\
k(g,N,\gamma_1, \ldots, U(\gamma_i), \ldots, \gamma_N) = k(g,N,\gamma_1, \ldots, \gamma_N)-2 \\
k(g,N,\gamma_1, \ldots, T_{e_a}\gamma_i, \ldots, T_{e_b}\gamma_j, \ldots, \gamma_N) = k(g,N,\gamma_1,\ldots, \gamma_N) - 2.
\end{gather*}
\end{lemma}
\begin{proof}
The first of these equations follows from Lemma~\ref{lemma: weight of U},
and the 3rd and 4th follows from Lemma~\ref{lemma:T action}.
For the second, recall that for $X=S^{[n]}$ we have
\begin{equation} \label{degree zero virtual class}
[ \Mbar_{g,N}(X,0)]^{\std} =
\begin{cases}
[\Mbar_{0,N} \times X] & \text{ if } g=0, N \geq 3 \\
[\Mbar_{1,N} \times X] \pi_2^{\ast}(c_{2n}(X)) & \text{ if } g=1, N \geq 1 \\
0 & \text{ if } g \geq 2
\end{cases}
\end{equation}
If $F_{g_2}^{\std}(\taut_2 ; \gamma_B , b_i ) \neq 0$, we hence find:
\[ \sum_i F_{g_2}^{\std}(\taut_2 ; \gamma_B , b_i ) a_i =
\left( \int_{\Mbar_{g_2,|B|+1}}  \taut_2\right)  \cdot 
\begin{cases}
U\left( \prod_{i \in B} \gamma_i \right) & \text{ if } g_2=0 \\
U\left( c_{2n}(X) \prod_{i \in B} \gamma_i \right) & \text{ if } g_2=1.
\end{cases}
\]
Hence using Lemma~\ref{lemma:wt multiplicative} we get:
\begin{align*}
k(g_1,|A|+1,\gamma_A, a_i) 
& =
\begin{cases}
k\left(g,|A|+1,\gamma_A, U\left( \prod_{i \in B} \gamma_i \right) \right) & \text{ if } g_2=0 \\
k\left(g-1,|A|+1,\gamma_A, U\left( c_{2n}(X) \prod_{i \in B} \gamma_i \right) \right) & \text{ if } g_2=1
\end{cases}\\
& = k(g,N,\gamma_1, \ldots, \gamma_N)
\end{align*}
where we used $\wt(c_{2n}(X)) = n$ in the last step.
\end{proof}

\begin{lemma}\label{lemma:e evals}
If $U=\sum_i a_i \otimes b_i$ is a $\wt$-homogeneous K\"unneth decomposition of $U \in H^{\ast}(S^{[n]})^{\otimes 2}$,
then for every $i$ we have
\begin{gather*}
e(\gamma_1, \ldots, \gamma_N, a_i, b_i) = e(\gamma_1, \ldots, \gamma_N)+1 \\
e(\gamma_A, a_i) = e(\gamma) + 1, \quad \text{ if } F_{g_2}^{\std}(\taut_2 ; \gamma_B , b_i ) \neq 0 \text{ for some $g_2$}  \\
e(\gamma_1, \ldots, U(\gamma_i), \ldots, \gamma_N) = e(\gamma_1, \ldots, \gamma_N)+1  \\
e(\gamma_1, \ldots, T_{e_a}\gamma_i, \ldots, T_{e_b}\gamma_j, \ldots, \gamma_N) = e(\gamma_1,\ldots, \gamma_N) + 2.
\end{gather*}
\end{lemma}
\begin{proof}
With the notation of Section~\ref{sec:Weight grading} define $h_{FW} := \act(F \wedge W)$,
which acts semisimply on $H^{\ast}(S^{[n]})$.
For an eigenvector $\gamma$, define $\deg_{FW}(\gamma)$ to be the eigenvalue of $h_{FW}$, i.e.
\[ h_{FW}(\gamma) = \deg_{FW}(\gamma) \gamma. \]
Then because we have
\[ (\deg(\gamma)-n-\wt(\gamma)) \gamma = (h - \Wt)\gamma = - \act(W \wedge F)\gamma = h_{FW}(\gamma) \]
we find
\[ e(\gamma_1, \ldots, \gamma_N) = \sum_i \deg_{FW}(\gamma_i). \]
The claim now follows parallel to Lemma~\ref{lemma:k evals}
(use that $h_{FW}=h-\Wt$, so the corresponding properties for the grading operator $h_{FW}$ are easily derived).
\end{proof}

\section{Fiber classes} \label{sec:fiber case}
\subsection{Overview}
We study the generating series of Gromov-Witten invariants of $S^{[n]}$ in fiber classes
of the Lagrangian fibration $S^{[n]} \to \p^n$,
\[
F_{g,0}(\taut; \gamma_1, \ldots, \gamma_N) = 
\sum_{d \geq 0} \sum_{\substack{r \in \BZ\\ (d,r) \neq (0,0)}} \left\langle \taut ; \gamma_1, \ldots, \gamma_N \right\rangle^{S^{[n]}}_{g, dF+rA} q^d (-p)^r.
\]
Recall from Theorem~\ref{thm:2132}
the weight $n$ (meromorphic) quasi-Jacobi forms
\[
\A_n(z,\tau) = B_n + \delta_{n,1}\frac{1}{2} \frac{p^{1/2}+p^{-1/2}}{p^{1/2} - p^{-1/2}} - n \sum_{k,d \geq 1} d^{n-1} (p^{k} + (-1)^n p^{-k}) q^{kd} \ \ 
\in \MQJac_{0,n}.
\]
For any $(\deg,\wt)$-bihomogeneous class $\gamma$ define the modified degree:
\[ \deg_{WF}(\gamma) = n + \wt(\gamma) - \deg(\gamma). \]

\begin{rmk} \label{rmk:WFgrading}
Consider the basis of $H^{\ast}(S,\BQ)$ given by $\CB = 
\{ 1, \pt, W, F, e_a \}$, where $\{ e_a \}$ is a basis of $\{ W, F \}^{\perp} \subset H^2(S,\BQ)$.
If $\gamma = \prod_{i} \Fq_{n_i}(\delta_i) \1$ for $\delta_i \in \CB$, we have
\[ \deg_{WF}(\gamma) = | \{ i : \delta_i = W \} | - | \{ i | \delta_i = F \} |. \]
By Section~\ref{sec:Weight grading},
$\deg_{WF}(\gamma)$ is also the eigenvalue of the operator
$h_{WF} := \act(W \wedge F)$. \qed
\end{rmk}

The main result of this section is the following:
\begin{thm} \label{thm:fiber classes} Fix $g,N$ with $2g-2+N>0$ such that the following holds:
\begin{itemize}
\item[(i)] the Multiple Cover Conjecture (Conjecture~\ref{MCConjecture}) holds for this $g,N$, 
\item[(ii)] $\left\langle \taut ; \gamma_1, \ldots, \gamma_N \right\rangle_{g,dF+rA} = 0$
for all $(d,r) \neq (0,0)$, whenever $\sum_i \deg_{WF}(\gamma_i) < 0$.
\end{itemize}
Let $\gamma_i$ be $(\wt,\deg)$-bihomogeneous classes and let
\[ a=3g-3+N-\deg(\taut), \quad b=\sum_{i=1}^{N} \deg_{WF}(\gamma_i). \]
If $a,b \geq 0$, then in $\BC((p))[[q]]/\BC$ we have
\begin{multline}
\label{multline:sum Fg0}
F_{g,0}(\taut; \gamma_1, \ldots, \gamma_N)
\equiv
\left\langle \taut ; \gamma_1, \ldots, \gamma_N \right\rangle^{S^{[n]}}_{g, F} \sum_{d,k \geq 1} k^a d^b q^{kd} \\
+ 
\sum_{r \geq 1} 
(-1)^r \left\langle \taut ; \gamma_1, \ldots, \gamma_N \right\rangle^{S^{[n]}}_{g, F+rA}
\left( \frac{-1}{b+1} \left( p \frac{d}{dp} \right)^a \A_{b+1}(p,q) \right)\Big|_{p \mapsto p^{r}}.
\end{multline}
In particular,
$F_{g,0}(\taut; \gamma_1, \ldots, \gamma_N)$
is a meromorphic quasi-Jacobi form of weight $k=n(2g-2+N) + \sum_i \wt(\gamma_i)$ and index $0$,
with poles at torsion points.
\end{thm}

Here for two power series $f(p,q),g(p,q) \in \BC((p))[[q]]$, we write $f \equiv g$ if they are equal in $\BC((p))[[q]]/\BC$
that is if there exists a constant $c \in \BC$ such that $f(p,q)=g(p,q)+c$.

In \eqref{multline:sum Fg0} the sum over $r$ is finite by Lemma~\ref{lemma:Laurent polynomial},
hence the statement of the theorem is well-defined.
If $a<0$ in Theorem~\ref{thm:fiber classes},
then $\taut=0$, so all Gromov-Witten invariants would vanish.
Condition (ii) in Theorem~\ref{thm:fiber classes} would follow from 
a family-version of the GW/Hilb correspondence (Section~\ref{sec:hilb/GW correspondence}),
where one does not fix the complex structure of the source curve.
Hence (ii) is expected to hold for all $g,N$ with $2g-2+N>0$.
We prove condition (ii) for $(g,N)=(0,3)$ below and obtain the following.

\begin{thm} \label{thm:HAE g=0,N=3}
For any $\gamma_1, \gamma_2 ,\gamma_3 \in H^{\ast}(S^{[n]})$
the series
$F_{g=0,0}(\taut; \gamma_1, \ldots, \gamma_N)$
is a meromorphic quasi-Jacobi form of weight $n + \sum_i \wt(\gamma_i)$ and index $0$
with poles at torsion points (of the form given in \eqref{multline:sum Fg0}).
Moreover, in $\BC((p))[[q]]/\BC$ we have
\begin{gather}
\frac{d}{dG_2} F_{0,0}(\taut; \gamma_1, \gamma_2, \gamma_3) \equiv 0  \notag \\
\frac{d}{d \A} F_{0,0}(\taut; \gamma_1, \gamma_2, \gamma_3) \equiv 
T_{\delta} F_{0,0}(\taut; \gamma_1, \gamma_2, \gamma_3). \label{HAE dda for g=0N=3}
\end{gather}
\end{thm}

\subsection{Multiple cover conjecture}
We first recall an equivalent form of the Multiple Cover Conjecture (Conjecture \ref{MCConjecture}).
Let $S$ be any K3 surface with an effective curve class $\beta \in H_2(S,\BZ)$.
For every divisor $k|\beta$ let $S_k$ be some K3 surface and
consider any {\em real} isometry
\[ \varphi_k : H^{2}(S,\BR) \to H^{2}(S_k, \BR) \]
such that $\varphi_k(\beta/k) \in H_2(S_k,\BZ)$
is a primitive effective curve class.
We extend $\varphi_k$ to the full cohomology lattice
by $\varphi_k(\pt) = \pt$ and $\varphi_k(1) = 1$.
Define an extension to the Hilbert scheme by acting factorwise in the Nakajima operators:
\[
\varphi_k : H^{\ast}(S^{[n]}) \to H^{\ast}( S_k^{[n]} ),
\quad \prod_i \Fq_{n_i}(\delta_i) \1 \mapsto \prod_i \Fq_{n_i}( \varphi_k(\delta_i)) \1.
\]

\begin{conj} \label{conj mc hilb} We have
\[
\blangle \taut ; \gamma_1, \ldots \gamma_N \brangle^{S^{[n]}}_{g, \beta+rA} 
= \sum_{k | (\beta,r)}  k^{3g-3+N-\deg(\taut)} (-1)^{r + \frac{r}{k}}
\blangle \taut ; \varphi_k(\gamma_1), \ldots \varphi_k(\gamma_N) \brangle^{S^{[n]}}_{g, \varphi_k\left( \frac{\beta}{k} \right)+\frac{r}{k} A}.
\]
\end{conj} 

This conjecture is equivalent to the one we have given in the introduction.
\begin{lemma}[{\cite[Lemma 3]{ObMC}}]
Conjecture~\ref{conj mc hilb} is equivalent to Conjecture~\ref{MCConjecture}.
\end{lemma}

\subsection{Proof of Theorem~\ref{thm:fiber classes}}

\emph{Step 1: Positive part.}
We apply the Multiple Cover Conjecture (in the form of Conjecture~\ref{conj mc hilb})
to the following series, where we sum only over curve classes which have positive fiber degree:
\[
F_{g,0}^{+}(\taut; \gamma_1, \ldots, \gamma_N) = 
\sum_{d \geq 1} \sum_{r \in \BZ} \left\langle \taut ; \gamma_1, \ldots, \gamma_N \right\rangle^{S^{[n]}}_{g, dF+rA} q^d (-p)^r.
\]
For any $k|(d,r)$ let
$\varphi_k : H^2(S,\BQ) \to H^2(S,\BQ)$ be the isometry defined by:
\[ F \mapsto \frac{k}{d} F, \quad W \mapsto \frac{d}{k} W, \quad \varphi_{k}|_{\{W,F\}^{\perp}} = \id. \]
Assuming that all $\gamma_i$ are written in the Nakajima basis with weightings from the fixed basis $\CB$ (defined in Remark~\ref{rmk:WFgrading}),
we obtain:
\[ F^{+}_{g,0}(\taut; \gamma_1, \ldots, \gamma_N)
= \sum_{d \geq 1} \sum_{r \in \BZ} \sum_{k|(d,r)} k^{b} \left( \frac{d}{k} \right)^a (-1)^{r/k} \left\langle \taut ; \gamma_1, \ldots, \gamma_N \right\rangle^{S^{[n]}}_{g, F+rA} p^r q^d \]
Using the monodromy \eqref{monodromy:involution} we have
\[
\left\langle \taut ; \gamma_1, \ldots, \gamma_N \right\rangle^{S^{[n]}}_{g, F+rA} = (-1)^{n N + \sum_i l(\gamma_i)}
\left\langle \taut ; \gamma_1, \ldots, \gamma_N \right\rangle^{S^{[n]}}_{g, F-rA}.
\]
Hence we conclude that:
\begin{multline*}
F^{+}_{g,0}(\taut; \gamma_1, \ldots, \gamma_N)
=
\left\langle \taut ; \gamma_1, \ldots, \gamma_N \right\rangle^{S^{[n]}}_{g, F} \sum_{d,k \geq 1} k^a d^b q^{kd} \\
+ 
\sum_{r \geq 1} (-1)^r \left\langle \taut ; \gamma_1, \ldots, \gamma_N \right\rangle^{S^{[n]}}_{g, F+rA}
\left( \sum_{k,d \geq 1} k^{a} d^{b} (p^k + (-1)^{n N + \sum_i l(\gamma_i)} p^{-k}) q^{kd} \right)\Bigg|_{p \mapsto p^r}.
\end{multline*}

We analyse now the second term on the right.
Since otherwise all invariants vanish, we can assume the dimension constraint
\[ \vd \ \Mbar_{g,N}(S^{[n]},\beta) = (2n-3)(1-g)+N+1 = \deg(\taut) + \sum_i \deg(\gamma_i) \]
or equivalently,
\begin{equation} a = 3g-3+N-\deg(\taut) = 2n(g-1)-1+\sum_i \deg(\gamma_i). \label{dim constraint} \end{equation}

Let further $\gamma_{i,j} \in H^{\ast}(S)$ be the cohomology weights of $\gamma_i$ in the Nakajima basis.
Let $V = \{ W, F \}^{\perp} \subset H^2(S,\BZ)$.
Since $\ev_{\ast}[ \Mbar_{g,N}(S^{[n]}, dF+rA) ]^{\vir}$
is invariant under the monodromy group $O(V,\BZ)$,
by standard invariant theory for the orthogonal group (e.g. \cite[Sec.6.1]{MarkedRelative})
we can assume that there is an even number of $\gamma_{ij}$ such that $\gamma_{ij} \in V$.
Indeed, otherwise, all the invariants $\left\langle \taut ; \gamma_1, \ldots, \gamma_N \right\rangle^{S^{[n]}}_{g, F+rA}$ vanishes and there is nothing to prove.
We obtain the following parity result.

\begin{lemma} \label{lemma:parity} $a + nN + \sum_i l(\gamma_i) \equiv b-1$ modulo $2$. \end{lemma}
\begin{proof}[Proof of Lemma~\ref{lemma:parity}]
Using \eqref{deg Nakajima cycle} we have
\[
\sum_i \deg(\gamma_i) = nN - \sum_{i} l(\gamma_i) + \sum_{i,j} \deg(\gamma_{ij}).
\]
Hence by the dimension constraint \eqref{dim constraint} and modulo $2$ we have
\begin{align*}
a + nN + \sum_i l(\gamma_i)
& \equiv -1+\deg(\gamma_i) + n N + \sum_i l(\gamma_i) \\
& \equiv -1+\sum_{i,j} \deg(\gamma_{ij}) \\
& \equiv -1+ \sum_i |\{ j : \gamma_{ij} \in H^2(S) \}| \\
& \overset{(\ast)}{\equiv} -1+ \sum_i |\{ j : \gamma_{ij} \in \{ W, F \} \}| \\
& \equiv b - 1
\end{align*}
where in (*) we used that there is an even number of $\gamma_{ij}$ in $\{ W, F \}^{\perp}$.
\end{proof}

\begin{multline} \label{Fgplus eval}
F^{+}_{g,0}(\taut; \gamma_1, \ldots, \gamma_N)
=
\left\langle \taut ; \gamma_1, \ldots, \gamma_N \right\rangle^{S^{[n]}}_{g, F} \sum_{d,k \geq 1} k^a d^b q^{kd} \\
+ 
\sum_{r \geq 1} (-1)^r \left\langle \taut ; \gamma_1, \ldots, \gamma_N \right\rangle^{S^{[n]}}_{g, F+rA}
\left( \left( p \frac{d}{dp} \right)^a \sum_{k,d \geq 1} d^{b} (p^k + (-1)^{b+1} p^{-k}) q^{kd} \right)\Bigg|_{p \mapsto p^r}.
\end{multline}

\noindent \emph{Step 2: Fiber degree zero part.} It remains to compute the degree zero part
\[ F^{(0)}_{g,0}(\taut; \gamma_1, \ldots, \gamma_N) = \sum_{r \geq 1} \left\langle \taut ; \gamma_1, \ldots, \gamma_N \right\rangle^{S^{[n]}}_{g, rA} (-p)^r. \]

\begin{lemma} \label{lemma positive}
If $\sum_i \deg_{WF}(\gamma_i) \neq 0$, then $F^{(0)}_{g,0}(\taut; \gamma_1, \ldots, \gamma_N) = 0$. 
\end{lemma}
\begin{proof}
By monodromy invariance, the class
\[ \ev_{\ast}\left( \taut \cdot [ \Mbar_{g,N}(S^{[n]},rA) ]^{\vir} \right) \in H^{\ast}(S^{[n]})^{\otimes N} \]
has weight zero with respect to the grading operator $h_{WF} = \act(W \wedge F)$.
On the other hand,
\[ h_{WF}(\gamma_1 \otimes \cdots \otimes \gamma_N) = \sum_{i} \gamma_1 \otimes \cdots h_{WF}(\gamma_i) \otimes \cdots \gamma_N
= (\sum_i \deg_{WF}(\gamma_i)) \gamma_1 \otimes \cdots \otimes \gamma_N. \]
Hence if $\sum_i \deg_{WF}(\gamma_i) \neq 0$, the pairing between these two classes vanishes.
\end{proof}

\begin{lemma} \label{lemma:zero}
If $\sum_i \deg_{WF}(\gamma_i) = 0$ and under the assumptions of Theorem~\ref{thm:fiber classes}, we have
\[ F^{(0)}_{g,0}(\taut; \gamma_1, \ldots, \gamma_N) = \sum_{r \geq 1} (-1)^r \left\langle \taut ; \gamma_1, \ldots, \gamma_N \right\rangle^{S^{[n]}}_{g, F+rA}
\sum_{k \geq 1} k^a p^{kr}. \]
\end{lemma}
\begin{proof}
Recall the monodromy $e^{-T_{\delta}}$ from Section~\ref{monodromy:shift}
which satisfies $e^{-T_{\delta}} A = A + F$.
We conclude that
\begin{equation} \label{39sig0ig}
\left\langle \taut ; \gamma_1, \ldots, \gamma_N \right\rangle^{S^{[n]}}_{g, rA}
=
\left\langle \taut ; e^{-T_{\delta}} \gamma_1, \ldots, e^{-T_{\delta}} \gamma_N \right\rangle^{S^{[n]}}_{g, rF+rA}. \end{equation}
The operator $T_{\delta}$ satisfies the commutation relation
\[ [h_{WF}, T_{\delta}] = [ \act(W \wedge F), \act(\delta \wedge F) ] = -T_{\delta} \]
and hence $\deg_{WF}(T_{\delta} \gamma) = \deg_{WF}(\gamma)-1$.
Because we assumed $\sum_i \deg_{WF}(\gamma_i)=0$ and condition (ii) of Theorem~\ref{thm:fiber classes},
only the leading term in $e^{-T_{\delta}} \gamma_i$ can contribute:
\[ 
\text{(Term in Eqn. \eqref{39sig0ig})}
\, = \, \left\langle \taut ; \gamma_1, \ldots, \gamma_N \right\rangle^{S^{[n]}}_{g, rF+rA}. \]
Using the multiple cover formula (Conjecture~\ref{conj mc hilb}) 
and $b=\sum_i \deg_{WF}(\gamma_i)=0$ this becomes:
\[ = \sum_{k|r} k^{a} (-1)^{r+r/k} \left\langle \taut ; \gamma_1, \ldots, \gamma_N \right\rangle^{S^{[n]}}_{g, F+\frac{r}{k}A}. \]
The claim of the lemma follows from this by rearranging the sums.
\end{proof}

\noindent \emph{Step 3: Proof of Eqn. \eqref{multline:sum Fg0}.}
If $b=\sum_i \deg_{WF}(\gamma_i) > 0$, then
by Lemma~\ref{lemma positive} 
the series $F_{g,0}(\taut; \gamma_1, \ldots, \gamma_N)$
is given by \eqref{Fgplus eval},
and since $[\A_{b+1}]_{q^0}$ is a constant in $p$,
the right hand side of \eqref{Fgplus eval} is precisely as claimed in \eqref{multline:sum Fg0}.
If $b=0$, we add the evaluation of Lemma~\ref{lemma:zero} to \eqref{Fgplus eval}
and use the straightforward identity:
\[
\sum_{k \geq 1} k^a p^{kr} = \text{constant} + \left( - \left( p \frac{d}{dp} \right)^a \frac{1}{2} \frac{p^{1/2}+p^{-1/2}}{p^{1/2} - p^{-1/2}} \right)_{p \mapsto p^r}
\]

\noindent \emph{Step 4: Quasi-Jacobi form property}.
Since $\A_{b+1} \in \MQJac_{b+1,0}$, the derivative $p\frac{d}{dp}$ increases the weight by $1$,
and if the operator $f(p,q) \mapsto f(p^r,q)$ sends quasi-Jacobi forms of weight $k$ and index $m$
to quasi-Jacobi forms of weight $k$ and index $m r^2$ (see \cite[Thm I.4.1]{EZ}),
we see that the second term on the right in \eqref{multline:sum Fg0}
is a quasi-Jacobi form of weight
\[ a+b+1 = 2n(g-1)+\sum_i \deg(\gamma_i) + \sum_i \deg_{WF}(\gamma_i) = n(2g-2+N) + \sum_i \wt(\gamma_i). \]

By the monodromy of Section~\ref{mon:involution} we have
\[ \left\langle \taut ; \gamma_1, \ldots, \gamma_N \right\rangle^{S^{[n]}}_{g, F} = 
(-1)^{n N + l(\gamma_1) + \ldots + l(\gamma_N)}
\left\langle \taut ; \gamma_1, \ldots, \gamma_N \right\rangle^{S^{[n]}}_{g, F}. \]
Hence the first term in \eqref{multline:sum Fg0} unless 
$n N + l(\gamma_1) + \ldots + l(\gamma_N)$ is even, in which case 
$a \equiv b+1$ modulo $2$ by~\ref{lemma:parity}.
If $a > b$ we find in $\BC[[q]]/\BC$ the equality
\[ \sum_{d,k \geq 1} k^a d^b q^{kd} \equiv \left( q \frac{d}{dq} \right)^{b} \sum_{m \geq 1} \sum_{k|m} k^{a-b} q^m, \]
and since this is the $q$-derivative of an Eisenstein series we get
\[ \text{constant} + \sum_{d,k \geq 1} k^a d^b q^{kd} \in \QMod_{a+b+1}. \]
The case $b>a$ is parallel.
\qed

\subsection{Conclusion}
We prove Theorem~\ref{thm:HAE g=0,N=3} and Theorem~\ref{thm:fiber classes intro} of the introduction.

\begin{proof}[Proof of Theorem \ref{thm:fiber classes intro}(i)]
If there is an index $i$ (let us say $i=1$) with $\gamma_1 = F \cdot \tilde{\gamma}_1$,
then by using $F = [E]$ for a smooth elliptic fiber $\iota : E \hookrightarrow S$ a straightforward computation gives
\begin{equation}\label{reduction to E}
F_{g,0}^S(\taut; \gamma_1, \ldots, \gamma_N) = \sum_{d=1}^{\infty}
\left\langle \taut \cdot (-1)^{g-1} \lambda_{g-1} ; \iota^{\ast}(\tilde{\gamma}_1), \iota^{\ast}(\gamma_2) , \ldots , \iota^{\ast}(\gamma_N) \right\rangle^{E}_{g,d[E]} q^d,
\end{equation}
where we used the standard notation for the (ordinary, non-reduced) Gromov-Witten invariants of the elliptic curve $E$.
In this case Conjecture~\ref{conj:Fiber class}(i) follows from \cite{OkPand_Completed_Cycles}, and one checks that the holomorphic anomaly equation of \cite{HAE}
implies the one stated in Conjecture~\ref{conj:Fiber class}(ii).

In case there is no such $i$,
by expressing $\taut$ as boundary classes and the splitting formula, as well as using the divisor equation,
we can reduce the claim to the case $(g,N)=(0,3)$. This base case holds by inspection from the explicit evaluation: 
\[ F_{g,0}^S(1; W,W,W) = \langle W,W,W \rangle^{S}_{0,F} \sum_{k,d \geq 1} d^3 q^{kd} = \text{cst} + 24 G_4(q). \qedhere \]
\end{proof}

We prove a basic vanishing for the Gromov-Witten theory of the elliptic K3 surface $S$:
\begin{lemma} \label{lemma:K3} If $\sum_i \deg_{WF}(\gamma_i) < 0$, then
every $\left\langle \taut ; \gamma_1, \ldots, \gamma_N \right\rangle^S_{g,dF}=0$ for all $d>0$.
\end{lemma}
\begin{proof}
We assume $\gamma_i \in \CB$ for all $i$.
By Remark~\ref{rmk:WFgrading}, if $\sum_i \deg_{WF}(\gamma_i) < 0$
there exists at least one cohomology class with $\gamma_i=F$.
Hence by expressing the invariants of $S$ in terms of the invariants of the elliptic fiber $E$ as in \eqref{reduction to E} we see that,
(i) if $\gamma_j = F$ for some $j \neq i$, then the invariant vanishes,
and (ii) if there are no other cohomology classes with $\gamma_i=W$ then the integrand on $\Mbar_{g,N}(E,d)$ is
invariant under translation by $E$ and hence the integral vanishes (see e.g. \cite[Sec.5.4]{OkPandVir}).
Since we are always at least in one of the two cases, this proves the claim.
\end{proof}

\begin{proof}[Proof of Theorem~\ref{thm:HAE g=0,N=3}]
In case $(g,N)=(0,3)$ we can take $\taut=1$, so $a=0$.
By Theorem~\ref{MCConjecture for g=0 N<=3} the multiple cover conjecture holds for this $(g,N)$.
Moreover, using the GW/Hilb correspondence (Theorem~\ref{thm:GW/Hilb})
the product formula for the relative Gromov-Witten theory of $(S \times \p^1, S_{0,1,\infty})$,
and Lemma~\ref{lemma:K3} also condition (ii) of Theorem~\ref{thm:fiber classes} holds.
Hence the first two claims follow directly from Theorem~\ref{thm:fiber classes} and the $\frac{d}{dG_2}$-holomorphic anomaly equation
for $\A_n$ proven in \ref{thm:2132}.
It remains to prove \eqref{HAE dda for g=0N=3}.
This last claim follows by 
either (a) using the monodromy of Section~\ref{monodromy:shift}
to derive the elliptic transformation law in the meromorphic case
or by (b) applying the GW/Hilb correspondence (possible since the multiple cover conjecture is proven
for fiber classes \cite{BB}, see Remark~\ref{rmk:GW/Hilb imprimitive})
 and then use \eqref{fiberwise ddg2} to calculate the $\frac{d}{d\A}$ derivative
in terms of the $z$-expansion (similarly to what was done in Section~\ref{sec:HAE primitive case}).
We leave the details to the reader.
\end{proof}

\section{Applications} \label{sec:applications}
In this section we prove two applications of the holomorphic anomaly equation for the Hilbert scheme stated in the introduction.
The first considers the $2$-point function on the Hilbert scheme 
(Corollary~\ref{cor:transformation law of 2 pt operator})
which is implied by Proposition~\ref{prop:HAE for Z} below.
The second concerns the Jacobi form property for CHL Calabi-Yau threefolds (Theorem~\ref{thm:CHL}).
Here we first prove by a deformation argument a version of the holomorphic anomaly equation for generating series which keep track of curve classes of the form $W+dF+\alpha$ where $\alpha$ runs over a lattice $E_8(-2) \subset \Pic(S)$ orthogonal to $W,F$ (Proposition~\ref{prop:extended HAE}).
Then Theorem~\ref{thm:CHL} follows formally by the degeneration formula and the GW/Hilb correspondence (Theorem~\ref{thm:GW/Hilb}).

\subsection{The $2$-point function}
\label{subsec:2 pt function}
Recall the notation of Section~\ref{sec:Weight grading}, in particular the LLV algebra
\[ \Fg(S^{[n]}) = \wedge^2( V \oplus U_{\BR} ), \quad V = H^2(S^{[n]}). \]
Extend the definition of the operator $T_{\alpha}$ by defining 
\[ T_{\alpha} := \act(\alpha \wedge F). \]
for all $\alpha \in V \oplus U_{\BR}$ with $\alpha \perp \{ W, F \}$.
In particular, we have
\begin{equation} T_{e} = \act(e \wedge F) = e_{F}, \quad T_{f} = \act(f \wedge F) = -U. \label{esfdg09s} \end{equation} 

For any operator $a \in \Fg(S^{[n]})$ which is homogeneous of degree $\deg(a)$
(i.e. if $\deg(a\gamma) = \deg(\gamma)+\deg(a)$ for all homogeneous $\gamma$),
define the induced operator
\begin{equation} \label{esfdg09s2}
\begin{gathered}
 a : H^{\ast}(S^{[n]})^{\otimes N} \to H^{\ast}(S^{[n]})^{\otimes N}, \\
a(\gamma_1 \otimes \ldots \otimes \gamma_{N}) 
= \sum_{i=1}^{N} \gamma_1 \otimes \cdots \gamma_{i-1} \otimes ((-1)^{i \cdot \deg(a)} a \gamma_i) \otimes \gamma_{i+1} \otimes \ldots \otimes \gamma_{N}.
\end{gathered}
\end{equation}

By the quasi-Jacobi form part of Theorem~\ref{thm:Main result}
the generating series $\Z^{S^{[n]}}(p,q)$ defined in \eqref{def 2 pt function}
can be identified with a vector with entries quasi-Jacobi forms.
We prove the following anomaly equation,
which combined with Lemma~\ref{lemma:QJac transformation laws} (and using that $\Wt$ is anti-symmetric)
immediately implies Corollary~\ref{cor:transformation law of 2 pt operator}.

\begin{prop} \label{prop:HAE for Z}
\begin{gather*}
\frac{d}{d G_2} \Z^{S^{[n]}}(p,q) = - \sum_{\alpha,\beta} (\tilde{g}^{-1})_{\alpha \beta} T_{\alpha} T_{\beta} \Z^{S^{[n]}}(p,q) \\
\frac{d}{d \A} \Z^{S^{[n]}}(p,q) = - T_{\delta} \Z^{S^{[n]}}(p,q),
\end{gather*}
where $\alpha,\beta$ run over a basis of $\{ W, F \}^{\perp} \subset V \oplus U_{\BQ}$
with intersection matrix $\tilde{g}_{ab} = \langle \alpha, \beta \rangle$.
\end{prop}
\begin{proof}
By Theorem~\ref{thm:Main result} 
for any $\gamma_1, \gamma_2 \in H^{\ast}(S^{[n]})$ we have
\begin{multline*}
\frac{d}{d G_2} F_{0,1}^{S^{[n]}}(\gamma_1, \gamma_2)=
2 F_{0,1}^{S^{[n]}}( U(\gamma_1 \cup \gamma_2)) - 2 F_{0,1}^{S^{[n]}}(\psi_1 ; U \gamma_1, \gamma_2) - 2 F_{0,1}(\psi_2 ; \gamma_1, U \gamma_2) \\
- \sum_{a,b} (g^{-1})_{ab} F_{0,1}^{S^{[n]}}( T_{e_a} T_{e_b} ( \gamma_1 \otimes \gamma_2 )).
\end{multline*}
Let $\pt_{S^{[n]}} = \Fq_1(\pt)^n \1$ be the class of a point on $S^{[n]}$.
By $U(\pt_{S^{[n]}}) = n \Fq_1(F) \Fq_1(\pt)^{n-1} \1$ and the evaluation \cite[Thm.2]{HilbK3} we have
\[
2 F_0^{S^{[n]}}( U(\gamma_1 \cup \gamma_2)) = 2 F_0^{S^{[n]}}(U(\pt)) \int_{S^{[n]}} \gamma_1 \cup \gamma_2
= 2n \frac{\mathbf{G}(p,q)^{n-1}}{\Delta(q)} \int_{S^{[n]}} \gamma_1 \cup \gamma_2.
\]
Similarly, using the divisor equation with respect to $\frac{1}{(n-1)!} \Fq_1(F) \Fq_1(1)^{n-1} \1$ to add a marking,
rewriting the $\psi$-class in terms of boundary and applying the splitting axiom of Gromov-Witten theory
(see for example Section~1.2 of \cite{COT} for a similar case) yields
\[
F_{0,1}^{S^{[n]}}(\psi_1 ; U \gamma_1, \gamma_2) = 
F_{0,1}^{S^{[n]}}(U \gamma_1, e_F \gamma_2) - F_{0,1}^{S^{[n]}}( e_F U \gamma_1, \gamma_2).
\]
Rewriting this using \eqref{esfdg09s} and using convention \eqref{esfdg09s2} we get:
\begin{multline*}
- 2 F_{0,1}^{S^{[n]}}(\psi_1 ; U \gamma_1, \gamma_2) - 2 F_{0,1}^{S^{[n]}}(\psi_2 ; \gamma_1, U \gamma_2) 
=
2 F_{0,1}^{S^{[n]}}( U e_F ( \gamma_1 \otimes \gamma_2 )) =
-2 F_{0,1}^{S^{[n]}}( T_e T_f (\gamma_1 \otimes \gamma_2 )).
\end{multline*}
Finally, by the commutation relations \eqref{eq:comm relations 1} we have
\[ \frac{d}{d G_2} \mathbf{G}(p,q) = 2 \Theta(p,q)^2. \]
Putting all this together we obtain:
\begin{align*}
\frac{d}{d G_2} \int_{S^{[n]} \times S^{[n]}} \Z^{S^{[n]}}(p,q) \cup (\gamma_1 \otimes \gamma_2)
& =
\frac{d}{d G_2} F_{0,1}^{S^{[n]}}(\gamma_1 \otimes \gamma_2) - \left( \int_{S^{[n]}} \gamma_1 \cup \gamma_2 \right) \frac{d}{d G_2} \frac{\mathbf{G}^n}{\Theta^2 \Delta(q)} \\
& =
- \sum_{\alpha,\beta} (G^{-1})_{\alpha \beta} T_{\alpha} T_{\beta} \Z^{S^{[n]}}(p,q)
\end{align*}
The first claim now follows since $T_{\alpha}$ is anti-symmetric if $\alpha \in V$,
and symmetric if $\alpha \in U_{\BQ}$ (both orthogonal to $W,F$).
The second claim follows from
$\frac{d}{d \A} \mathbf{G} = 0$,
the holomorphic anomaly equation for $\frac{d}{d\A}$ (proven in Theorem~\ref{thm:Main result}), and since $T_{\delta}$ is anti-symmetric.
\end{proof}

\subsection{CHL Calabi-Yau threefolds}
We work in the setting introduced in Section~\ref{subsec:intro CHL}.
For a general element $\alpha \in E_8(-2)$ (where $E_8(-2) \subset \Pic(S)$ is the anti-invariant part of the symplectic involution $g : S \to S$)
and with $W=B+F$ as usual, consider the curve class
\[ W + d F + \alpha \in H_2(S,\BZ). \]

Let $b_1, \ldots, b_8$ be a fixed integral basis of $E_8(-2)$,
and identify $w=(w_1,\ldots, w_8) \in \BC^8$ with $\sum_i w_i b_i \in E_8(-2) \otimes \BC$.
Given a class $\alpha \in E_8(-2)$, we write
\begin{equation}\label{eq:zetabeta}
\zeta^{\alpha} = \exp( \langle w, \alpha\rangle ) = \prod_{i=1}^{8} e( \langle b_i, \alpha \rangle w_i ).
\end{equation}
We also refer to \cite[Section 2.1.4]{OPix2} for parallel definitions.

Form the extended generating series
\begin{equation*}
\widetilde{F}^{S^{[n]}}_{g}(\taut; \gamma_1, \ldots, \gamma_N) = 
\sum_{d = -1}^{\infty} \sum_{r \in \BZ} \sum_{\alpha \in E_8(-2)} \left\langle \taut ; \gamma_1, \ldots, \gamma_N \right\rangle^{S^{[n]}}_{g, W+dF+\alpha+rA} q^d (-p)^r \zeta^{\alpha}.
\end{equation*}
Usually we drop the supscript $S^{[n]}$. The first step is to prove the following:
\begin{prop} \label{prop:extended HAE}
If Conjecture~\ref{Conj: Quasi Jacobi property} and Conjecture~\ref{Conj:HAE} holds for $(g,N)$, then
\[
\widetilde{F}^{S^{[n]}}_{g}(\taut; \gamma_1, \ldots, \gamma_N) \in  \frac{1}{\Delta(q)} \QJac_{k+12,(n-1) \oplus \frac{1}{2} E_8(-2)}(\Gamma_0(2) \ltimes (2 \BZ \oplus \BZ)),
\]
where $k=n(2g-2+N)+\sum_i \wt(\gamma_i) - 6$ and
$\QJac_{k,L}$ is the vector space of weight $k$ multi-variable quasi-Jacobi forms of lattice index $L$ 
as defined in \cite[Sec.1]{OPix2},
except that here we work with respect to the Jacobi group $\Gamma_0(2) \ltimes (2 \BZ \oplus \BZ)$.\footnote{More explicitly, the quasi-Jacobi forms we consider will simply be linear combinations of derivatives of the theta function of the $E_8(2)$-lattice, see the proof.}
Moreover,
\begin{equation} \label{HAEextended}
\begin{aligned}
\frac{d}{d G_2}
\widetilde{F}_{g}^{S^{[n]}}(\taut; \gamma_1, \ldots, \gamma_N) 
&= 
\widetilde{F}_{g-1}^{S^{[n]}}(\taut'; \gamma_1, \ldots, \gamma_N, U) \\
& + 2 \sum_{\substack{g=g_1 + g_2 \\ \{ 1, \ldots, N \} = A \sqcup B }} \widetilde{F}_{g_1}^{S^{[n]}}(\taut_1 ; \gamma_A, U_1) F_{g_2}^{S^{[n]}, \std}(\taut_2 ; \gamma_B , U_2 ) \\
& - 2 \sum_{i=1}^{N} \widetilde{F}_{g}^{S^{[n]}}( \psi_i \cdot \taut ; \gamma_1, \ldots, \gamma_{i-1}, U \gamma_i, \gamma_{i+1}, \ldots, \gamma_N) \\
& - \sum_{a,b} (\widehat{g}^{-1})_{ab} T_{e_a} T_{e_b} F_{g}^{S^{[n]}}(\taut; \gamma_1, \ldots, \gamma_N),
\end{aligned}
\end{equation} 
where the $e_a$ form a basis of $( \mathrm{Span}_{\BZ}(B,F) \oplus E_8(-2) )^{\perp} \subset H^2(S,\BQ)$
with intersection matrix $\widehat{g}_{ab} = \langle e_a, e_b \rangle$, and
\[
\frac{d}{d \A} \widetilde{F}_{g}^{S^{[n]}}(\taut; \gamma_1, \ldots, \gamma_N)  = T_{\delta} \widetilde{F}_{g}^{S^{[n]}}(\taut; \gamma_1, \ldots, \gamma_N).
\]
\end{prop}
\begin{proof}
For $\alpha \in E_8(-2)$ the operator $T_{\alpha} = \act(\alpha \wedge F)$ 
satisfies
\[ e^{-T_{\alpha}}( W + dF+rA + \alpha ) = W + \left(d + \frac{1}{2}\langle \alpha, \alpha \rangle \right)F+ r A. \]
Moreover, $e^{-T_{\alpha}}$ can either be viewed as a monodromy operator (as in Section \ref{subsec:monodromy})
or identified with the induced action on the Hilbert schemes coming from the automorphism $t_{-\alpha} : S \to S$ given by translation by the section labeled by $-\alpha$,
compare \cite[Sec.3.4]{OPix2}.
In either case, we have invariance of Gromov-Witten invariants, so we find:
\begin{align*}
& \widetilde{F}_{g}^{S^{[n]}}(\taut; \gamma_1, \ldots, \gamma_N) \\
& = 
\sum_{d,r} \sum_{\alpha \in E_8(-2)} 
\left\langle \taut ; e^{-T_{\alpha}} \gamma_1, \ldots, e^{-T_{\alpha}} \gamma_N \right\rangle^{S^{[n]}}_{g, W+\left(d + \frac{1}{2}\langle \alpha, \alpha \rangle \right)F+ r A} q^d (-p)^r \zeta^{\alpha} \\
& = \sum_{\tilde{d},r} 
\left\langle \taut ; e^{-T_{\alpha}} \gamma_1, \ldots, e^{-T_{\alpha}} \gamma_N \right\rangle^{S^{[n]}}_{g, W+ \tilde{d} F + r A}
 q^{\tilde{d}}  q^{- \frac{1}{2} \langle \alpha, \alpha \rangle} (-p)^r \zeta^{\alpha} \\
& = \sum_{\alpha \in E_8(-2)} F_{g}(\taut ; e^{-T_{\alpha}} \gamma_1, \ldots, e^{-T_{\alpha}} \gamma_N ) q^{- \frac{1}{2} \langle \alpha, \alpha \rangle} \zeta^{\alpha}.
\end{align*}
Let $h^{ij}$ be the inverse matrix of the intersection matrix $\langle b_i, b_j \rangle$. Then
\[ T_{\alpha} = \sum_{i,j} h^{ij} \langle \alpha, b_i \rangle T_{b_j}. \]
Moreover, let
\begin{equation} \Theta_{E_8(2)}(\zeta,q) = \sum_{\alpha \in E_8(-2)} q^{- \frac{1}{2} \langle \alpha, \alpha \rangle} \zeta^{\alpha} \label{esgsdf} \end{equation}
be the theta functions of the $E_8(2)$ lattice,
 which is a Jacobi form of weight $\frac{1}{2} \mathrm{rk} E_8(-2) = 4$ 
and lattice index $\frac{1}{2} E_8(2)$ for the Jacobi group $\Gamma_0(2) \ltimes (2 \BZ \times \BZ)$, see 
\cite[Sec.3]{Ziegler}.\footnote{Concretely, the theta function $\Theta_{E_8}(\tau,z)$ for the unimodular lattice $E_8$ is a Jacobi form
for the full Jacobi group $\SL_2(\BZ) \ltimes \BZ^2$, and we replace $\tau$ by $2 \tau$, which introduces the congruence subgroup.}
Similarly, if we multiply the summand in \eqref{esgsdf} with products of $\langle \alpha, b_i \rangle$,
the function becomes derivatives of the theta functions by the differential operators 
\[ D_{b_i} = \frac{1}{2 \pi i} \frac{d}{d w_i}. \]
For example,
\[
\sum_{\alpha \in E_8(-2)} \langle \alpha, b_i \rangle q^{- \frac{1}{2} \langle \alpha, \alpha \rangle} \zeta^{\alpha} = D_{b_i} \Theta_{E_8(2)}(\zeta,q).
\]
Putting this together we find that,
\begin{equation} \label{FGefsd}
\widetilde{F}_{g}^{S^{[n]}}(\taut; \gamma_1, \ldots, \gamma_N)
=
F_{g}\left( \taut ; e^{ -\sum_{i,j} h^{ij} D_{b_i} T_{b_j} } \gamma_1,\, \ldots, \, e^{-\sum_{i,j} h^{ij} D_{b_i} T_{b_j}} \gamma_N \right) \Theta_{E_8(2)}(\zeta,q)
\end{equation}
which is understood as expanding all the exponentials and then applying the derivatives $D_{b_i}$ to the theta function.
The operator $D_{b_i}$ preserves the algebra of quasi-Jacobi forms, see \cite{OPix2}.
Moreover, since $D_{b_i}$ increases the weight by $1$, and $T_{b_i}$ is of degree $-1$ with respect to the weight grading $\wt$ on cohomology,
we conclude that \eqref{FGefsd} is a quasi-Jacobi form of weight 
equal to the weight of $F_{g}(\taut; \gamma_1, \ldots, \gamma_N)$ plus $4$.
Finally, the claimed holomorphic anomaly equations also follow from \eqref{FGefsd} by a straightforward computation:
The terms where $\frac{d}{d G_2}$ does not interact with the derivatives $D_{b_i}$ are evaluated by Conjecture~\ref{Conj:HAE}.
For any $\alpha \in E_8(-2)$ one has $(e^{-T_{\alpha}} \otimes e^{-T_{\alpha}})(U) = U$ (proven by differentiating with respect to $\alpha$ and then as in Lemma~\ref{lemma: weight of U}).
Hence one sees that these terms give precisely the $4$ terms in \eqref{HAEextended}
up to the extra term coming from summing over the basis of $E_8(-2)$ in the last term.
This extra term cancels with the terms coming from interactions of $\frac{d}{d G_2}$ with the $D_{b_i}$. These are calculated using the commutation relations \cite[Eqn.(12)]{OPix2}.
Since the $E_8$-theta function does not depend on $p$, the $\frac{d}{d \A}$ derivative follows directly
from the one in Conjecture~\ref{Conj:HAE}.
\end{proof}

\begin{proof}[Proof of Theorem~\ref{thm:CHL}]
By the arguments of \cite{ObReduced} we can work with stable pairs invariants of $X$.
We then use the degeneration formula for the degeneration
\[ (S \times E)/\BZ_2 \rightsquigarrow (S \times \p^1)/ ( (s,0) \sim (gs, \infty)), \]
which was worked out explicitly in \cite[Sec.1.6]{BO-CHL}.
This reduces us to invariants of $(S \times \p^1, S_{0,\infty})$
with relative condition specified with the graph
of the automorphism of $S^{[n]}$ induced by the involution $g : S \to S$,
\[ \Gamma_{g} \in H^{\ast}(S^{[n]} \times S^{[n]}). \]
We then apply Nesterov's wall-crossing \cite{N1,N2,QuasiK3}.
Putting all together yields:
\begin{align} 
\DT_n(p,q) 
& = \frac{1}{2} \tilde{F}_0(\Gamma_g)|_{\zeta^{\alpha}=1} - \frac{1}{2} \sum_{\alpha, d, r} q^d p^r \mathrm{Coeff}_{q^{d + \frac{1}{2} \langle \alpha, \alpha \rangle} p^r}
\left( \frac{\mathbf{G}(p,q)^{n}}{\Theta(p,q)^2 \Delta(q)} \right) \int_{S^{[n]} \times S^{[n]}} \Delta_{S^{[n]}} \cdot \Gamma_g \notag \\
& = \frac{1}{2} \tilde{F}_0(\Gamma_g)|_{\zeta^{\alpha}=1} - \frac{1}{2} \frac{\mathbf{G}(p,q)^n}{\Theta^2(p,q) \Delta(q)} \Theta_{E_8(2)}(q) \mathrm{Tr}(g |H^{\ast}(S^{[n]}) ), \label{e0sdif0d}
\end{align}
where $\Theta_{E_8(2)}(q) = \sum_{\alpha \in E_8(-2)} q^{-\frac{1}{2} \langle \alpha, \alpha \rangle} = E_4(q^2)$ is the theta function of the $E_8$-lattice.
This shows that $\DT_n(p,q)$ is a quasi-Jacobi form of weight $-6$ and index $n-1$ for $\Gamma_0(2)$.

It remains to compute the derivative with respect to $G_2$ and $\A$ of the first term (the second is clearly Jacobi).
Since the anomaly operators $\frac{d}{d G_2}$ and $\frac{d}{d \A}$ commute with specializing of the variable $\zeta$
(compare \cite[Sec.1.3.5]{OPix2}) we have
\[
\frac{d}{d G_2}\left(  \tilde{F}_0(\Gamma_g)\Big|_{\zeta^{\alpha}=1} \right)
= \left( \frac{d}{dG_2} \tilde{F}_{0}(\Gamma_g) \right)\Big|_{\zeta^{\alpha}=1}.
\]
By Proposition~\ref{prop:extended HAE}, arguing then as in the proof of Proposition~\ref{prop:HAE for Z},
and using $[T_{e_a},g]=0$ for $e_a \in E_8(-2)^{\perp}$
and $[U,g]=0$ one finds
\[
\frac{d}{d G_2} \tilde{F}_0(\Gamma_g)|_{\zeta^{\alpha}=1}
=
2n \frac{\mathbf{G}(p,q)^{n-1}}{\Delta(q)} \Theta_{E_8(2)}(q) \mathrm{Tr}(g |H^{\ast}(S^{[n]}) ).
\]
Since this cancels precisely with the $G_2$-derivative of the second term in \eqref{e0sdif0d} we get
\[ \frac{d}{d G_2} \DT_n(p,q) = 0. \]
The claim $\frac{d}{d \A} \DT_n(p,q)=0$ follows from $T_{\delta}(\Gamma_g) = [T_{\delta},g]=0$.
\end{proof}

KTH Royal Institute of Technology, Department of Mathematics 

georgo@kth.se\\

\end{document}